\documentclass[11pt, letterpaper, oneside, reqno]{amsart}

\usepackage{latexsym, amsmath, amssymb, amsfonts, amscd, amsthm,bm}
\usepackage{amsthm}
\usepackage{t1enc}
\usepackage[mathscr]{eucal} 
\usepackage{indentfirst}
\usepackage{graphicx, pb-diagram}
\usepackage{fancyhdr}
\usepackage{fancybox}
\usepackage[shortlabels]{enumitem}
\usepackage{xcolor}
\usepackage[all]{xy}
\usepackage{xspace}
\usepackage{tikz}
\usepackage{tikz-cd}
\usetikzlibrary{matrix}
\usepackage{mathrsfs}
\usepackage{relsize}
\usepackage{url}
\usepackage{hyperref}

\theoremstyle{plain}
\newtheorem{thm}{Theorem}[section]
\newtheorem*{thm*}{Theorem} 

\newtheorem{prop}[thm]{Proposition}

\newtheorem{lemma}[thm]{Lemma}
\newtheorem{cor}[thm]{Corollary}

\theoremstyle{definition}
\newtheorem{defi}[thm]{Definition}

\theoremstyle{remark}
\newtheorem{remark}[thm]{Remark}
\newtheorem{ex}[thm]{Example}

\usepackage{mdframed}
\usepackage{lipsum}

\newenvironment{fthm}
  {\begin{mdframed}\begin{thm}}
  {\end{thm}\end{mdframed}}
\newcommand{\ZZ}{\ensuremath{\mathbb Z}}

\newcommand{\RR}{\ensuremath{\mathbb R}}
\newcommand{\g}{\ensuremath{\mathfrak{g}}}

\newcommand{\h}{\ensuremath{\mathfrak{h}}}
\newcommand{\vX}{\ensuremath{\mathfrak{X}}}

\newcommand{\li}{\ensuremath{L_{\infty}}}
\newcommand{\ad}{\ensuremath{\mathrm{ad}}}
\newcommand{\der}{\ensuremath{\mathrm{der}}}
\newcommand{\id}{\ensuremath{\mathrm{id}}}
\newcommand{\Span}{\ensuremath{\mathrm{Span}}}

\newcommand{\otimesM}{\otimes_{C^\infty(M)}} %{\otimes_{C^\infty(M)}}
\newcommand{\cN}{\mathcal{N}}

\newcommand{\cM}{\mathcal{M}}

\newcommand{\cA}{\mathcal{A}}
\newcommand{\cO}{\mathcal{O}}

\newcommand{\cE}{\mathcal{E}}
\newcommand{\cU}{\mathcal{U}}

\newcommand{\cV}{\mathcal{V}}

\newcommand{\Ver}{\mathrm{Ver}}

\newcommand{\Lie}{\mathcal{L}}
\renewcommand{\d}{\mathrm{d}}

\newcommand{\pd}[1]{{\partial_{#1}}} %\pd{x}

\newcommand{\End}{\mathrm{End}}

\newcommand{\CDO}{\mathrm{CDO}}
\newcommand{\WDGLA}{\mathscr{D}}
 
 \definecolor{darkgreen}{rgb}{0.0, 0.17, 0.1}

\definecolor{forest}{rgb}{0,0.5,0}

\newcommand{\la}{\langle}
\newcommand{\ra}{\rangle}

\newcommand{\R}{\mathbb{R}}

\DeclareMathOperator{\Hom}{Hom}
\newcommand{\linfty}{\ensuremath{L_{\infty}}}
\definecolor{mygray}{rgb}{0.90,0.90,0.95}
\begin{document}

\title{$L_{\infty}$-actions of Lie algebroids}
  
\author{Olivier Brahic}
\email{olivier@ufpr.br}  
\address{Departamento de Matem\'atica - UFPR
Centro Polit\'ecnico - Jardim das Am\'ericas
CP 19081
CEP:81531-980   Curitiba - Paran\'a, Brazil}

\author{Marco Zambon}
\email{marco.zambon@kuleuven.be}
\address{KU Leuven, Department of Mathematics, Celestijnenlaan 200B box 2400, BE-3001 Leuven, Belgium.}

%\date{\today}
\subjclass[2010]{}
%{53D05, 16E45}

%\thanks{
%}

\begin{abstract}
{
We consider homotopy actions of a Lie algebroid on a graded manifold, defined as suitable $L_{\infty}$-algebra morphisms. On the ``semi-direct product'' we construct a homological vector field   that projects to the Lie algebroid. Our main theorem states that this construction is a bijection. 

{Since several  classical geometric structures can be described by homological vector fields {as above}, we can display many explicit examples, involving Lie algebroids (including extensions, representations up to homotopy and their  cocycles) as well as transitive Courant algebroids}.
}
 
\end{abstract}
\maketitle

\setcounter{tocdepth}{1} %doesn't display subsections in TOC 
\tableofcontents
%To add parts: see below
%To add space: put in the previous section \addtocontents{toc}{\protect\mbox{}\protect}

%%%%%%%%%%%%%%%%%%%%%%%%%%%%%%%%%%%%
%%%%%%%%%%%%%%%%%%%%%%%%%%%%%%%%%%%%
\section*{Introduction}%%%%%%%%%%%%% 
%%%%%%%%%%%%%%%%%%%%%%%%%%%%%%%%%%%%
%%%%%%%%%%%%%%%%%%%%%%%%%%%%%%%%%%%%

{
Many structures in differential geometry and mathematical physics can be described using the language of graded geometry. Often they are described by a graded manifold endowed with a homological vector field, and come together with a map preserving the latter. For instance Lie algebroids (and their extensions, representations up to homotopy, cocycles,...) can be described in this way, and Courant algebroids too.}{}

{
Our main construction encodes such data as an ``infinitesimal action up to homotopy''. The latter is given by a relatively simple kind of $L_{\infty}$-algebra morphism, 
whose components are typically described explicitly by classical differential geometric objects (connections, differential forms...). 
In some of our examples we recover known compatibility equations relating the various components, but now these equations acquire a new interpretation: they are the structure equations of an $L_{\infty}$-algebra morphisms.}
Below we list such examples, {which are phrased in classical terms without reference to graded geometry:}
\begin{enumerate}
	\item There is a canonical representation of $TM$ on the trivial bundle $\RR\times M$, due to the fact that vector fields on $M$ are the same as derivations of the algebra of functions $C^\infty(M)$. Now this action can be 'twisted' by any closed $n$-form so as to obtain an $L_\infty$-action of $TM$ on $\RR\times M$. In other words, the closedness of a differential form can be seen as the structure equation of an $L_\infty$-action. This construction further generalizes to Lie algebroid cocycles with values in a representation, as well as representations up to homotopy (Prop. \ref{prop:inftyrep} and Cor. \ref{cor:inftyrepmor}) together with their cocycles (see Thm. \ref{cor:finftyfn} {and Cor. \ref{cor:inftyrepmorcocycle}}, and in the context of exact Courant algebroids, Prop. \ref{prop:exactexpl}). 
	
\item A transitive Lie algebroid $A\to M$ together with a splitting $\sigma_A:TM\to A$ of its anchor $\rho:A\to TM$ is known to induce a $TM$-connection on the bundle of Lie algebras $\g_M:=\ker \rho$, together with a $\g_M$-valued $2$-form on $M$.  There are well known compatibility conditions between $\nabla$ and $\omega$,  that encode the Jacobi identity on $A$. Our construction displays these equations as the structure equations of an $L_\infty$-action of $TM$ on the bundle of Lie algebras $\g_M$ (Prop. \ref{prop:liclasstransLA}).

\item There is a similar construction for any transitive Courant algebroid $E\to M$, obtained by choosing an isotropic splitting $\sigma_E$ of its anchor. In fact, $\sigma_E$ descends to a splitting of the underlying Lie algebroid $A:=E/(\ker \rho)^\perp$, so we recover a connection $\nabla$ and a $\g_M$-valued $2$-form as above. There are however extra data appearing:  an invariant pairing $\la\ ,\ \ra_{\g_M}$ on $\g_M$ and a $3$-form ${H}$ on $M$, together with compatibility conditions {relating} all of $\nabla$, $\omega$, $\la\ ,\ \ra_{\g_M}$ and ${H}$. Again, all these equations fit into those of an $L_\infty$-action, of $TM$ on the bundle of quadratic Lie algebras $(\g_M,\la\ ,\ \ra_{\g_M})$ (see  Prop. \ref{prop:tr-CASE-Loo-map}).
\end{enumerate}

\subsection*{The main construction}
The following construction is well-known. Given  a Lie algebra  $\g$  and a manifold $M$, an infinitesimal action of $\g$ on $M$ (\emph{i.e.} a Lie algebra map $\phi\colon \g\to \vX(M)$) can be encoded by its transformation Lie algebroid, namely     $\g\times M\to M$, with bracket of constant sections given by the Lie bracket of $\g$, and with anchor given by $\phi$.
In turn, Lie algebroids are in bijective correspondence with  degree $1$ manifolds endowed with a homological vector field.
This way we obtain a bijection between:
\begin{itemize}
\item infinitesimal actions of $\g$ on $M$
\item Lie algebroid structures on the vector bundle  $\g\times M\to M$ such that the projection to $\g$ is a Lie algebroid morphism,
\item homological vector fields on the N-manifold
 $$\g[1]\times M$$
  for which the projection to $\g[1]$ preserves homological vector fields.
\end{itemize}

 %%%%
 
%\begin{defi}%%%%%%%%%%%%%%%%%%%%%%%%%%%%%%%%%%%%
Recall that a {\em Lie algebroid} $A$ over a manifold $M$ is a vector
bundle over $M$, such that the global sections of $A$ form a Lie
algebra with Lie bracket $[\cdot,\cdot]_A$ and the  Leibniz rule holds:
\[ [a, fb]_A=f[a, b]_A + \rho(a)(f) b , \quad a, b \in \Gamma(A), f
\in C^\infty(M), \]where $\rho: A \to TM$  is
  a vector
bundle morphism called the   {anchor}. 
An \emph{$N$-manifold} should be thought of as a ``smooth space'' whose coordinates are assigned non-negative degrees (see  Appendix \ref{app:Q} for details).
%\end{defi}%%%%%%%%%%%%%%%%%%%%%%%%%%%%%%%%%%%%

The aim of this paper is to present a result analogous  to the bijection between the first and last item above, in a general setting. The main case of interest is obtained
considering a Lie algebroid\footnote{More generally, we prove our results for Lie $n$-algebroids (see  Appendix \ref{app:Q}), recovering Lie algebroids for $n=1$.}
 $A\to M$ and
 an $N$-manifold $\cM$,
whose  ``degree zero part'' (the  underlying differentiable manifold) is $M$.
In this case, our main result Thm. \ref{charvf} reads:

\begin{thm*}%%%%%%%%%%%%%%%%%%%%%%%%%%%%%%%%%%%% 
 There is a one-to-one correspondence between 
     \begin{enumerate}
\item $\linfty$-actions of $A$ on $\cM$
\item  homological vector fields on 
$$A[1]\times_M \cM$$
 for which the projection  to $A[1]$  preserves homological vector fields. 
\end{enumerate} 
\end{thm*}%%%%%%%%%%%%%%%%%%%%%%%%%%%%%%%%%%%%
We provide an explicit construction of the above bijection.
The $\linfty$-actions appearing above are a kind of ``infinitesimal actions up to homotopy'', given by a suitable curved $\linfty$-algebra morphism  {$F$}{} from the Lie algebra $\Gamma(A)$ to the 
 vector fields $\vX(\cM)$. {Conceptually, one can think of passing from $(1)$ to $(2)$ in the above theorem as forming the semi-direct product associated with the action.}{} Notice that $\cM$ itself  ({the N-manifold acted upon}) inherits a ``vertical'' homological vector field, which can be described either as  the restriction of the one on $A[1]\times_M \cM$ or as the zero-th component of $F$.

\smallskip
Our main use of the above theorem is as follows:
there are many natural structures in differential geometry that are encoded as in 
 $(2)$ above. We then apply   the direction  ``$(2)\rightarrow (1)$'' of the Theorem, decomposing the given structure into components that fit together into an $L_{\infty}$-algebra morphism. In specific examples, this 
reveals that known compatibility equations satisfied by the components of the structure 
can be interpreted in terms of an $L_{\infty}$-algebra morphism, 
giving a conceptual explanation for {the}{} former.

\subsection*{Classes of examples}
{
The second part of the paper is devoted to examples of the main theorem, for which we fix a Lie algebroid $A\to M$. We also fix a N-manifold whose body is $M$, which by a version of Batchelor's theorem (see for instance \cite[Thm. 1]{BonavolontaPoncin})  we may assume to be $V[1]$ for some graded vector bundle $V\to M$ concentrated in degrees $\le 0$. 
 $$\SelectTips{cm}{}
\xymatrix{
& V[1] \ar[d]  \\
A[1] \ar[r]& M}
$$

It is easy (typically, by making the choice of a ``splitting'' or a connection) to encode several classes of classical geometric objects in graded geometric terms, as a surjective morphism of N-manifolds with homological vector fields
\begin{equation}\label{eq:introQmor}
(A\oplus V)[1]\to A[1].
\end{equation}  
By our main theorem, it corresponds to an $L_{\infty}$-actions of $A$ on $V[1]$, \emph{i.e.}
 a curved $L_{\infty}$-morphism
\begin{equation}\label{eq:introli}
F\colon \Gamma(A)  \rightsquigarrow \vX(V[1])
\end{equation}
satisfying certain compatibility conditions. 
} We spell out a few special cases of this construction  and the corresponding geometric structures we recover:
\begin{itemize}
\item[a)] {The graded vector bundle}{} \emph{$V$ is an ordinary vector bundle (\emph{i.e.} concentrated in degree zero)}.\\
  Then a morphism \eqref{eq:introQmor} is the same thing as a surjective morphism of Lie algebroids $A\oplus V\to A$, and therefore an \emph{extension of Lie algebroids}. 
By the main theorem, it is encoded by an $L_{\infty}$-action $F$. We remark that the structure of bundle of Lie algebras on the kernel $V$ is determined by the zero-th component of $F$. 
\item[b)]  \emph{The homological vector field on $(A\oplus V)[1]$ is fiberwise linear.}\\
If the homological vector field on $(A\oplus V)[1]$ preserves the functions that are linear along on the fibers of the morphism \eqref{eq:introQmor}, then it corresponds to a \emph{representation up to homotopy} { \cite{rep-hom}\cite{MackenzieVB.RajAlfonsoVBas}} of  $A$ on $V$. By the main theorem, it is encoded by an $L_{\infty}$-action $F$  as in formula \eqref{eq:introli} whose components  are (up to sign) the components of the representation up to homotopy. In particular the structure of  cochain complex on $V$ is given by the zero-th component of $F$. 
Further, for any $n\ge 2$,
 representations up to homotopy together with an $n$-cocyle are encoded by 
$L_{\infty}$-actions on $V[n-1]$ by vector fields which are sums of fiberwise linear ones and vertical  constant ones.
\item[c)] {The graded vector bundle}{} \emph{$V[1]$ is trivial as a bundle of  {graded vector spaces} with homological vector fields}.\\
 This means that we prescribe a homological vector field $Q$ defined on $V[1]$, and  such  that $(V[1],Q)\cong (\g[1]\times M, Q_{\g[1]})$ for some $L_{\infty}$-algebra $\g$. Here  $Q_{\g[1]}$
denotes the homological vector field on $\g[1]$ encoding the $L_{\infty}$-structure on $\g$. 
Using the main theorem one sees that \emph{Maurer-Cartan elements} in   
$\Gamma(\wedge^{\ge 1}A^*)\otimes \vX(\g[1])$, with the DGLA structure induced from $\vX(\g[1])$, correspond to  $L_{\infty}$-actions of $A$ on $\g[1]\times M$ compatible with $Q_{\g[1]}$. 
 \end{itemize}
We display these three classes of examples respectively in \S \ref{sec:algebroids} (see Prop. \ref{prop:sumsurjLA}), in \S \ref{sec:special} and
\S \ref{sec:inftyrep} (see Prop. \ref{prop:inftyrep} and Thm. \ref{cor:finftyfn}),
and in \S \ref{sec:vs} (see Prop. \ref{prop:alphaaction}).
The intersection between {a) and b)}  above are Lie algebroid representations with a 2-cocycle, see the introduction to \S \ref{sec:algebroids} and Prop. \ref{lem:n1}.
\bigskip

{
Courant algebroids play a major role in differential geometry and physics, as they are the structures underlying Dirac geometry and generalized complex geometry, and admit a characterization in terms of graded geometry. Courant algebroids  with transitive anchor provide examples for the main theorem, in various ways:
\begin{itemize}
\item[d)] A transitive Courant algebroid $E$ over $M$, together with an isotropic splitting of the anchor  $\rho:E\to TM$  and a connection on $E$,
induces an $L_{\infty}$-action of $TM$ on  
$(\g_M[1]\oplus T^*[1]M)\oplus T^*[2]M$, where $\g_M:=\ker \rho/(\ker \rho)^\perp$ is the associated bundle of quadratic Lie algebras. 
In the case of exact Courant algebroids this recovers the coadjoint representation up to homotopy together with a cocycle obtained in   {\cite{CCShengHigherExt}}.
\item[e)]  The  only instance in the paper in which 
  we use the direction ``$(2)\rightarrow (1)$'' in the main theorem is the following.
A transitive Courant algebroid $E$, together with an isotropic splitting of the anchor $\rho:E\to TM$, give rise to and is encoded by a  homological vector field on $A[1]\times \RR[2]$ {making it an $\RR[2]$-principal bundle}, where $A:= E/(\ker \rho)^\perp$ is the underlying Lie algebroid.    This {recovers the characteristic class of \cite{ChenStienonXu}} and allows us to interpret the Pontryagin class of a  {bundle of quadratic Lie algebras}  as the obstruction for a lifting problem of certain $L_\infty$-actions.
\end{itemize}
}

This is explained respectively in \S \ref{sec:CA} (see Prop. \ref{prop:CAroy})
 and in \S \ref{sec:CAhom}   (see Thm. \ref{thm:tottransCA}).

 \bigskip

\noindent{\bf Comparison with the literature}:
This paper can be considered as an {analogue} 
 of \cite{RajMarco} (see also \cite{ChuangLazarevLiActions} for some special cases), in which $A$ was required to be an $L_{\infty}$-algebra instead of a {Lie algebroid or Lie $n$-algebroid. The proof of our main theorem is also analog to the one of \cite{RajMarco}.} {One of the motivations for {our work} is that it embraces several geometric examples.}{}

{
In a forthcoming work \cite{LiactionsDIff} we plan to generalize the main Theorem (see above) allowing the ``degree zero part'' of the N-manifold $\cM$ to be different from the base of the Lie $n$-algebroid $A$ (we will just require a map from the former  to the latter). The generalized set-up   includes, in particular,
the case of infinitesimal actions of Lie algebras on ordinary manifolds presented at the beginning of this introduction, as well as  
the results of \cite{RajMarco}. 
}

{
We point out that a result of Vitagliano  \cite[Cor. 4.14]{VitReprSHLR} is more general than our main result Thm. \ref{charvf} and even of the set-up of \cite{LiactionsDIff} described above. Indeed, Vitagliano's result is phrased in algebraic terms, and when specialized to a geometric setting \cite[Ex. 4.12]{VitReprSHLR} it allows for $\ZZ$-graded manifolds (hence, not just N-manifolds) fibered over other $\ZZ$-graded manifolds (in particular, not just over ordinary manifolds).  Vitagliano's result  provides a bijection with \emph{actions of $LR_{\infty}[1]$ algebras} as defined in  \cite[Def. 4.9]{VitReprSHLR}, which are  specific kinds of \emph{left connections  with vanishing curvature of $LR_{\infty}[1]$ algebras} as defined in  \cite[Def. 4.1]{VitReprSHLR}. On the other hand, our main results provides a  bijection with     $L_{\infty}$-morphisms satisfying certain properties (namely, $L_{\infty}$-actions) and as such, have the advantage of being explicit and in many case are described in terms of classical geometric data. This explains the wealth of concrete examples that we display in the second part of this paper.
{Note   that } the {relation} between Vitagliano's \emph{actions of $LR_{\infty}[1]$ algebras} and $L_{\infty}$-morphisms is not worked out in \cite{VitReprSHLR}. 
} 

{Finally we point out that there are results in the literature that provide bijections in the same spirit as our main theorem (one example pointed out to us is Schlessinger-Stasheff in \cite{SchlessingerStasheff2012}, building on their work on the 1980's). However none of them seem to be stated in terms of $L_{\infty}$-algebra morphisms, whose explicit formulae allow us to provide a variety of geometric examples.}
\bigskip
 
\noindent{\bf Acknowledgements:}
We thank Alexei Kotov, Dmitry Roytenberg, Jim Stasheff and Luca Vitagliano for comments that allowed us to improve this paper.

Both authors acknowledge support of the Pesquisador Visitante Especial grant  number 88881.030367/2013-01 (CAPES/Brazil). 
The work of M.Z. was partially supported by grants  
MTM2011-22612 and ICMAT Severo Ochoa  SEV-2011-0087 (Spain), IAP Dygest (Belgium) and the long term structural funding -- Methusalem grant of the Flemish Government.  M.Z. thanks UFPR Curitiba for  hospitality during the preparation of this paper.

 \vspace{5mm}%%%%Appears in TOC
\part{The main construction}

{The purpose of the first part of the paper is to present our main theorem, namely Thm. \ref{charvf} in \S \ref{sec:mainthm}. In \S \ref{sec:actions} we introduce the notion of $L_{\infty}$-action which is needed to state the theorem. In order to clarify the main theorem, in \S \ref{sec:comp} we compare our construction  with other constructions found in the literature.} 
   
%%%%%%%%%%%%%%%%%%%%%%%%%%%%%%%%%%%%%%%%%%%%%%%%%%%%
%%%%%%%%%%%%%%%%%%%%%%%%%%%%%%%%%%%%%%%%%%%%%%%%%%%%
\section{$L_{\infty}$-actions}\label{sec:actions}%%%
%%%%%%%%%%%%%%%%%%%%%%%%%%%%%%%%%%%%%%%%%%%%%%%%%%%%
%%%%%%%%%%%%%%%%%%%%%%%%%%%%%%%%%%%%%%%%%%%%%%%%%%%%

{In this section we introduce  {the notion of $L_{\infty}$-action appearing in}{} one of the two sides of the bijection stated in Thm. \ref{charvf}.}

Let $A\to M$ be a Lie algebroid, with Lie bracket $[\ ,\ ]_A$ and anchor $\rho$,
and $\cM$ be an $N$-manifold with body $M$.  {We denote by $C(\cM)$ the graded algebra of functions on $\cM$, and by $\vX(M)$ the graded Lie algebra of derivations of $C(\cM)$.
We refer the reader to Appendix \ref{app:Q} for the basics on N-manifolds and homological vector fields.
}

\begin{defi}\label{def:liac}%%%%%%%%%%%%%%%%%%%%
An \emph{$L_{\infty}$-action of $A$ on $\cM$ } is a curved $L_{\infty}$-morphism
$$F\colon (\Gamma(A),[\ ,\, ]_A) \rightsquigarrow (\vX(\cM),  [\ ,\, ])$$
satisfying the following conditions:
\begin{itemize}
\item[i)] $F$ is $C^{\infty}(M)$-multilinear: $F(fw)=f F(w)$ for all $w\in  \Gamma(\wedge A)$ and $f\in C^{\infty}(M)$  \
\item[ii)] the first component  $F_1\colon\Gamma(A)\to \vX_0(\cM)$ satisfies 
 $F_1(a)|_M=\rho(a)$ for    all $a\in \Gamma(A)$
\item[iii)]   $F_0(C^{\infty}(M))=0$, where $F_0$ is viewed as an element of $\vX_1(\cM)$.
\end{itemize}
\end{defi}%%%%%%%%%%%%%%%%%%%%%%%%%%%%%%%%%%%%

Notice that $F$ is defined on a  Lie algebra 
and takes values in a graded Lie algebra.
{We refer to   Appendix \ref{app:li} for the details on $L_{\infty}$-algebras and curved $L_{\infty}$-morphisms.}

\begin{remark}

When   $\cM$ is concentrated in degree $0$ (\emph{i.e.} it equals $M$), the only $L_\infty$-action {of $A$ on $\cM$}{} is (the map induced on sections by) the anchor map of $A$. For an arbitrary $N$-manifold $\cM$, let us consider the canonical projection 
$p\colon \cM\to M$  induced by the inclusion $C^{\infty}(M)=C_0(\cM)\hookrightarrow C(\cM)$ and denote by $\vX^{vert}(\cM)$ the space of vector fields on $\cM$ that annihilate $C^{\infty}(M).$ Then one may notice the following facts:
\begin{itemize}
\item $F_0$ applied to the constant function $1$ on $M$ is an element of $\vX_1(\cM)$, which we denote by $Q_{\cM}$. The condition iii)  requires that it lies in $\vX_1^{vert}(\cM)$. Further,  the curved $L_{\infty}$-morphism condition implies that $Q_\cM$ is a homological vector field, by   Lemma \ref{lemma:curved}.
\item  $F_1$ takes values in $\vX_{0}(\cM)$, and elements of $\vX_{0}(\cM)$ are always 
{tangent to the body $M$.  Condition ii) in Def. \ref{def:liac} is a condition on the restriction to $M$.}
\item for $k\ge 2$, $F_k$ 
takes values in 
$\vX_{1-k}(\cM)=\vX^{vert}_{1-k}(\cM)$.
\end{itemize}  
\end{remark}
 
\begin{remark}
{
The condition i) in Def. \ref{def:liac} does not imply automatically that 
conditions ii) and iii)  hold, however the latter two conditions   are natural in view of i). This is meant in the same sense as a bracket-preserving bundle map $\varphi$ (covering the identity) between Lie algebroids does not {automatically preserve} the anchors in general, but it does at points where $\varphi$ is not the zero map.
} 
\end{remark}

In many situations, $\cM$ already comes equipped with a homological vector field $Q_{\cM}$, \emph{i.e.} $Q_{\cM}$ has degree one is and commutes with itself. Throughout the whole paper,
{by \emph{$Q$-manifold} we will always mean an $N$-manifold endowed with a homological vector field.}

\begin{defi}\label{def:action}
Let $A\to M$ be a Lie algebroid, and let $(\cM, Q_{\cM})$ be a $Q$-manifold with body $M$.
An {$\linfty$-action} $F$ of $A$ on $\cM$ is \emph{compatible with $Q_\cM$} if the zero-th component $F_0$ equals $Q_\cM$.
\end{defi}
 
\begin{prop}\label{prop:liacQ}
Let $A\to M$ be a Lie algebroid, and let $(\cM, Q_{\cM})$ be a $Q$-manifold with body $M$. Assume that $Q_{\cM}\in \vX_1^{vert}(\cM)$.

Then the higher\footnote{I.e., all components except for the zero-th.} components of an $L_{\infty}$-action of $A$ on $\cM$ compatible with $Q_\cM$ {define} an $L_{\infty}$-morphism 
\begin{align}\label{prop:compatible1}
F\colon\bigl(\Gamma(A),[\ ,\,]_A\bigr)\rightsquigarrow \bigl(\vX(\cM), -[Q_{\cM},\ ], [\ ,\,]\bigr)
\end{align}
such that 
\begin{itemize}
\item[i)] $F(fs)=f F(s)$ for all $s\in  \Gamma(\wedge A)$ and $f\in C^{\infty}(M)$  \
\item[ii)] the first component  $F_1\colon\Gamma(A)\to \vX_0(\cM)$ satisfies $ F_1(a)|_M=\rho(a)$ for    all $a\in \Gamma(A)$.
\end{itemize}
Conversely, given an $L_\infty$-morphism $F$ as in \eqref{prop:compatible1}  satisfying the conditions $i)$ and $ii)$, one obtains  an $L_\infty$-action compatible with $Q_\cM$: the higher components  are given by $F$  and   the zero-th component is given by $Q_\cM$.  
\end{prop}
  
Notice that  the image of $F$ lies in the sub-DGLA $$\vX_{<0}(\cM)\oplus \{X\in \vX_0(\cM): [Q_{\cM},X]=0\},$$ as a consequence of the fact that $F_1$ is a chain map.
\begin{proof}[Proof of Prop. \ref{prop:liacQ}.]
Given an $\linfty$-action  $F$ of $A$ on $\cM$,
{consider the $L_{\infty}[1]$-morphism from $\Gamma(A)[1]$ to $(\vX(\cM)[1],   \{\ ,\ \})$ corresponding to $F$. Here $\{\ ,\ \}$ is the bracket corresponding to $[\ ,\,]$ under the {d\'ecalage} isomorphism \eqref{deca} in Appendix \ref{app:li}.}  
By Lemma \ref{lemma:curved} 
the higher components of {the curved $L_{\infty}[1]$-morphism corresponding to $F$}  are a (noncurved) $\linfty[1]$-algebra morphism from $\Gamma(A[1])$ to  $(\vX(\cM)[1], \{Q_\cM, \ \}, \{\ ,\, \})$, an $L_{\infty}[1]$-algebra whose only non-trivial multibrackets are the unary and binary one.
Therefore they correspond to a (noncurved)
$\linfty$-algebra morphism from $\Gamma(A)$ to the DGLA $(\vX(\cM), -[Q_\cM, \ ], [\ ,\,])$. {Note here that}{} the minus sign comes from the d\'ecalage isomorphism. The converse implication is proven reversing the argument.
\end{proof}

\begin{remark}\label{rem:linftyMarkl}
We spell out the condition that $F$ be an $L_{\infty}$-morphism (see for instance \cite[Def. 5.2]{LadaMarkl}): for all $n\ge 1$ and $a_1, \dots, a_n \in \Gamma(A)$,
\begin{align}\label{eq:Aloo}
&
 \sum_{1\le i<j\le n} (-1)^{i+j+1}F_{n-1}\bigl([a_i,a_j],a_1,\dots,\widehat{a_i},\dots,
\widehat{a_j},\dots,a_n\bigr)=\\
 &
 d_{Q_{\cM}}\bigl(F_n({a_1},\!...,a_n)\bigr)
 +\sum_{s=1}^{n-1}\sum_{\substack{\tau\in Sh(s,n-s)\\ \tau(1)<\tau(s+1)}}
\!\!(-1)^{\tau}(-1)^{s-1}\left[F_s(a_{\tau(1)},\!...,a_{\tau(s)}),
F_{n-s}(a_{\tau(s+1)},\!...,a_{\tau(n)})\right]\nonumber
\end{align}
where we write $ d_{Q_{\cM}}=-[{Q_{\cM}},\ ]$, and where $Sh(s,n-s)$ denotes the permutations $\tau$ of $n$ elements such that
 $\tau(1)<\dots<\tau(s)$ and $\tau(s+1)<\dots<\tau(n)$.
\end{remark}

{Notice that the definition of  $\linfty$-action (Def. \ref{def:liac}) extends immediately to any 
 Lie $n$-algebroid $A$  ({see  Appendix \ref{app:Q}}) with the following obvious modifications:   $\Gamma(A)$ by definition has an $L_{\infty}$-algebra structure, so it makes sense to ask that $F$ be
  a curved $L_{\infty}$-morphism from $\Gamma(A)$ to $(\vX(\cM),  [\ ,\,])$,
  and in {the}{} condition ii) one requires {that}{} $a\in \Gamma(A_0)$. }

\section{The main theorem}\label{sec:mainthm}

In this section we derive our main results extending the arguments
of \cite[\S 3,\S 4]{RajMarco}.
The main results, Thm. \ref{charvf} and Thm. \ref{charvfQM}, are stated at {the}{} end of this section.

Let $A\to M$ be a {Lie $n$-algebroid}, with  and anchor $\rho\colon A_0\to TM$.
Let $\cM$ be an $N$-manifold with body $M$  (hence the "fiber coordinates'' of $\cM$ have degrees $>0$). We have a natural projection\footnote{Here we use the fact that $\cM$ is an N-manifold rather than an arbitrary  $\ZZ$-graded manifold.} 
 $p\colon \cM\to M$ and an embedding $M\hookrightarrow \cM$, {see  Appendix \ref{app:Q}}.

We will consider the product manifold $A[1]\times \cM$ and its submanifold
$$\cN:=A[1]\times_M \cM$$
where the fiber product is taken over the vector bundle projection $A[1]\to M$
and over $p\colon \cM\to M$. The body of $\cN$ is the diagonal $\Delta M \subset M\times M$.

\begin{remark}\label{rem:aftercN}
1) {We have $\cN=p^*(A[1])$, the pull-back of the graded vector bundle $A[1]\to M$ along the map $p$, as in the diagram below:
\begin{equation*}
\SelectTips{cm}{}
\xymatrix{
\cN\ar[d]  \ar[r] & \cM\ar[d]^p \\
A[1] \ar[r]^{\pi}& M}
\end{equation*}
Hence  $\cN$ is a vector bundle over $\cM$, with zero section equal to $\varepsilon(\cM)$, for the  
  natural embedding $\varepsilon\colon \cM\to A[1]\times \cM$.}

2) We describe in coordinates the embeddings
$$\varepsilon(\cM)\;\subset\; \cN\;\subset\; A[1]\times \cM.$$ 
On $A[1]\times \cM$ we can choose coordinates $x,\xi$ on $A[1]$ (where the $x$ have degree $0$ and are coordinates on the body $M$, while $\xi$ are fiber coordinates\footnote{When $A$ is a Lie algebroid, the coordinates $\xi$  have degree $1$.}), and coordinates $x,\eta$ on $\cM$ (where the $y$ have degree $0$ and are coordinates on the body $M$, while $\eta$ denotes the coordinates of positive degree). Then locally $$\cN=\{x=y\},\quad \quad\quad\text{           }\quad\quad\quad\varepsilon(\cM)=\{x=y, \xi=0\}.$$
\end{remark}

\subsection{Part I of the proof of the main theorem}
Recall that the  functions on $A[1]$ are given by
 $C(A[1])=\Gamma(S(A[1])^*)$.
We do not provide a proof for the following lemma, which is straight-forward and generalizes the well-known fact that differential $k$-forms on a manifold $M$ are the same thing as $C^{\infty}(M)$-multilinear maps $\Gamma(\wedge^k TM)\to C^{\infty}(M)$.

\begin{lemma}\label{lem:bijeasy}
{There is a bijection between
\begin{itemize}
\item Degree $1$ elements $$X\in \Gamma(S(A[1]^*)\otimesM\vX(\cM),$$ 
\item $C^{\infty}(M)$-multilinear maps of degree $0$
$$\Phi\colon \Gamma(S(A[1]))\to \vX(\cM)[1].$$ 
\end{itemize}
The bijection is given by $X\mapsto \phi_X$ where 
\begin{equation}\label{eq:phi}
(\phi_X)_k({s_1,\dots,s_k})  :=\bigl[[\,\cdots [X, \iota_{s_1}],\dots],\iota_{s_k}\bigr]|_{\cM} 
\end{equation}
for all $k\ge 0$ and $s_i \in \Gamma(A[1])$. 
Here $|_{\cM} $ denotes the restriction to $\cM$.
 By definition, the brackets on the right side of \eqref{eq:phi} are given by\footnote{Upon choosing a suitable extension of $X$ to a vector field on 
 $A[1]\times \cM$, this is the graded Lie bracket on $\vX(A[1]\times \cM)$, see also Rem. \ref{rem:vfonN} 2).
Because of this, the definition of $\phi_X$ we give here differs in the signs from the one given in an analogous case in \cite[\S3]{RajMarco} (there  $\vX(\cM)[1]$ was used).}
   $[{Y}, {\iota_s}] = - (-1)^{|Y||s|} \sum_j(\iota_s f_j)\otimesM v_j$ 
where $Y=\sum_j f_j\otimesM v_j$ for $f_j\in \Gamma(S(A[1]^*))$ and $v_j\in \vX(\cM)$. 
}
\end{lemma}

\begin{remark}\label{rem:vfonN}
1) 
Geometrically, $X$ is a vector field   on $A[1]\times \cM$ supported on $\cN$, \emph{i.e.} a section of $T((A[1]\times \cM)|_{\cN}$). {This means that it} corresponds to  a map: $$D\colon C(A[1]\times \cM)\to C(\cN)$$ such that
$D(F\cdot G)=D(F)\cdot G|_{\cN}^C+F|_{\cN}^C\cdot D(G)$,
where: $$|_{\cN}^C\colon C(A[1])\otimes_{\RR}C(\cM)\to C(A[1])\otimesM C(\cM)$$ is the natural quotient map. Writing $X=\sum_j f_j\otimes v_j$ as in Lemma \ref{lem:bijeasy}, $D$ is given as follows: it annihilates $C(A[1])$ and $D(\varphi)=
\sum_j f_j\otimesM (v_j(\varphi))$ {for all $\varphi\in C(\cM)$}.

2) The map $\phi_X$   in Lemma \ref{lem:bijeasy} can be alternatively described replacing the r.h.s. of eq. \eqref{eq:phi} by
$$\bigl[\,\cdots [\tilde{X}, \iota_{{s}_1}],\dots],\iota_{s_k}\bigr]|_{\varepsilon(\cM)}$$
where $\tilde{X}$  is an extension of $X$  to a vector field on   $A[1]\times \cM$ which maps to zero under the projection  $A[1]\times \cM\to A[1]$,  the square bracket is the graded Lie bracket on $\vX(A[1]\times \cM)$, and $|_{\varepsilon(\cM)}$ denotes the restriction to ${\varepsilon(\cM)}$ composed with the canonical isomorphism
$(\{0\}\oplus T\cM)|_{\varepsilon(\cM)}\cong T\cM$.
 
 Notice that $\tilde{X}$ can be obtained (locally) by taking  a lift
of $X$
w.r.t. the projection 
\begin{equation}\label{eq:iN}
C(A[1])\otimes_{\RR}\vX(\cM)\to C(A[1])\otimesM\vX(\cM).
\end{equation}
Being a vector field on $A[1]\times \cM$, $\tilde{X}$ corresponds to  a derivation $\tilde{D}$ of $C(A[1]\times \cM)$. {It is related to  the map $D$ above by} $D=|_{\cN}^C\circ \tilde{D}$ (geometrically, $|_{\cN}^C$ is the restriction to $\cN$ of functions).\end{remark}

\begin{lemma}\label{lem:bijphi}
{The  bijection {$X\mapsto \phi_X$} of Lemma \ref{lem:bijeasy} restricts to a bijection between 

\begin{itemize}
\item Degree 1 elements $$X\in \Gamma(S(A[1]^*))\otimesM \vX(\cM)$$ such that $Q_{A[1]}|_{\cN}+X$ is tangent to $\cN$
\item $C^{\infty}(M)$-multilinear maps of degree $0$ $$\Phi\colon \Gamma(S(A[1]))\to \vX(\cM)[1]$$  such that $\Phi_1(a{[1]})|_M=\rho(a)$ for all $a\in \Gamma(A_0)$, and $\Phi_0\in \vX^{vert}_1(\cM)$.
\end{itemize}
 }
\end{lemma}
\begin{remark}
{The proof becomes slightly more transparent when $A$ is a Lie algebroid. In that case, the degree $k$ functions on $A[1]$ are given by $C_k(A[1])=\Gamma(S^k(A[1])^*)$.
}
\end{remark}
\begin{proof}[Proof of Lemma \ref{lem:bijphi}] Consider 
$X\in {C(A[1])}\otimesM
\vX(\cM)$ of  degree $1$, 
and decompose it as $X=X_0+X_1+X_2+\dots$,
where $$X_k\in\Gamma\bigl(C_k(A[1])\bigr)\otimesM
\vX_{1-k}(\cM).$$
So for instance $X_0\in \vX(\cM)_{1}$, $X_1\in C_1(A[1])\otimesM
\vX(\cM)_{0}$,\
$X_2\in C_2(A[1])\otimesM
\vX(\cM)_{-1}$.
Since the coordinates of $\cM$ have degrees $\ge 0$, we have that $X_k(C^{\infty}(M))=0$ for $k\ge 2$, whereas  the components
\begin{align*}
X_0(C^{\infty}(M))&\subset C_1(\cM),\\
X_1(C^{\infty}(M))&\subset  C_1(A[1])\otimesM
C^{\infty}(M)\cong C_1(A[1])
\end{align*}
might be non-zero.

On the other hand $Q_{A[1]}$  is a degree 1 vector field on $A[1]$, so $Q_{A[1]}(C^{\infty}(M))\subset C_1(A[1]) $ (this encodes the anchor of $A$).

Let $Y$ be a vector field   on $A[1]\times \cM$ supported on $\cN$. Then $Y$ is tangent to $\cN$ iff, viewed as a derivation
$C(A[1]\times \cM)\to C(\cN)$
(see Rem. \ref{rem:vfonN} 1)), it annihilates the ideal of $C(A[1]\times \cM)$ generated by $f\otimes_{\RR}1-1\otimes_{\RR}f$ where $f\in  
C^{\infty}(M)$. {We are interested in the vector field  $Q_{A[1]}|_{\cN}+X$, for which we have}
$$(Q_{A[1]}|_{\cN}+X)(f\otimes_{\RR}1-1\otimes_{\RR}f)=\left(Q_{A[1]}(f)-X_1(f)\right)-X_0(f).$$
Separating the $C_1(A[1]) $-component from the $C_1(\cM)$-component,  we see that   $Q_{A[1]}|_{\cN}+X$ is tangent to $\cN$ if{f} 
$$X_0(f)=0 \quad\text{   and   }\quad X_1(f)=Q_{A[1]}(f)$$ for all  $f\in  
C^{\infty}(M)$. The latter equation {takes place in $C_1(A[1])={\Gamma((A_0[1])^*)}$} and is equivalent to $(\phi_X)_1({s})|_M=\rho(a)$ for all $a\in \Gamma({A_0})$, {where $s:=a[1]$}. This follows, using the Jacobi identity, from
$(\phi_X)_1(s)|_{\cM}=[X,\iota_s]|_{\cM}=[X_1,\iota_s]$
and $\rho(a)=[Q_{A[1]},\iota_s ]|_M$.  
\end{proof}

 We present an example for Lemma \ref{lem:bijphi}.  
\begin{ex}\label{ex:M}
{
Let $\cM=M$ be an ordinary manifold and, for the sake of simplicity, let $A=TM$. Then necessarily $X=X_1\in \Gamma(A[1]^*)\otimesM\vX(M)$, and under the bijection of Lemma \ref{lem:bijeasy} it corresponds to a map $\Phi=\Phi_1\colon \Gamma(A[1])\to \vX(M)[1].$ 
Choose coordinates on the manifold $M$, giving rise to canonical coordinates $x,\xi$ on $A[1]$ and coordinates $y$ on $\cM=M$. We have
$Q_{A[1]}=\xi_i{{\partial_{x_i}}{}}$ (the de Rham vector field) and $$X=\psi_i^j\xi_i\frac{\partial}{\partial{y_j}}$$ for some functions $\psi_i^j$ on $M$. Further, the map $\Phi$ is determined by its values  $\Phi(a_i{[1]})=\psi_i^j {\partial_{y_j}}{}$ for any $a_i={\partial_{x_i}}{}\in\Gamma(A)$.
}

{
We see that  $Q_{A[1]}|_{\cN}+X$ is tangent to $\cN=\{x=y\}$ if{f} ${\partial_{x_i}}{}+\psi_i^j{\partial_{y_j}}{}$ is tangent to $\cN$ for every $i$, \emph{i.e.} if
$\psi_i^j=\delta_i^j$. This means exactly that $\Phi_1(\cdot)|_M=\id_{TM}$,  the anchor of the tangent Lie algebroid.
}
\end{ex}

{
\begin{remark}
Lemma \ref{lem:bijphi} no longer holds if we allow $\cM$ to be an arbitrary $\ZZ$-graded manifold instead of an N-manifold. This is evident looking at the following variation of Ex. \ref{ex:M}: as there, take $A=TM$, but now $\cM:=\RR[1]\times \RR[-1]\times M$. On $\cM$ we take coordinates $\eta$ of degree $1$ and $\theta$ of degree $-1$, in addition to the coordinates $y$ on $M$. Consider
$$X= \xi_i\left(\frac{\partial}{\partial{y_i}}+\eta\theta\frac{\partial}{\partial{y_i}}\right),$$ an element of
$\Gamma(A[1]^*)\otimes_{C^{\infty}(M)}\vX_0(\cM)$.
Then  $Q_{A[1]}|_{\cN}+X$ is \emph{not} tangent to $\cN=\{x=y\}$. On the other hand,
$\Phi_1(a_i{[1]})={\partial_{y_i}}{}+\eta\theta{\partial_{y_i}}{}$ and hence $\Phi_1(a_i{[1]})|_M={\partial_{y_i}}{}$, showing that $\Phi_1(\cdot)|_M$ is  the anchor of the tangent Lie algebroid. Here $a_i={\partial_{x_i}}{}\in\Gamma(A)$.
\end{remark}
}

\begin{remark}\label{rem:phiXwithQ}
{When $X$ satisfies the assumptions of Lemma \ref{lem:bijphi}, we can give an alternative description of the associated map $\Phi=\phi_X$ as follows:
\begin{equation}\label{eq:phiQ}
(\phi_X)_k({s_1,\dots,s_k})  := pr_{\vX(\varepsilon(\cM))}\left( \bigl[[\,\cdots [Q_{A[1]}|_{\cN}+X, \iota_{s_1}],\dots],\iota_{s_k}\bigr]|_{\varepsilon(\cM)}\right)
\end{equation}
for all $k\ge 0$ and $s_i \in \Gamma(A[1])$. This follows from the characterization of $\phi_X$
in terms of iterated Lie brackets of vector fields on $A[1]\times \cM$ given in Rem. \ref{rem:vfonN} 2),  together with the observation that $Q_{A[1]}$ as no component in the $\cM$-direction.}
 Here  the square brackets denote the Lie bracket of vector fields on $\cN$, and
 the projection is defined as follows. 
 
 We know by Rem. \ref{rem:aftercN}-1) that
 $\cN$ is a vector bundle with zero-section $\varepsilon(\cM)$. The total space of every vector bundle has a natural splitting  (into a horizontal and a vertical component) of its tangent space 
at points of the zero section. In particular, given $Y\in \vX(\cN)$, it makes sense to restrict $Y$ to $\varepsilon(\cM)$ and then take the component tangent to $\varepsilon(\cM)$. We denote the resulting vector field on $\varepsilon(\cM)$
by $pr_{\vX(\varepsilon(\cM))}(Y|_{\varepsilon(\cM)})$.
 \end{remark}

\subsection{Part II of the proof of the main theorem} The tensor product 
$\Gamma(S(A[1]^*)\otimes_{\RR}\vX(\cM)$ ({that we emphasize is taken}{} over $\RR$) agrees with the subspace of elements of $\vX(A[1]\times \cM)$ which project to the zero vector field
under $A[1]\times \cM \to A[1]$, therefore it is a graded Lie algebra. We denote
by $|_{\cN}$ the map \eqref{eq:iN}, which geometrically is the restriction to $\cN$ of vector fields on $A[1]\times \cM$.  

Fix a frame $\{{a}_i\}$ for $A$ (on an open subset $U\subset M$), which gives rise to coordinates $\{\xi_i\}$ on $A[1]$. This provides a specific extension of an element\footnote{We abuse notation by omitting restrictions to $U$: we should write $Y  \in \Gamma(S(A|_U[1]^*)\otimes_{C^{\infty}(U)}\vX(\cM|_U)$.}
$Y  \in \Gamma(S(A[1]^*)\otimesM\vX(\cM)$ to 
$\widetilde{Y}  \in \Gamma(S(A[1]^*)\otimes_{\RR}\vX(\cM)\subset \vX(A[1]\times \cM)$:   write $Y$ as
\begin{equation}
\label{eq:deco}
Y=\sum_j f_j\otimesM v_j
\end{equation}
 where each $f_j$
is a linear combination \emph{with constant coefficients} of products of the $\xi_i$ and where
 $v_j\in \vX(\cM)$, and take
\begin{equation}\label{eq:R}
\widetilde{Y}=\sum_j f_j\otimes_{\RR} v_j
\end{equation}

The following lemma is analogous to \cite[Lemma 3.2]{RajMarco}.
\begin{lemma}\label{lemma:binsuv} 
Let
$Y, Y' \in \Gamma(S(A[1]^*)\otimesM\vX(\cM)$. Fix a frame $\{{a_i}\}$ for $A$ (on an open subset $U\subset M$), {denote $s_i:=a_i[1]$}, and consider the lifts
 $\widetilde{Y}, \widetilde{Y'} \in \vX(A[1]\times \cM)$ given in eq. \eqref{eq:R}. The following equation holds for all\footnote{Here we regard each $s_i$ as an element of $\Gamma(A[1])$.}  subsets of  $\{s_i\}$:
\begin{multline}\label{frame}  
	     \phi_{[\widetilde{Y},\widetilde{Y'}]|_{\cN}}(s_{i_1} \cdots s_{i_n}) = \\
\sum_{l=0}^n \sum_{\tau \in Sh(l,n-l)} \epsilon(\tau)(-1)^{|Y'|(|s_{i_{\tau(1)}}| + \cdots + |s_{i_{\tau(l)}}|)} \bigl[\phi_Y (s_{i_{\tau(1)}} \cdots s_{i_{\tau(l)}}), \phi_{Y'}(s_{i_{\tau(l+1)}} \cdots s_{i_{\tau(n)}}) \bigr].
\end{multline}
Here the square bracket on the l.h.s. is the graded Lie bracket on
 $\vX(A[1]\times \cM)$, while on the r.h.s. it is the one on $\vX(\cM)$.
\end{lemma}
\begin{remark}
The equation \eqref{frame} generally does not hold if we replace $\widetilde{Y}, \widetilde{Y'}$
by arbitrary vector fields extending $Y,Y'$ or if we replace $s_{i_1} \cdots s_{i_n}$ by an arbitrary element of $\Gamma(S^n(A[1]))$.  
\end{remark}

\begin{proof} 
By applying the Jacobi identity for the 
  graded Lie bracket on
 $\vX(A[1]\times \cM)$, we obtain the following equation:
\begin{multline}\label{eq:jacsigns}
\Bigl[\bigl[\dots \bigl[[\widetilde{Y},\widetilde{Y'}], \iota_{s_{{i_{1}}}}\bigr],\dots\bigl],\iota_{s_{i_{n}}}\Bigr]=\\
\sum_{l=0}^n \sum_{\tau \in Sh(l,n-l)} \epsilon(\tau)(-1)^{|Y'|(|s_{i_{\tau(1)}}| + \cdots + |s_{i_{\tau(l)}}|)} 
\Bigl[\bigl[ [\dots[\widetilde{Y}, \iota_{s_{i_{\tau(1)}}}],\dots],\iota_{s_{i_{\tau(l)}}}\bigr],\\
\bigl[ [\dots[\widetilde{Y'}, \iota_{s_{i_{\tau(l+1)}}}],\dots],\iota_{s_{i_{\tau(n)}}}\bigr]
\Bigr].
\end{multline} 
By using Remark \ref{rem:vfonN} 2), we see that  applying $|_{{\varepsilon{(\cM)}}}$ to the l.h.s. of the equation \eqref{eq:jacsigns} we obtain the l.h.s. of the equation \eqref{frame}. 
The r.h.s. is more subtle\footnote{This is due to the fact  that the map
$|_{\varepsilon(\cM)}\colon C(A[1])\otimes_{\RR}\vX(\cM)\to 
C^{\infty}(M)\otimesM\vX(\cM)=\Gamma(\{0\}\oplus T\cM)|_{\varepsilon(\cM)}) \cong\vX(\cM)$ introduced in Remark \ref{rem:vfonN} 2) does not preserve Lie brackets.}.
  Writing $\widetilde{Y}$ as in \eqref{eq:R}, we have
$$\bigl[[\dots [\widetilde{Y}, \iota_{s_{i_{\tau(1)}}}],\dots],\iota_{s_{i_{\tau(l)}}}\bigr]=\sum_j g_j\otimes_{\RR} v_j,$$
where the $g_j$ are  linear combinations \emph{with constant coefficients} of products of the $\xi$'s (indeed, $g_j=\pm\iota_{s_{i_{\tau(1)}}}\cdots\iota_{s_{i_{\tau(l)}}}f_j$). Therefore, restricting to ${{\varepsilon{(\cM)}}}$ we obtain the element of $\vX({{\varepsilon{(\cM)}}})$ given by $(\sum_j g_j\otimes_{\RR} v_j)|_{\cM}=\sum_j (g_j|_M)\cdot v_j$, where the $(g_j|_M)$ are \emph{constants}. Because of this, {and using an analogous notation for $\widetilde{Y'}$,} no cross-terms appear in the bracket
 $$\Bigl[\sum_j (g_j|_M)\cdot v_j,\sum_k (g'_k|_M)\cdot v'_k\Bigr]=\sum_{j,k} (g_j|_M)(g'_k|_M)\cdot [v_j,v'_k],$$ which therefore equals $\bigl[\sum_j g_j\otimes_{\RR} v_j,\sum_k g'_k \otimes_{\RR} v'_k\bigr]|_{{{\varepsilon{(\cM)}}}}$. This shows that
 applying $|_{{{\varepsilon{(\cM)}}}}$ to the r.h.s. of equation \eqref{eq:jacsigns} is the same thing as applying $|_{{{\varepsilon{(\cM)}}}}$ to both entries of the brackets in the r.h.s. of equation \eqref{eq:jacsigns}, giving the r.h.s. of equation \eqref{frame}.
\end{proof}

Since $\vX(\cM)$  is a graded Lie algebra,   $\vX(\cM)[1]$ is a 
$L_{\infty}[1]$-algebra whose only non-trivial bracket is 
\begin{equation}\label{eq:curly}
\{P,R\} = (-1)^{|P|} [P,R],
\end{equation}
 where $|P|$ denotes the degree of $P$ in $\vX(\cM)$. (The sign in this formula arises from the d\'{e}calage isomorphism \eqref{deca}.)

 The following lemma and its proof are analogous to \cite[Prop. 3.3]{RajMarco}.
\begin{prop}\label{prop:limor}
Let $X$ be as in  Lemma \ref{lem:bijphi}.  Fix a frame $\{{a}_i\}$ for $A$ (on an open subset $U\subset M$), and consider the lift
 $\widetilde{X} \in \vX(A[1]\times \cM)$ given in eq. \eqref{eq:R}.
   Then, {over $U$,} $\phi_X$ is a curved $\linfty[1]$-algebra morphism $\Gamma(A[1]) \rightsquigarrow \vX(\cM)[1]$ if and only if 
\begin{equation*}
     [Q_{A[1]}, \widetilde{X}]|_{\cN} = -\frac{1}{2}[\widetilde{X},\widetilde{X}]|_{\cN}.
\end{equation*}    
\end{prop}
\begin{proof}
The maps $\phi_X$ are $C^{\infty}(M)$-multilinear by Lemma \ref{lem:bijphi}, so $\phi_X$ is a   {curved} $\linfty[1]$-morphism if and only if it satisfies the corresponding equations for
local frames of ${A[1]}$. Concretely (see Definition \eqref{li1mor}) this means: if and only if, for all subsets $({a}_{i_1},\dots,a_{i_n})$ of the local frame ($n\ge 0$), {writing $s_i:=a_i[1]$} we have
\begin{multline}\label{eqn:dglmorphA}
        \sum_{j={1}}^n \sum_{\tau \in Sh(j,n-j)} \epsilon(\tau) \phi_X (\{s_{i_{\tau(1)}}, \dots,  s_{i_{\tau(j)}}\} s_{i_{\tau(j+1)}} \cdots s_{i_{\tau(n)}} ) = \\ 
         \frac{1}{2} \sum_{l=0}^n \sum_{\tau \in Sh(l,n-l)} \epsilon(\tau) \left\{ \phi_X (s_{i_{\tau(1)}} \cdots s_{i_{\tau(l)}}), \phi_X(s_{i_{\tau(l+1)}} \cdots s_{i_{\tau(n)}}) \right\}.
\end{multline}
By
Lemma \ref{lemma:binsuv} the r.h.s. of \eqref{eqn:dglmorphA} equals $-\frac{1}{2}\phi_{[\widetilde{X},\widetilde{X}]|_{\cN}}(s_{i_1} \cdots s_{i_n})$, {as one sees using \eqref{eq:curly} and since $X$ has degree $1$.}

For the l.h.s. we observe that:
\begin{multline}
        \phi_{[Q_{A[1]}, \widetilde{X}]|_{\cN}} (s_{i_1} \cdots s_{i_n}) =\Bigl[\bigl[\dots[[Q_{A[1]},\widetilde{X}], \iota_{s_{i_1}}] \dots\bigl],\iota_{s_{i_n}}\Bigl]|_{{\varepsilon{(\cM)}}} \\
=  \sum_{j=0}^n \sum_{\tau \in Sh(j,n-j)} \epsilon(\tau)(-1)^{|s_{i_{\tau(1)}}| + \cdots + |s_{i_{\tau({j})}}|}  \Bigl[\bigl[[\dots [Q_{A[1]}, \iota_{s_{i_{\tau(1)}}}]\dots],\iota_{s_{i_{\tau(j)}}}\bigl]\vert_{{\varepsilon{(\cM)}}}\;,\quad\quad\\ \bigl[[\dots [\widetilde{X}, \iota_{s_{i_{\tau(j+1)}}}]\dots],\iota_{s_{i_{\tau(n)}}}\bigr] \Bigr]\vert_{{\varepsilon{(\cM)}}}
\end{multline} 
where {in the first equality we used Remark \ref{rem:vfonN}-2)}, and in the second equality we
repeatedly applied the Jacobi identity (exactly as for eq. \eqref{eq:jacsigns}), used that $X$ has degree $1$, and the identity $[\widehat{Q},\widehat{X}]\vert_{{\varepsilon{(\cM)}}}=[\widehat{Q}\vert_{{\varepsilon{(\cM)}}},\widehat{X}]\vert_{{\varepsilon{(\cM)}}}$ for all vector fields $\widehat{Q}\in  \vX(A[1])$ and
 $\widehat{X}\in C(A[1]) \otimes_{\RR} \vX(\cM)$. {Notice that the summand corresponding to $j=0$ vanishes, as $Q_{A[1]}$ vanishes on the body of $A[1]$.}
 Since  the term 
 \begin{equation}\label{eq:vertr}
[[\dots [Q_{A[1]}, \iota_{s_{i_{\tau(1)}}}]\dots],\iota_{s_{i_{\tau(j)}}}]\vert_{{\varepsilon{(\cM)}}}
\end{equation}
 commutes with the $\iota_{s_{i}}$'s,  applying the Jacobi identity allows us to put this term  
 just in front of the $\widetilde{X}$, with no extra signs arising. Switching the   term {\eqref{eq:vertr} with $\widetilde{X}$} and using the derived bracket construction 
\eqref{eq:adp}
to describe the $L_{\infty}[1]$-brackets on $\Gamma(A[1])$  
  we obtain the  l.h.s. of \eqref{eqn:dglmorphA}.  
 
 We conclude that the equation \eqref{eqn:dglmorphA} holds for all subsets of local frames $\{{a}_i\}$   if and only if $\phi_{[Q_{A[1]}, \widetilde{X}]|_{\cN}} =- \frac{1}{2}\phi_{[\widetilde{X},\widetilde{X}]|_{\cN}}$. Applying Lemma \ref{lem:bijeasy} finishes the proof.
\end{proof}

\begin{prop}\label{thm:equiv} Let $X\in \Gamma(S(A[1]^*))\otimesM \vX(\cM)$ {as in  Lemma \ref{lem:bijphi}}.
The following statements are equivalent:
\begin{enumerate}
     \item $Q_{tot}  := Q_{A[1]}|_{\cN}+X$ is a homological vector field on $\cN$.
     \item $\phi_X$ is a curved $\linfty[1]$-algebra morphism $\Gamma(A[1]) \rightsquigarrow \vX(\cM)[1]$.     \end{enumerate}  
\end{prop}

\begin{proof}
Fix a frame $\{s_i\}$ for $A$ (on an open subset $U\subset M$), and consider the lift
 $\widetilde{X} \in \vX(A[1]\times \cM)$ given in eq. \eqref{eq:R}. 
 Then $Q_{tot}$ is a homological vector field on $\cN$ if{f}
 $$0=\frac{1}{2}[Q_{tot},Q_{tot}]= \frac{1}{2}\left[Q_{A[1]}+\widetilde{X},Q_{A[1]}+\widetilde{X}\right]|_{\cN}=
 \frac{1}{2}[\widetilde{X},\widetilde{X}]|_{\cN}+
 [Q_{A[1]}|_{\cN},\widetilde{X}]|_{\cN},$$ 
 where in the second equality we used that the Lie bracket of vector fields on $\cN$ can be computed taking arbitrary extensions, and in the third equality we used
 $Q_{A[1]}^2 = 0$. This happens if{f} $\phi_X$ is a curved $\linfty[1]$-algebra morphism, by
 Prop. \ref{prop:limor}.
 \end{proof}

\subsection{Statement of the main theorem}
{We can finally state our main theorem:}
\bigskip

\begin{fthm}\label{charvf}
Let $A\to M$ be a Lie $n$-algebroid,
and $\cM$ be an $N$-manifold with body $M$. Let 
$\cN:=A[1]\times_M \cM$.
 There is a one-to-one correspondence between 
\begin{enumerate}
\item $\linfty$-actions of $A$ on $\cM$
\item  homological vector fields $Q_{tot}$ on $\cN$ for which the projection map $\cN \to A[1]$ is a $Q$-manifold morphism. 
\end{enumerate} 
\smallskip
\end{fthm} 
\smallskip

\begin{proof}
{We describe both assignments, which are clearly inverses of each other.}
\begin{itemize}
\item  $(1)\rightarrow (2):$ An $\linfty$-action is a map $\Gamma(\wedge A)\to \vX(\cM)$
as in Def. \ref{def:liac}, which we can interpret as
$C^{\infty}(M)$-multilinear map $\Phi\colon \Gamma(S(A[1])\to \vX(\cM)[1]$ of degree $0$ such that $\Phi_1(a[1])|_M=\rho(a)$ for all $a\in \Gamma(A_0)$, and $\Phi_0
\in \vX^{vert}_1(\cM)$. {By Lemma \ref{lem:bijphi} and Prop.  \ref{thm:equiv},  $Q_{tot} :=Q_{A[1]}|_{\cN}+X$ is
a  homological vector field on $\cN$}, where  $X\in \Gamma(S(A[1]^*)\otimesM\vX(\cM)$ of  degree $1$
satisfies $\Phi=\phi_X$ (the latter is defined in Lemma \ref{lem:bijeasy}).

\item $(2)\rightarrow (1):$ Given the homological vector field $Q_{tot}$, define $X:=Q_{tot}-Q_{A[1]}|_{\cN}$, {which satisfies the hypotheses of Lemma \ref{lem:bijphi}.}
Then $\phi_X\colon \Gamma(S(A[1]))\to \vX(\cM)[1]$ is a curved
$L_{\infty}[1]$-algebra morphism by Prop.  \ref{thm:equiv}, and the  corresponding $L_{\infty}$-algebra morphism
$\Gamma(\wedge A)\to \vX(\cM)$ is an $\linfty$-action of $A$ on $\cM$ {by Lemma \ref{lem:bijphi}}.
\end{itemize}
\end{proof}

\begin{remark}\label{rem:tangcm}
{{In the setting of Thm. \ref{charvf},}
the N-manifold $\cM$  inherits a homological vector field, which is vertical in the sense that it annihilates (pull-backs of) functions on $M$. Indeed, since the projection map $\cN \to A[1]$ is a $Q$-manifold morphism and the homological vector field on $A[1]$ vanishes on the zero section $M$,  $Q_{tot}$ restricts to ${{\varepsilon{(\cM)}}}\cong\cM$. Alternatively, the homological vector field of $\cM$ can be described as the zero-th component of the $\linfty$-action on $\cM$.}
\end{remark}
{In view of Theorem \ref{charvf}, when $Q_{\cM}$ is prescribed we obtain:}
\begin{thm}\label{charvfQM}
Let $A\to M$ be a Lie $n$-algebroid,   and let $(\cM,Q_{\cM})$ be a $Q$-manifold such that
$Q_\cM\in \vX^{vert}_1(\cM)$. Denote 
$\cN:=A[1]\times_M \cM$.

 There is a one-to-one correspondence between 
     \begin{enumerate}
\item $\linfty$-actions of $A$ on $\cM$ compatible\footnote{See   Definition  \ref{def:action}.}
 with $Q_\cM$
\item  homological vector fields $Q_{tot}$ on $\cN$ for which
\begin{equation*}
(\cM,Q_\cM)\overset{\varepsilon}{\to} (\cN, Q_{tot}) \to (A[1],Q_{A[1]})
\end{equation*}
 is a sequence of $Q$-manifold morphisms.
\end{enumerate}  \end{thm}

%%%%%%%%%%%%%
\section{Comparison with other homotopy structures} \label{sec:comp}

 {
In the Introduction 
we saw that, after
applying the N-manifold version of Batchelor's theorem, the set-up addressed in this paper is the following:
$A\to M$ a Lie algebroid\footnote{More generally, in this paper we allow Lie $n$-algebroids.}, $V\to M$ a graded vector bundle  concentrated in degrees $\le 0$, 
 and $Q_{tot}$ a homological vector field  on $\cN:=A[1]\times_M V[1]$ such that the projection 
$$(\cN,{Q_{tot}})\to (A[1],Q_{A[1]})$$ is a  morphism of $Q$-manifolds. The situation is summarized by this diagram:
\begin{equation*} 
\SelectTips{cm}{}
\xymatrix{
\cN\ar[d]  \ar[r] & V[1]\ar[d]^p \\
A[1] \ar[r]^{\pi}& M}
\end{equation*}
We describe how the homological vector field ${Q_{tot}}$ can be ``decomposed'' in various equivalent ways according to various notions of degree,
and  compare these ``decompositions''.}

\begin{itemize}
\item[a)] 
By\footnote{Item a) holds also replacing $V[1]$ by any N-manifold.}
  Thm. \ref{charvf}, ${Q_{tot}}$ is encoded by the given Lie algebroid structure on $A$ together with   an $L_{\infty}$-action, \emph{i.e.} 
a curved $L_{\infty}$-morphism
\begin{equation*} 
F\colon \Gamma(A)  \rightsquigarrow \vX(V[1])
\end{equation*}
with components $F_k$ ($k\ge 0)$
satisfying compatibility conditions (see Def. \ref{def:liac}).

\item[b)] Following  \cite{CCShengHigherExt}, since $\cN=(A\oplus V)[1]$ is a graded vector bundle over $M$,  $Q$ is encoded by a  
 a split Lie $n$-algebroid structure on it, see  \S \ref{subsec:split}. The Lie $n$-algebroid structure consists of an anchor map $\rho\colon A\oplus V_0\to TM$ and an $L_{\infty}$-algebra structure on the sections $$\Gamma(A\oplus V),$$ whose  multibrackets  we denote by $l_k$ for $k\ge 1$. Notice that here we did not use the fact that ${Q_{tot}}$ projects to $Q_{A[1]}$.
 
\item[c)] Following\footnote{Given a Lie algebroid $A\to B$ and a vector bundle $E\to B$,    
a Kapranov DG-manifold structure is given by a homological vector field $Q$ on $A[1]\oplus E$ such that both the projection  $A[1]\oplus E\to A[1]$ and the inclusion $A[1]\hookrightarrow A[1]\oplus E$ are Q-manifold morphisms \cite[\S 4]{KapranovDGman}.  In \cite[Thm. 4.10]{KapranovDGman} various equivalent characterizations are given, and here we are extending in a straightforward way  characterization (6) there.}
 \cite[Thm. 4.10]{KapranovDGman}, we use that
$\cN=\pi^*(V[1])$ is   a vector bundle over $A[1]$, with the property that the vector bundle projection preserves homological vector fields (here $\pi\colon A[1]\to M$). 
 Applying Voronov's derived bracket construction, in complete analogy to \S \ref{subsec:split}, we see that ${Q_{tot}}$ is encoded by the Lie algebroid structure on $A$ and a curved $L_{\infty}$-algebra structure on the sections $$C(A[1])\otimesM \Gamma(V)=\Gamma(\wedge A^*\otimes V),$$
whose  multibrackets  we denote by $m_k$ for $k\ge 0$, 
with the property that 
\begin{itemize}
\item $m_k$ is $\Gamma(\wedge A^*)$-multilinear for $m\ge 2$ and  
\item $m_1\colon \Gamma(\wedge A^*\otimes V)\to \Gamma(\wedge A^*\otimes V)$ satisfies $
m_1(\alpha\otimes v)=(Q_{A[1]}\alpha)\otimes v+(-1)^{|\alpha|}\alpha\otimes (m_1v)$.
\end{itemize}
(In general $m_1$ is does not square to zero, \emph{i.e.} it is not an representation up to homotopy. When the inclusion of $A[1]$ in $\cN$ as the zero section  is a $Q$-manifold morphism, then $m_0$ vanishes and $m_1$ squares to zero.) 
\end{itemize}

Now we  compare the 3  ``decompositions'' of ${Q_{tot}}$ described above.
The (polynomial) functions on $\cN$ are given by $\Gamma(\wedge  A^*)\otimes_{C^{\infty}(M)} \Gamma(\wedge V^*)$, hence -- in addition to the usual degree as functions on an N-manifold -- they have a bidegree: we will say that an element of
$\Gamma(\wedge^p A^*)\otimes_{C^{\infty}(M)} \Gamma(\wedge^q {V^*})$
\emph{polynomial bidegree} $(p,q)$.

\begin{remark}  Since the graded vector bundle projection $\cN=\pi^*(V[1])\to A[1]$ maps ${Q_{tot}}$ to $Q_{A[1]}$, the vector field ${Q_{tot}}$ can be written as a $C(\cN)$-linear combination of sections of this graded vector bundle   plus a lift of $Q_{A[1]}$.
The description in coordinates is as follows: choose coordinates $x$ on $M$, linear coordinates $\xi$ (of degree 1) on the fibers of $A[1]$, and linear coordinates $\eta^i$ on the fibers of $V_i[1]$ for $i=0,1,2,\dots$, where we omit the indexing of the coordinates for the sake of readability.
Then ${Q_{tot}}=Q_{A[1]}+X$ where in coordinates, omitting coefficients of degree zero (\emph{i.e.} in $C^{\infty}(M)$) for the sake of readability, $Q_{A[1]}$ has the form $\xi {\partial_x}+\xi\xi {\partial_{\xi}}$ and $X$ has the form
 $$\left(\xi\xi+\xi\eta^1 +\eta^1\eta^1+\eta^2\right)\frac{\partial}{\partial{\eta^1}}+\left(\xi\xi\xi+\cdots+\eta^3\right)\frac{\partial}{\partial{\eta^2}}+\cdots.$$ The polynomial bidegree of a monomial coefficient function is $(p,q)$ if the number of ``$\xi$'' appearing is $p$ and the number of ``$\eta$'' appearing is $q$.
\end{remark}

{
For each other 3 ``decompositions'' of $Q$ above, a computation in coordinates shows that we have:
\begin{itemize}
\item[a)] for the $L_{\infty}$-action: for all  $k\ge 0$, $F_k$ is determined by the summands in $X$ whose coefficients have bidegree $(k,*)$ (with $*$ arbitrary),
\item[b)] for the $L_{\infty}$-algebra structure on $\Gamma(A\oplus V)$: for all $k\ge 1$, $l_k$ is determined by the summands in $Q$ whose coefficients have bidegree $(p,q)$ with $p+q=k$, except for $l_2$ which is determined by $Q_{A[1]}$ too.
\item[c)] for the curved $L_{\infty}$-algebra structure on $\Gamma(\wedge A^*\otimes V)$: for all $k\ge 0$, $m_k$ is determined by the summands in $X$ whose coefficients have bidegree $(*,k)$ (with $*$ arbitrary), except for $m_1$ which is determined by $Q_{A[1]}$ too.
\end{itemize}
}

 \begin{remark}
{In item a) we ``decomposed'' $Q_{tot}$ according to the corresponding $L_{\infty}$-action. We now comment on how the latter can be seen in terms of  multi-anchors.}
{Recall that $\cN=p^*A[1]\to V[1]$ is   a graded vector bundle, and  by Rem. \ref{rem:tangcm} 
 the homological vector field $Q_{tot}$   is tangent to the zero section $V[1]$.
Hence $Q_{tot}$  makes this graded vector bundle a 
\emph{$L_{\infty}[1]$ algebroid with graded base} in the sense of  Vitagliano, see \cite[\S 3]{VitReprSHLR} and references therein. This is the geometric notion corresponding to the
\emph{homotopy Lie-Rinehart algebras}  defined in \cite[\S 3]{VitReprSHLR} and in  \cite[Def. 1.6]{VitaglianoFol}.
The latter in particular carry ``multi-anchors'', which in this case take tuples of sections of $p^*A[1]$ to vector fields on $V[1]$. The $L_{\infty}$ action we obtained in Thm. \ref{charvf} is {expected to be} the restriction of the multi-anchors to sections which are pull-backs of  sections of $A[1]$.}
\end{remark}

 \vspace{5mm}
 \part{Examples and applications }

The second part of the paper is devoted to examples of Thm. \ref{charvf}, and in all of them we consider $L_{\infty}$-actions of Lie algebroids  {(rather than of arbitrary Lie $n$-algebroids)}.  {In \S \ref{sec:examplesDG} we briefly discuss how vector bundles of $Q$-manifolds give classes of examples. In the remaining sections we    discuss in detail several classes of examples, as outlined    in the Introduction.}

 %%%%%

 \section{Vector bundles of $Q$-manifolds}\label{sec:examplesDG}

Let $\cA$ be an N-manifold with body $M$, and denote by $\pi_{\cA,M}\colon \cA\to M$
the natural projection. Let
$\cE\to \cA$ be a ($\ZZ$-)graded vector bundle  over $\cA$, which we assume to be {finite dimensional}.
There is a non-canonical isomorphism of graded vector bundles
\begin{equation}\label{eq:isoRaj}
\cE\cong  (\pi_{\cA,M})^*\cM=\cA\times_M \cM
\end{equation}
where\footnote{More formally: $\cM:=(0_{M,\cA})^*\cE$ for $(0_{M,\cA})$
the embedding of the body $M$ into $\cA$.} $\cM$ is the graded vector bundle over $M$ given by: $$\cM:=\cE|_{M}.$$  The existence of the isomorphism is proven by
Mehta \cite[Thm. 2.1]{m:lamods}, and the fiber product is over $\pi_{\cA,M}$ and the vector bundle projection of $\cM$.
{\begin{equation*} 
\SelectTips{cm}{}
\xymatrix{
\cE\ar[d]  \ar[r] & \cM\ar[d] \\
\cA \ar[r]^{\pi_{\cA,M}}& M}
\end{equation*}
}

We now bring homological vector fields into play. Let $A$ be a Lie $n$-algebroid, and consider a vector bundle  
\begin{equation}\label{eq:vbQ}
(\cE,Q_{\cE})\to (A[1],Q_{A[1]})
\end{equation}
in the category of $Q$-manifolds, {whose fibers are assumed to be concentrated in negative degrees}.
 Notice that, since $Q_{A[1]}$ vanishes on the body $M$, the homological vector field $Q_{\cE}$ restricts to {a vertical vector field on} the N-manifold $\cM$, so we actually have a sequence of $Q$-manifold morphisms
$$
(\cM,(Q_\cE)|_{\cM})\to (\cE,Q_{\cE}) \to (A[1],Q_{A[1]}).
$$
The choice of an isomorphism as in \eqref{eq:isoRaj} allows us to apply Thm. \ref{charvfQM} and obtain an $L_{\infty}$-action of $A$ on $\cM$ {compatible with $(Q_\cE)|_{\cM}$}.

Below, we provide some classes of examples of vector bundles as in \eqref{eq:vbQ}.
   
\begin{ex}
 [Lie algebroid modules]
When $A$ is a Lie algebroid and $Q_{\cE}$ is a fiberwise linear vector field, we obtain a Lie algebroid module in the sense of Vaintrob \cite{vaintrob} (see also \cite[\S4]{m:lamods}). Upon choosing an isomorphism as in \eqref{eq:isoRaj}, one obtains 
representations  up to homotopy \cite[Lemma 4.4]{m:lamods}, which we consider in \S \ref{sec:special}.
\end{ex}

\begin{ex}[{$Q$-bundles}]
When the bundle is locally trivial in the category of $Q$-manifolds, we obtain an instance of the $Q$-bundles  of Kotov-Strobl \cite{Qbundles}.
  \end{ex}

\begin{ex}[Principal bundles]
There are situations in which the bundle \eqref{eq:vbQ} is a 
principal $\RR[n]$-bundle   (in the category of $Q$-manifolds), as 
considered in \cite{Lupercio:2012ha} in order to encode classical geometric structures.
An example appears in \S \ref{sec:r2}.
\end{ex}

\section{Lie algebroids extensions }\label{sec:algebroids}

In this section  we show that, given a surjective Lie algebroid morphism $E\to A$ covering the identity on the base (this is a special case of Lie algebroid extension), a  choice of linear splitting induces an $L_{\infty}$-action of $A$ on the kernel of the morphism (shifted by $1$).
A typical example is when $E$ is a transitive Lie algebroid and $\rho\colon E\to TM$ is its anchor map,  \emph{e.g.} an Atyiah Lie algebroid.

{This extends the following  well-known correspondence  (see  \cite[\S 4.5]{MK2}, also \cite{Cr-preq}).
Fix a Lie algebroid $A$ and a vector bundle $V$ over the same base. Given a representation of $A$ on $V$ (\emph{i.e.} a flat connection $\nabla$) and a 2 cocycle $\omega$ for the representation, the direct sum $A\oplus V$ is endowed with a Lie algebroid structure\footnote{The Lie algebroid bracket is given by $[(a,v),(b,w)]=\left([a,b]_A, \nabla_aw-\nabla_bv+\omega(a,b)\right)$.}
such that 
$V\hookrightarrow A\oplus V\to A$ is an \emph{abelian} extension of Lie algebroids. (Here abelian means that the Lie algebroid structure on $V$ is trivial.) Conversely, given any abelian extension of Lie algebroids, 
 a choice of splitting delivers a flat connection $\nabla$ and a 2 cocycle $\omega$.
}

{
An \emph{arbitrary} (not necessarily abelian) extension is of the form $V\hookrightarrow E \to A$
where $V$ now is a bundle of Lie algebras, \emph{i.e.} a Lie algebroid with vanishing anchor. The choice of a splitting delivers a connection $\nabla$ and a tensor $\omega\in \Gamma(\wedge^2A^*\otimes V)$. The connection has curvature, so it no longer defines a representation, and $\omega$ is not a cocycle in any strict sense, hence \emph{\`a priori} it is not clear how to interpret the compatibility equations they satisfy {\cite{MK2}}. As we explain in this section, our formalism shows that $\nabla$ and $\omega$ assemble into an $L_{\infty}$-morphism.
}

\subsection{From Lie algebroid extensions to $L_{\infty}$-actions}

We start from the following data:
\begin{itemize}
\item[i)] Lie algebroids $(E,[\ ,\,]_E)$ and $(A,[\ ,\,]_A)$  over $M$, together with a surjective Lie algebroid morphism $\rho\colon E\to A$ covering $\id_M$, as follows:
 \begin{equation}\label{eq:surj}\SelectTips{cm}{}
\xymatrix{
E\ar[d] \ar@{->>}[r]^{\rho}   & A \ar[d]  \\
M\ar[r]^{\id_M}& M 
}
\end{equation}
\item[ii)] A linear splitting $\sigma\colon A\to E$ of the short exact sequence of Lie algebroids
 $$\g_M\hookrightarrow E \overset{\rho}{\to} A,$$ 
 where $\g_M:=\ker(\rho)$ (a bundle of Lie algebras). Notice that a choice of linear splitting  is always possible. 
 \end{itemize}
  We obtain an identification
$$E=\sigma(A)\oplus \g_M\cong  A \oplus\g_M,$$
where  the above isomorphism of vector bundles is $({\rho|_{\sigma(A)},\id_{\g_M}})$ and has inverse $({\sigma,\id_{\g_M}})$. One can then transport the structure, so as to obtain an  isomorphism of $Q$-manifolds: 
\begin{equation}\label{eq:isogr}
(E[1], Q_{E[1]})\cong ( A[1]\times_M\g_M[1], Q)=:\cN
\end{equation}
 where $Q_{E[1]}$ encodes the Lie algebroid structure on $E$ and $Q$ corresponds to it under the above isomorphism. The projection $\cN\to A[1]$ maps $Q$ to the homological vector field $Q_{A[1]}$, since   $\rho\colon E\to A$
is a Lie algebroid morphism. Hence we are in the situation of item (2) in
Thm. \ref{charvf}, with  $\cM:=\g_M[1]$.
By that theorem, there is an associated $L_{\infty}$-action of $A$ on $\g_M[1]$. Summarizing:

{
\begin{prop}\label{prop:sumsurjLA}
A surjective morphism of Lie algebroids $\rho\colon E\to A$  covering $\id_M$, together with a choice of linear splitting of $\rho$, induce an $L_{\infty}$-action of $A$ on $\g_M[1]$. Here
$\g_M:=\ker(\rho)$.
\end{prop}
}

\subsection{{Description of the  $L_{\infty}$-action in graded geometric terms}}
\label{subsec:extclasgraded}

We describe the $L_{\infty}[1]$-algebra morphism  corresponding via d\'ecalage to the $L_{\infty}$-action, \emph{i.e.} $$\phi_X\colon \Gamma(S(A[1]))\to \vX(\cM)[1]$$ as given by Lemma \ref{lem:bijeasy}, where $X$ is determined by $Q=Q_{A[1]}|_{\cN}+X$. {A description in classical terms will be given in the next subsection.} 
{Recall from}{}{ Rem. \ref{rem:phiXwithQ} {that}{} the map $\phi_X$ can be computed {by}{}
{taking iterated brackets with the homological vector field $Q$ on $\cN$.}
Using the isomorphism \eqref{eq:isogr}
(which restricts to the identity of $\g_M[1]$), we will compute $\phi_X$  
working on $E[1]$
 with the homological vector field $Q_{E[1]}$.} 
We refer the reader to Appendix \ref{app:Q}
and \cite[\S 1.1]{ZZL} for the background material on degree 1 graded manifolds used in the sequel.

The zero-th component of $\phi_X$ maps $1\in \RR$ to $X|_{\cM}=Q|_{\cM}=(Q_{E[1]})|_{\cM}\in \vX_1(\cM)$, which just is the homological vector field $Q_{\g_M[1]}$ on $\g_M[1]$ encoding its Lie algebroid structure.  

We now describe the first component $\Gamma(A[1])\to \vX_0(\cM)[1]$ of $\phi_X$ {using  Rem. \ref{rem:phiXwithQ}}.
Let\footnote{{For the sake of readability we will denote by $a$ also the corresponding element of $\Gamma(A)$.}}
 $a\in \Gamma(A[1])$, we have:
$$\phi_X(a)=pr_{\vX({{\varepsilon{(\cM)}}})}([Q,\iota_{a}]|_{\varepsilon(\cM)})=pr_{\vX({{\varepsilon{(\cM)}}})}([Q_{E[1]},\iota_{\sigma(a)}]|_{\varepsilon(\cM)}).$$
 Notice that $\vX_0(\cM)$ {identifies with}{} $\CDO(\g_M)$, the {space of}{} covariant differential operators\footnote{That is, $\RR$-linear maps $D:\Gamma(\g_M)\to \Gamma(\g_M)$ satisfying the derivation property
  $D(f\mu)=fD(\mu)+\underline{D}(f)\mu$  for all $f\in C^{\infty}(M)$ and $\mu \in \Gamma(\g_M)$,
 where $\underline{D}\in \vX(M)$ denotes the symbol of $D$.
 } 
  on the vector bundle $\g_M\to M$ {(specifically,{} the isomorphism  identifies $X\in \vX_0(\cM)$ with $[X,\;]$, the Lie bracket of $X$ with elements of $\vX_{-1}(\cM)$).}  
  Hence to describe $\phi_X(a)$ it suffices to describe the action of the corresponding covariant differential operator on $v\in \Gamma(\g_M)$, \emph{i.e.} it suffices to describe $[\phi_X(a),\iota_v]$ where 
$\iota_v\in \vX_{-1}(\g_M[1])$.
{The element of $\CDO(\g_M)$ corresponding to $\phi_X(a)$ is just the restriction to $ \Gamma(\g_M)$ of the element of $\CDO(E)$ corresponding to $[Q_{E[1]},\iota_{\sigma(a)}]$.}
We have\footnote{This lies in 
  $\vX_{-1}(\cM)$, as a
 consequence of $[\sigma(a),v]_E\subset \ker(\rho)=\g_M$.} 
 $[[Q_{E[1]},\iota_{\sigma(a)}],\iota_v]=\iota_{[\sigma(a),v]_E}$, {see eq. \eqref{eq:br-rho}.}  
   We conclude that, under the identification $\vX_0(\cM)\cong \CDO(\g_M)$, we have: 
\begin{equation}\label{eq:phix1}
\phi_X(a)=[\sigma(a),\ ]_E.
\end{equation}
 
Finally, we  describe the second component $\Gamma(S^2(A[1]))\to \vX_{-1}(\cM)[1]$  of $\phi_X$.
Fix $a,b\in \Gamma(A[1])$.
 We have
$$\phi_X(a,b)=pr_{\vX({\varepsilon(\cM)})}([[Q,\iota_{ a}],\iota_{b}]|_{\varepsilon(\cM)})=pr_{\vX(\varepsilon(\cM)}([[Q_{E[1]},\iota_{\sigma(a)}],\iota_{\sigma(b)}]|_{\varepsilon(\cM)}),$$
where the restriction to $\varepsilon(\cM)$ can be omitted since we are dealing with a degree $-1$ vector field.
We know that $[[Q_{E[1]},\iota_{\sigma(a)}],\iota_{\sigma(b)}]=\iota_{[{\sigma(a)},{\sigma(b)}]_E}$,
so 
\begin{equation}\label{eq:phix2}\phi_X(a,b)=\iota_{[{\sigma(a)},{\sigma(b)}]_E-\sigma([a,b]_{A})}.\end{equation}
Since $\vX_i(\cM)$ is trivial  for $i< -1$, $\phi_X$ has no higher components.

{By construction, the $\linfty$ action of $A$ on $\g_M[1]$ obtained in Prop. \ref{prop:sumsurjLA} is
  compatible with ${Q_{\g_M[1]}}{}$. Hence, by Prop. \ref{prop:liacQ}, it is tantamount to an  $L_{\infty}$-morphism  
  \begin{equation}\label{eq:lintrans}
  \bigl(\Gamma(A),[\ ,\,]_A\bigr)\rightsquigarrow \bigl(\vX({\g_M[1]}{}), -[Q_{{ \g_M[1]}{}}, \ ], [\  ,\,]\bigr)
\end{equation}
 consisting of a unary and a binary component, given by eq. \eqref{eq:phix1} and \eqref{eq:phix2} after
 adjusting signs due to the d\'ecalage isomorphism.}

\begin{remark}\label{rem:Qlocaltranslie}
Let us describe the homological vector field $Q$ of equation \eqref{eq:isogr} in the case $A=TM$.  One chooses local coordinates $(x,v,\xi)$ on $T[1]M\oplus \g_M[1]$, where $x^i$ are coordinates on $M$, $v^i$ the corresponding coordinates on the fibers of $T[1]M$, and $\xi^i$ are coordinates on $\g_M[1]$, then locally we have:
\begin{align}\label{eq:Qtot:atiyah}
-Q=
\frac{1}{2}\sum_{i,j,k} c_{ij}^k\xi^i\xi^j\frac{\partial}{\partial{\xi^k}}
+
\sum_{i,j,k} v^i\left(\Gamma_{ij}^k \xi^j \frac{\partial}{\partial{ \xi^k}} - \frac{\partial}{\partial{x^{{i}}}}\right)
+
 \sum_{i,j,k}\frac{1}{2}\omega_{ij}^k v^i v^j\frac{\partial}{\partial{\xi^k}}
\end{align} 
where $c_{ij}^k,\ \Gamma_{ij}^k, \omega_{ij}^k \in C^\infty(\mathcal{U})$ are defined by:
\begin{align*}
[\xi_i,\xi_j]_{E}=\sum_{k} c_{ij}^k\xi_k,\quad
 [v_i,\xi_j]_E=\sum_{k}\Gamma_{ij}^k \xi_k,\quad
 [v_i,v_j]_E=\omega_{ij}^k \xi_k.
\end{align*}
Here, $v_i,\xi_j$ denotes the basis of sections of $E\simeq TM\oplus \g_M$ dual to  ${v}^i,\xi^j.$ One may notice in particular, on the right hand side of eq. \eqref{eq:Qtot:atiyah}, how each term corresponds to a different component of the $L_\infty$-action $\phi_X$.
\end{remark} 

\subsection{Description of the  $L_{\infty}$-action in classical terms}
\label{subsec:extclas}

It is possible to describe the $L_{\infty}$-morphism \eqref{eq:lintrans} without referring to the graded geometrical setting, as we now explain.

 The splitting $\sigma:A\to E$ induces an $A$-connection $\nabla$ on $\g_M$  given by:
    $$\nabla_av:=[\sigma(a),v]_E,$$
together with a $\g_M$-valued $2$-form on $A$, namely $\omega\in \Omega^2(A,\g_M):=\Gamma(\wedge^2 A^*\otimes \g_M)$, which is defined by
    $$\omega(a,b):=[{\sigma(a)},{\sigma(b)}]_E-\sigma([a,b]_{A}).$$
{It is well known \cite{MK2, olivierextensions} that  the Lie algebroid $A$ and the bundle of Lie algebras $\g_M$, together with $\nabla$ and $\omega$, allow to reconstruct the Lie algebroid extension    
$\rho \colon E\to A$.}

Next, consider the graded vector space $\WDGLA(\g_M)$ with grading given by: 
\begin{align*}
 \WDGLA_{-1}:=\Gamma(\g_M), \quad
 \WDGLA_{0}:=\der(\g_M), 
\end{align*} 
 where $\der(\g_M)$ denotes the space of covariant differential operators that are derivations of the Lie bracket on $\g_M$, namely satisfying              
   $D[\mu,\nu]_{\g_M}=[D\mu,\nu]_{\g_M}+[\mu,D\nu]_{\g_M}$  for all $\mu, \nu \in \Gamma(\g_M)$.
  Then $\WDGLA(\g_M)$ as a structure of a DGLA, with differential  $\textbf{d} \mu :=-[\mu,\cdot]_{\g_M}$ and bracket given by:
\begin{align*}
[ D , D']&= D\circ D'-D'\circ D,\\
[ D , \mu]&= D(\mu).
\end{align*}

\begin{prop}\label{prop:liclasstransLA}
The  components   
\begin{align*}
F_1:=\nabla&:\Gamma({A}) \to \WDGLA_{0},\\
F_2:=\omega\,&:\Gamma({\wedge^2 A}) \to \WDGLA_{-1},
\end{align*}
define a $C^\infty(M)$-linear 
$L_\infty$-morphism $F:\Gamma(A)\leadsto \WDGLA(\g_M)$.

\end{prop}
\begin{proof}
It easily follows from the Jacobi identity in $E$ that $\nabla$ and $\omega$ satisfy the following compatibility conditions:
\begin{gather}
\label{eq:compJ1} \nabla_a[\mu,\nu]_{\g_M}=[\nabla_a \mu,\nu]_{\g_M}+[\mu,\nabla_a \nu]_{\g_M},\\
\label{eq:compJ2} \nabla_{[a,b]_A}-[\nabla_a,\nabla_b]=-\ad^{\g_M}\circ \omega(a,b),\\
\label{eq:compJ3} \oint_{a,b,c} \nabla_a(\omega(b,c))-\omega([a,b]_A,c) =0,
\end{gather}
for any $a,b,c\in \Gamma(A)$ and $v,w\in \Gamma(\g_M)$. {Here   $\oint$ denotes cyclic permutations.}

The condition \eqref{eq:compJ1} ensures that $F$ takes indeed values in $\WDGLA(\g_M)$.
The conditions for $F$ to define an $L_\infty$-morphism (see eq. \eqref{eq:explicit1} and \eqref{eq:explicit2} in Remark \ref{rem:explicit-Linfty-morphism}) coincide with the last two compatibility conditions \eqref{eq:compJ2} and \eqref{eq:compJ3}. Finally $F_1$ and $F_2$ are clearly $C^\infty(M)$-linear.
\end{proof}

To see that the $L_{\infty}$-morphism \eqref{eq:lintrans} agrees with the one given in 
Prop. \ref{prop:liclasstransLA}, notice that
$\WDGLA(\g_M)$ {is canonically identified, as a DGLA, with} $\vX_{-1}(\g_M[1])\oplus \{X\in \vX_0(\g_M[1]): [Q_{\g_M[1]},X]=0\}$, which {is}{} a sub-DGLA of the DGLA appearing in \eqref{eq:lintrans}. {In other words, we can regard $F$ as taking values in $\vX(\g_M[1])$.}{}
Further, the components $F_1$ and $F_2$ correspond to {the  components of $\phi_X$ as given in} eq.  \eqref{eq:phix1} and \eqref{eq:phix2}.

%%%%%%%%%%%
\section{Representations up to homotopy of Lie algebroids}\label{sec:special}

\subsection{Representations up to homotopy}\label{infrep}
{Fix a Lie algebroid $\pi\colon A\to M$ and a {(finite rank)} graded vector bundle $V\to M$.} 
We recall briefly the notion of  \emph{representation up to homotopy} of $A$ on $V$, see \cite{rep-hom}\cite{MackenzieVB.RajAlfonsoVBas}. It is given by 
 an operator $D$ on $$\Omega(A,V):=\Gamma(\wedge A^*\otimes V)$$ of degree $1$ with $D^2=0$, satisfying
 $D(\omega\eta)=(d_A\alpha)\eta+(-1)^{k}\alpha(D\eta)$ for all $\alpha\in \Gamma(\wedge A^*)$ and $\eta \in \Omega^k(A,V)$.
 
As explained in \cite[Prop. 3.2]{rep-hom}, $D$ is determined by its action on $\Gamma(V)$, which delivers the following data: $\partial\in \End^{1}(V)$, an $A$-connection $\nabla$ on $V$, and forms $\omega_i\in \Omega^i(A,\End^{1-i}(V))$ for $i\ge 2$, satisfying equations equivalent to the condition $D^2=0$. Concretely, for all $\eta\in \Omega(A,V)$
\begin{equation}\label{eq:Deta}
D(\eta)=\partial(\eta)+d_{\nabla}(\eta)+\omega_2\circ \eta+\cdots
\end{equation}
where $d_{\nabla}\colon  \Omega^k(A,V)\to  \Omega^{k+1}(A,V)$ is defined as usual by $d_{\nabla}(\alpha \otimes v)=d_A\alpha \otimes v+(-1)^k\alpha\wedge \nabla v$ for all $\alpha\in \Omega^k(A)$ and $v\in \Gamma(V)$.

We will need the following notion. Given a graded vector bundle $W\to M$, we denote by $\CDO(W)\to M$ the graded vector bundle whose sections are \emph{covariant differential operators}, \emph{i.e.} linear maps $Y: \Gamma(W) \to \Gamma(W) $ such that there exists a vector field
$\underline{Y} $  on $M$ (called \emph{symbol})  for which the Leibniz rule $Y(f\cdot w) = \underline{Y} (f) w + f \cdot Y(w)$ is satisfied for all  $f \in C^\infty(M)$ and $w \in \Gamma(W)$. Notice that the graded commutator bracket
\begin{equation}\label{eq:grcommut}
[Y,Z]=Y\circ Z-(-1)^{|Y||Z|}Z \circ Y
\end{equation}
 makes $\Gamma(\CDO(W))$ into a graded Lie algebra, and that $\CDO^i(W)=\End^i(W)$ for all $i\neq 0$. (In other words: covariant differential operators of non-zero degree are just endomorphisms).

 We now recall a description of representations up to homotopy in terms of graded geometry. Notice that $A\times_M V=\pi^*V\to A$ is a vector bundle over $A$.
\begin{lemma}\label{lem:inftyrep}
{Assume that $V$ is concentrated in degrees $\le 0$.}
There is a bijection between: 
\begin{itemize}
\item representations up to homotopy of $A$ on $V$
\item  homological vector fields on $\cN:=A[1]\times_M V[1]$  which are linear in the fibers of the projection to $A[1]$ and which are mapped to $Q_{A[1]}$ by this projection. \end{itemize}
\end{lemma}
The above lemma is {well known}{}, see for instance   \cite[Rem. 4.2, Lemma 4.4]{m:lamods} or \cite[Lemma 3.1]{CCShengHigherExt}, where the dual representation up to homotopy is used in order to construct\footnote{
In brief, given a representation up to homotopy  $D$, the construction is as follows. Take the dual representation  up to homotopy on $V^*$, which can be viewed as an representation  up to homotopy 
 $\Delta$ on $(V[1])^*$.  
Notice that $\Omega(A,(V[1])^*)$ consists exactly of the functions  on $A[1]\times_M V[1]$ that are fiber-wise linear {w.r.t. the fibers of the    projection $A[1]\times_M V[1]\to A[1]$.} Together with $C(A[1])$, {they}{} generate all the functions. Hence
we can extend $\Delta$ to a fiberwise linear homological vector field on $A[1]\times_M V[1]$ which projects to $Q_{A[1]}$.}
 the homological vector field. In what follows we will construct the homological vector field is a slightly different way, following \cite[\S 6]{RajMarco}, which makes more transparent the role of the components of $D$, {at the expense of being}{} less geometric.
 
{Even though we will not use this fact in the sequel, {let us point out}{} that the statement of Lemma \ref{lem:inftyrep} remains true even without the degree assumption on $V$, by the reasoning explained just before Prop. \ref{prop:inftyrep} below.}

\begin{proof}[Proof of Lemma \ref{lem:inftyrep}]
We start with a general remark (see for instance \cite[Lemma 1.6]{ZZL} for more details).
Given a graded vector bundle $W\to M$, there is an isomorphism of graded Lie algebras
\begin{equation}\label{eq:isocdo}
\Gamma(\CDO(W))\cong {{\vX}_{lin}(W)}{}
\end{equation}
between the covariant differential operators and the   vector fields on $W$ {that preserve}{} the fiberwise linear functions. The isomorphism is given by $B\mapsto Y_B$, where the vector field $Y_B$ is determined by $[Y_B,\iota_w]=\iota_{Bw}$ for all $w\in \Gamma(W)$ and $\iota_w$ denotes the corresponding vertical, constant vector field on $W$. (One checks easily that the  brackets correspond using the Jacobi identity.)

Fix a  representation  up to homotopy  $D$. The sum of its components $\partial+ \nabla+ \omega_2+\dots$ is a degree $1$ element of $\Omega(A,\CDO(V))$. Applying the   isomorphism {\eqref{eq:isocdo}} to $\Gamma(\CDO(V))=
\Gamma(\CDO(V[1]))$, we view the sum as an element $$X\in \Gamma(S(A[1]^*)\otimesM\vX_{lin}(V[1]).$$ Notice that $\nabla_a\in \CDO_0(V[1])$ has symbol $\rho(a)$ for all $a\in \Gamma(A)$, since $\nabla$ is an $A$-connection. Hence $Q_{A[1]}|_{\cN}+X$ is tangent to $\cN$, by\footnote{This is the only place in the present proof where it is used that  $V[1]$ is an N-manifold.} the proof of Lemma \ref{lem:bijphi}. 
This vector field on $\cN$  is related to $D$ by
\begin{equation}\label{eq:relDQ}
\iota_{D\eta}=[Q_{A[1]}+X,\iota_{\eta}]
\end{equation}
for all $\eta \in \Omega(A,V)\cong \Gamma(S(A[1]^*)\otimesM {\Gamma(V[1])}{}$, as a consequence of the   isomorphism \eqref{eq:isocdo} and eq. \eqref{eq:Deta}.
The vector field $Q_{A[1]}+X$ is {a}{} fiberwise linear vector field, projects to $Q_{A[1]}$, and it is homological: the   relation {\eqref{eq:relDQ}} and the Jacobi identity imply $$0=\iota_{D^2\eta}=\frac{1}{2}[[Q_{A[1]}+X,Q_{A[1]}+X],\iota_{\eta}],$$
which forces $[Q_{A[1]}+X,Q_{A[1]}+X]$ to vanish since it is a fiberwise linear vector field.
Reversing the above steps proves that this correspondence is indeed a bijection.
\end{proof}

We can now apply Thm. \ref{charvf}:  
\begin{prop}\label{prop:inftyrep}
{Let $A$ be a Lie algebroid and $V$ a  graded vector bundle over the same base.} 
There is a bijection between   
\begin{itemize}
\item representations up to homotopy  of $A$ on $V$
\item  $L_{\infty}$-actions of $A$  on $V[1]$ that take values in  $\vX_{lin}(V[1])$, the fiberwise linear vector fields on $V[1]$.
\end{itemize}
\end{prop}
\begin{proof}
{If $V$ is concentrated in degrees $\le 0$, then $V[1]$ is an N-manifold, and thanks to Lemma \ref{lem:inftyrep} we 
can  apply immediately Thm. \ref{charvf}. For the general case we argue as follows. If an operator $D$ defines a representation  up to homotopy of $A$ on $V$, then the same operator $D$  defines an representation  up to homotopy  on $V[k]$  for all $k\in \ZZ$.
 Since $V$ is assumed to be of finite rank, $V_i=\{0\}$ for $i$ larger than some integer $k$, and then $V[k]$ is concentrated in degrees $\le 0$ and $(V[k])[1]$ is an N-manifold. Hence we can apply  Lemma \ref{lem:inftyrep} and Thm. \ref{charvf}  to the N-manifold $V[k]$. We conclude {by}{} noticing that the degree shift induces an isomorphism of graded Lie algebras  $\vX_{lin}(V[1])\cong \vX_{lin}((V[k])[1])$.}
\end{proof}

We spell out the resulting $L_{\infty}$-action in classical terms, extending \cite[Prop. 6.7]{RajMarco}. 
\begin{cor}\label{cor:inftyrepmor}
Consider a representation up to homotopy  of $A$ on $V$ with components  $\partial, \nabla=:\omega_1, \omega_2,\dots$. It induces an $L_{\infty}$-morphism
$$F\colon \bigl(\Gamma(A),[\ ,\,]_A\bigr) \rightsquigarrow \bigl(\Gamma(\CDO(V)), [\ ,\,], -[\partial,\ ]\bigr),$$
into the DGLA of covariant differential operators  endowed with the graded commutator bracket \eqref{eq:grcommut}, whose components for all $k\ge 1$ are
$$F_k= (-1)^{\frac{k(k-1)}{2}} \omega_k.$$ 
 \end{cor}
\begin{proof} By Thm. \ref{charvf}, the $L_{\infty}$-action of Prop. \ref{prop:inftyrep} is {obtained by}{} applying the d\'ecalage isomorphism \eqref{deca} to the $L_{\infty}[1]$-morphism
 $\phi_X\colon \Gamma(S(A[1]))\to \vX_{lin}(V[1])[1]$. Since $\Gamma(A)$ is concentrated in degree zero, the d\'ecalage isomorphism does not introduce any new signs. {Furthermore,}{}
 $\phi_X$ is defined in Lemma \ref{lem:bijeasy}, using the  vector field $X$ appearing in the proof of Lemma \ref{lem:inftyrep},
 and reads
\begin{equation*}
(\phi_X)_k(a_1,\dots,a_k)  :=\bigl[[\,\cdots [X, \iota_{a_1}],\dots],\iota_{a_k}\bigr]|_{V[1]}=(-1)^{\frac{k(k-1)}{2}}\bigl[\iota_{a_k},[\,\cdots [ \iota_{a_1},X,],\dots]\bigr]|_{V[1]}
\end{equation*}
for all $a_i \in \Gamma(A[1])$. Recalling the isomorphism \eqref{eq:isocdo},
the evaluation on a section $\eta\in \Gamma(V[1])$ is obtained by taking the Lie bracket of the above expression with $\iota_{\eta}$, noticing that $Q_{A[1]}$ and  $\iota_{\eta}$ commute (viewed as vector fields on the product $A[1]\times V[1]$) and relating $Q_{A[1]}+X$ to $D$ via \eqref{eq:relDQ}. {We conclude by  applying  
Prop. \ref{prop:liacQ}}.
\end{proof}

 \subsection{Representations  up to homotopy arising as linearizations}
Consider a  a graded vector bundle $p\colon \cV\to M$  with fibers concentrated in negative degrees  and denote by $\vX_M(\cV)$ the vector fields that are tangent {to the zero section} $M$. 

There is a well-defined linearization map: 
$${\vX_M(\cV)}\to \vX_{lin}(\cV)$$
that associates to a vector field $Y$ its linearization $Y_{lin}$,  defined as the vector field on $\cV$ that acts on $p^*f$ as $p^*(Y(f))$ for $f\in C^{\infty}(M)$, and on a 
fiberwise linear function $\xi\in \Gamma(\cV^*)$ as
the linear component of $Y(\xi)$.  Note that for  this  to make sense, it is necessary that $Y\in \vX_M(\cV)$.

The following proposition states that certain $L_{\infty}$-actions can be linearized, giving rise to a representation  up to homotopy.

\begin{prop}\label{cor:gap}  
Let $F\colon \Gamma(A) \rightsquigarrow \vX(\cV)$
be an $L_{\infty}$-action of a Lie algebroid $A\to M$ on a graded vector bundle $\cV\to M$ concentrated in negative degrees. Assume that $F$ takes values in {$\vX_M(\cV)$}, 
then the composition
 \begin{equation}\label{eq:xlin}
\Gamma(A) \rightsquigarrow  {\vX_M(\cV)}\to \vX_{lin}(\cV)
\end{equation}
is also an $L_{\infty}$-action of $A$ on $\cV$. By Prop. \ref{prop:inftyrep} it  corresponds to a representation  up to homotopy of $A$ on $\cV[-1]$.
\end{prop}

\begin{proof}[Proof of the Prop. \ref{cor:gap}] 
Applying Thm. \ref{charvf} to the $L_{\infty}$-action $F$ delivers a homological vector field $Q$ on $\cN:=A[1]\times_M \cV$ that projects to $Q_{A[1]}$.
It is of the form $Q=Q_{A[1]}|_{\cN}+X$, where $X\in \Gamma(S(A[1]^*))\otimes_{C^{\infty}(M)}\vX_M(\cV)$ is as in Lemma \ref{lem:bijphi}. {Notice that here we used  our assumption on $F$.}
Split $X$ as $X=X_{lin}+X_{rest}$, where $X_{lin}$ consists of the components of $X$ that are fiber-wise linear on $\cV$, {\emph{i.e.} that preserve the fiber-wise linear functions}. Notice  
that $X_{rest}$
consists of vector fields that are tangent to the fibers of  
$p\colon \cV\to M$ ({by degree reasons}) and vanish at least quadratically on $M$.
{Then  $$\widehat{Q}:=Q_{A[1]}|_{\cN}+X_{lin}$$ is tangent to $\cN$. To see this apply Lemma \ref{lem:bijphi} using the fact that $[X_{rest},
\iota_a]|_{M}=0$ for all $a\in \Gamma(A)$ and that $X_{rest}|_M=0$.}

By construction,  $\widehat{Q}$ preserves  $\Gamma(\pi^*\cV^*)$,  the functions on $\cN$ that are  linear on the fibers of the projection $\cN\to A[1]$. Here $\pi\colon A[1]\to M$ is the vector bundle projection.

{We now argue that $[\widehat{Q},\widehat{Q}]$ vanishes.}
We have $$0=[Q,Q]=[\widehat{Q},\widehat{Q}]+2[\widehat{Q}, X_{rest}]+[X_{rest},X_{rest}],$$ and 
$[\widehat{Q},\widehat{Q}]$ preserves $\Gamma(\pi^*\cV^*)$ while the other two summands on the r.h.s. map it to $\Gamma(S^{\ge 2}\pi^*\cV^*)$, hence we  see that  $[\widehat{Q},\widehat{Q}]$ annihilates $\Gamma(\pi^*\cV^*)$. Notice that $\widehat{Q}$ projects to $Q_{A[1]}$ {under $\cN\to A[1]$} since 
$X_{rest}$ is vertical. So { $[\widehat{Q},\widehat{Q}]$ projects to $[Q_{A[1]},Q_{A[1]}]=0$, and}
 hence $[\widehat{Q},\widehat{Q}]$ annihilates $C(A[1])$ too.  We conclude that $\widehat{Q}$ is a homological vector field.  Therefore we can apply again Thm. \ref{charvf}, which delivers 
an $L_{\infty}$-action, namely the one given in \eqref{eq:xlin}.
\end{proof}

 %%%%%%% 
\section{Representations up to homotopy and  cocycles}

In this section we show that representations up to homotopy of a Lie algebroid, together with a cocycle, correspond to certain to $L_{\infty}$-actions. We do so in \S \ref{sec:cocyclesinty}, making essential use of   
 Thm. \ref{charvf}. As a warm-up, in \S \ref{sec:cocycles} we consider the case of ordinary representation and provide a proof by direct computation.

Before explaining this, we give a simple example of this correspondence.

\begin{ex}[{Lie algebra cocycles}]
Fix a Lie algebra $\g$ and $n\ge 2$.
There is a bijection between $n$-cocycles for $\g$ and $L_{\infty}$-algebra structure on $\RR[n-1]\oplus \g$ such that the projection to $\g$ is a (strict) morphism\footnote{This is equivalent to saying that the binary bracket of the $L_{\infty}$-algebra restricted to $\{0\}\oplus \g$ is just the Lie bracket of $\g$, and that the differential of the $L_{\infty}$-algebra vanishes.}.  Given a cocycle  $\eta$, the corresponding $L_{\infty}$-algebra structure has just two multibrackets, namely the Lie bracket of $\g$, and the $n$-ary bracket given by $\eta$. See \cite[Thm. 55]{baezcrans}, which is phrased for arbitrary representations of $\g$. This bijection is an example of the correspondence mentioned above.

Notice that there is also a bijection  between $n$-cocycles for $\g$ and $L_{\infty}$-morphisms from $\g$ to\footnote{Here $\RR[n-1]$ is endowed with the zero $L_{\infty}$-algebra structure (the only one it admits).}
 $\RR[n-1]$. Given an $n$-cocycle $\eta\in \wedge^n \g^*$, the corresponding $L_{\infty}$-morphism  $F$ has only one component, namely 
$F_n=\eta\colon \wedge^n\g\to \RR$. We learnt this from Christopher L. Rogers, and its proof is immediate using eq. \eqref{eq:Aloo}.  We will need an analogue of this bijection in order to prove our correspondence in \S \ref{sec:cocyclesinty}. 
\end{ex}

\subsection{Lie algebroid cocycles with values in representations}\label{sec:cocycles}

Let $A\to M$ be a Lie algebroid {and}{} $V\to M$  a representation of $A$, \emph{i.e.} a 
flat $A$-connection on a   vector bundle $V$. {We remark that this} is the same as a 
Lie algebroid morphism $ A\to \CDO(V)$ covering $\id_M$, where $\CDO(V)$ denotes the Lie algebroid whose sections are covariant differential operators on $V$.
{Further,}{} there is a canonical isomorphism\footnote{Under the identification 
$\vX_0(V[1])=\Gamma(\CDO(V^*))$ (given by
the action of degree zero vector fields of $V[1]$ on linear functions), this is just the isomorphism $\Gamma(\CDO(V^*))\cong \Gamma(\CDO(V))$ obtained dualizing covariant differential operators, see 
 \cite[Lemma 1.6]{ZZL}.}
 of Lie algebras
\begin{equation}\label{eq:cdos}
\vX_0(V[1])\cong \Gamma(\CDO(V))
,\ X_0\mapsto [X_0,\ ].
\end{equation}
More explicitly, to $X_0$ we assign the covariant differential operator  $\Gamma(V)\to \Gamma(V)$ which, under the identification $\Gamma(V)= \vX_{-1}(V[1])$, maps a degree $-1$ vector field $Y$ to  $[X_0,Y]$.

Hence {a}{} representation of $A$ can be viewed as a map $\phi \colon \Gamma(A)\to \vX_0(V[1])$ satisfying certain properties (namely, the map $\phi$ is an $L_{\infty}$-action of $A$ on $V[1]$). The same {holds}{}  using $V[k]$ instead of $V[1]$, for any non-zero integer $k$.

 \begin{prop}\label{lem:n1}
Let $A\to M$ be a Lie algebroid and $V\to M$  be a representation of $A$. 
Let $\eta\in \Gamma(\wedge^n A^*\otimes V)$ be a Lie algebroid cocycle with values in the representation $V$, for $n\ge 2$. Then $\eta$ induces an $L_{\infty}$-action $F$ of $A$ on $V[n-1]$,
whose only two components are
\begin{align*}
F_1=\phi&\colon \Gamma(A)\to \vX_0(V[n-1])\\
F_n=\eta&\colon \Gamma(\wedge^nA)\to\vX_{1-n}(V[n-1])=\Gamma(V).
\end{align*}
\end{prop}
\begin{proof} We need to check that Def. \ref{def:liac} is satisfied. Condition $i)$ is satisfied because $F$ and $\eta$ are $C^{\infty}(M)$-linear, $ii)$ because $A\to \CDO(V)$, being a Lie algebroid morphism, preserves the anchor maps, and iii) because $F_0$ is trivial.
 
We check that $F$ is a (non-curved) $L_{\infty}$-morphism from $\Gamma(A)$ to $\vX(V[n-1])$, using eq. \eqref{eq:Aloo}.  Since $F$ has only two non-trivial components, and  $d_{Q_{\cM}}=0$, eq. \eqref{eq:Aloo}  boils down to $F_1[a_1,a_2]=[F_1(a_1),F_1(a_2)]$ and 
\begin{multline*} 
\sum_{1\le i<j\le n+1} (-1)^{i+j+1}F_{n}\bigl([a_i,a_j],a_1,\dots,\widehat{a_i},\dots,
\widehat{a_j},\dots,a_{n+1}\bigr)\\=\sum_{i=1}^{n+1}(-1)^{i+1}[F_1(a_i),F_n(a_1,\dots,\widehat{a_i},\dots,a_{n+1})].
\end{multline*}
The first equation is satisfied since $A\to \CDO(V)$ is a Lie algebroid morphism. The second is equivalent to $\eta$ being a cocycle, {as can be seen using eq. \eqref{eq:cdos}}.
\end{proof}

The correspondence given by Prop. \ref{lem:n1} is a bijection:
\begin{cor}\label{cor:f1fn}
Let $A\to M$ be a Lie algebroid, $V\to M$ a vector bundle, fix $n\ge 2$. There is a bijection between
\begin{itemize}
\item pairs $(\phi,\eta)$ where $\phi$ is a representation of $A$ on $V$ and $\eta$
is an $n$-cocycle for this representation,
\item $L_{\infty}$-actions of $A$ on $V[n-1]$ for which all components are trivial except the first and $n$-th component.
\end{itemize}
\end{cor} 
\begin{proof}
The $L_{\infty}$-action is given by {Prop. \ref{lem:n1}}. 
Vice versa, given an $L_{\infty}$-action $F$, notice that $F_1$ takes values in $\vX_0(V[n-1])\cong \Gamma(\CDO(V))$ and $F_n$ takes values in $\vX_{1-n}(V[n-1])=\Gamma(V)$.
The proof of {Prop. \ref{lem:n1}}{} shows that
$F_1$ is a Lie algebroid morphism over $\id_M$, \emph{i.e.} a representation, and that $F_n$ is an $n$-cocycle.
\end{proof}

\begin{ex}[{Closed differential forms}]
{We have the following statement, obtained either by a simple computation or by 
applying Prop. \ref{lem:n1} and Cor. \ref{cor:f1fn} to the Lie algebroid $TM$ with the standard representation on $\RR\times M$: let $\eta\in \Omega^n(M)$. Then
 $\eta$ is closed if{f} $F_1=\id\colon \vX(M)\to \vX(M)$ and $ F_n=\eta\colon \Gamma(\wedge^n TM)\to C^{\infty}(M)$ are the components of an $L_{\infty}$-morphism from $\vX(M)$ to the graded Lie algebra $\vX(M)\oplus C^{\infty}(M)[n-1]$ with bracket $[(X,f),(Y,g)]=([X,Y],X(g)-Y(f))$.
}
\end{ex}

 \subsection{Lie algebroid cocycles with values in a representation  up to homotopy }\label{sec:cocyclesinty} 
 
 \label{sec:inftyrep}
  Now let $V$ be a representation  up to homotopy of $A$ (in particular, $V$ is a graded vector bundle {of}{} {finite rank}).   {Recall from \S \ref{infrep} that this means we are given}{} an operator $D$ on $\Omega(A,V):=\Gamma(\wedge A^*\otimes V)$ {that can be}{} encoded by  {either}{} a homological vector field on $A[1]\times_M V[1]$, or an $L_{\infty}$-action of {$A$}{} on $V[1]$ (by Lemma \ref{lem:inftyrep} and Prop. \ref{prop:inftyrep}, respectively). In this subsection we enhance this  by including cocycles into the picture.

Fix $n\ge 2$. {For the sake of simplicity, we assume that $V$ is concentrated in degrees $\le n-2$}. Let $\eta$ be an \emph{$n$-cocycle} on $A$ with values in the representation up to homotopy, \emph{i.e.} an   element of $\Omega(A,V)$ of   degree $n$ with $D\eta=0$.
The same   {operator}{} $D$   makes $V[n-1]$ {as well}{} into an  
representation up to homotopy of $A$, {for which} $\eta$ becomes an element of degree $1$, since $\Gamma(\wedge A^*\otimes V)_n=\Gamma(\wedge A^*\otimes V[n-1])_1$. {Denote by $Q_D$ the corresponding  homological vector field on the N-manifold $A[1]\times_M V[n-1]$, as in Lemma \ref{lem:inftyrep}.}
 
  We interpret $\eta$ as a vector field $\iota_{\eta}$ on $A[1]\times_M V[n-1]$, using
$$\eta\in \Gamma(\wedge A^*\otimes V[n-1])_1=\vX_{vert, const}(A[1]\times_M V[n-1])_1 \ni \iota_{\eta},$$ where the latter space denotes the space of degree $1$ vector fields  which are vertical (\emph{i.e.} tangent to the second factor) and  constant on each fiber. The equation  \eqref{eq:relDQ} immediately implies the following.
 \begin{lemma}\label{lem:qxzero}
$D\eta=0$ if{f} $[Q_D,\iota_{\eta}]=0$.
\end{lemma}

\begin{ex}\label{ex:closedfQ}
{{By}{} applying Lemma \ref{lem:qxzero} {to the standard representation of $TM$ on $\RR\times M$}{}, {we obtain the following well-known statement}. The assignment $\eta\mapsto  \iota_{\eta}$ is a bijection between closed forms $\eta\in \Omega^n(M)$ and degree 1 vector fields $\iota_{\eta}$ on $T[1]M\times \RR[n-1]$ which are vertical and constant along the fibers, and  so that $Q_{dR}+\iota_{\eta}$ is a homological vector field. (This {also}{} follows from Prop. \ref{prop:alphaaction}, by taking $\g[1]=\RR[n-1]$ with the trivial $L_{\infty}$-algebra structure).}
\end{ex}

For Lie algebroid representations, we obtained Prop.  \ref{lem:n1} by a direct computation. In the   setting of representations  up to homotopy it would be hard to obtain the analogue of {Prop. \ref{lem:n1}}{} directly,  but  using Thm. \ref{charvf}
{one can obtain it easily}:
\begin{prop}\label{prop:inftycoc}
Let $V$ be a representation  up to homotopy of $A$, let $n \ge 2$, {and assume that $V$ is concentrated in degrees $\le n-2$}. Every $n$-cocycle $\eta$ gives rise to an $L_{\infty}$-action  of $A$ on $V[n-1]$.
\end{prop}
 \begin{proof}
Lemma \ref{lem:qxzero} states that $[Q_D,\iota_{\eta}]=0$. Since $ \iota_{\eta}$ commutes with itself (this is clear in coordinates), it follows that $$Q_D+ \iota_{\eta}$$ is a homological vector field on $A[1]\times_M V[n-1]$. This vector field
maps to ${Q_{A[1]}}$ under the projection to $A[1]$,
since $\iota_{\eta}$ is vertical.
Applying Thm. \ref{charvf} finishes the proof.
\end{proof}
 
 The following theorem {states} that 
the correspondence given by Prop. \ref{prop:inftycoc} is a bijection, thus extending\footnote{Modulo the restrictions we impose on the degrees of $V$.} the results of 
Prop. \ref{prop:inftyrep} 
and Cor. \ref{cor:f1fn}.

\begin{thm}\label{cor:finftyfn}

Let $A\to M$ be a Lie algebroid, $V\to M$ a graded vector bundle, fix $n\ge 2$,  {and assume that $V$ is concentrated in degrees $\le n-2$}. There is a bijection between
\begin{itemize}
\item pairs $(D,\eta)$ consisting of a  representation up to homotopy  of $A$ on $V$ and   an $n$-cocycle for this representation,
\item $L_{\infty}$-actions of $A$ on $V[n-1]$ by vector fields on $V[n-1]$ which are sums of fiber-wise linear ones  and   ones {which are  vertical and} constant along the fibers.
\end{itemize}
\end{thm}

\begin{proof}
{
A pair $(D,\eta)$ induces a homological vector field $Q_D+\iota_{\eta}$ on $A[1]\times_M V[n-1]$ and therefore an $L_{\infty}$-action of $A$ on $V[n-1]$, as explained in  
Prop. \ref{prop:inftycoc} and its proof. With respect to the projection $A[1]\times_M V[n-1]\to A[1]$, the vector field $Q_D$ is fiber-wise linear and $\iota_{\eta}$ is tangent to the fibers and constant along the fibers. This means that the  $L_{\infty}$-action obtained is as described in the second item.
}

{For the converse statement, we argue in a way similar to Prop. \ref{cor:gap}. Under the bijection given in
Thm. \ref{charvf}, 
an $L_{\infty}$-action of $A$ on $V[n-1]$ as in the second item
corresponds to a homological vector field which is the sum of a fiber-wise linear ones (let us call it $Q$) and  a vertical one constant along the fibers (let us call it $Y$).
Clearly $[Q+Y,Q+Y]=[Q,Q]+2[Q,Y]$. Now $[Q,Q]$ is a fiber-wise linear vector field  and $[Q,Y]$ 
is vertical  constant along the fibers, so $Q+Y$ being  a homological vector field implies that both have to vanish. The equation $[Q,Q]=0$ implies that $Q=Q_D$ for some representation  up to homotopy  $D$ of $A$ on $V$. By Lemma \ref{lem:qxzero}, $[Q,Y]=0$ implies that $Y=\iota_{\eta}$ for some $n$-cocycle $\eta$.
}
\end{proof}

 To conclude this section, we spell out in classical terms the $L_{\infty}$-morphism obtained from Prop. \ref{prop:inftycoc}. As usual we denote by  $\partial,\, \nabla=:\omega_1,\, \omega_2,\dots$ the components of the representation  up to homotopy. The proof of the following corollary is similar to the one of Cor. \ref{cor:inftyrepmor} and will be omitted.

 \begin{cor}\label{cor:inftyrepmorcocycle}
Consider a representation up to homotopy and an $n$-cocycle $\eta$ with components $\eta_k\in \Gamma(\wedge^k A^*\otimes V_{n-k})$  as in Prop. \ref{prop:inftycoc} (notice that necessarily $\eta_0$ and $\eta_1$ vanish).
Then the $L_{\infty}$-action from Prop. \ref{prop:inftycoc} corresponds via Prop. \ref{prop:liacQ}
 to an $L_{\infty}$-morphism
$$F\colon \bigl(\Gamma(A),[\ ,\ ]_A\bigr) \rightsquigarrow \Gamma(V[n-1])\oplus\Gamma(\CDO(V[n-1]))$$
whose components\footnote{Notice that  changing at will the sign of the summand $\eta$ in $F_n$ one still obtains an $L_{\infty}$-morphism, since $\eta$ is a cocycle if{f} $-\eta$ is. 
This explains why it is not an issue that the present formulae do not specialize to those of {Prop. \ref{lem:n1}}{}.}
 for all $k\ge 1$ are
$$ F_k= (-1)^{\frac{k(k-1)}{2}} (\omega_k +\eta_k).$$
Here the r.h.s. is a DGLA with bracket given by the graded commutator  on $\Gamma(\CDO(V[n-1]))$ and  by $[Y,v]={Y(v)}$ for all $v\in \Gamma(V[n-1])$ and $Y\in \Gamma(\CDO(V[n-1]))$, while the differential 
 acts as $v\mapsto -{\partial v}$
and $Y\mapsto -[\partial,Y]$.
 \end{cor}

%%%%%%%%%%%%%%%%%%%%%%%%%%%%

\section{Maurer-Cartan elements}\label{sec:vs}

We first make a general remark. 
Let $M$ be a manifold and $\h$ an $L_{\infty}$-algebra. Since the differential forms 
$\Omega(M)$ are a differential graded commutative algebra, on $\h\otimes \Omega(M)$ there is an induced $L_{\infty}$-algebra structure,
with differential $v\otimes \omega\mapsto [v]\otimes \omega+(-1)^{|v|}v\otimes ({-}d)\omega$ and higher brackets obtained by extending those of $\h$ by the wedge product \cite[\S 4]{GetzlerAnnals}.
 {For later convenience,  we consider $\Omega(M)$  endowed with the negative of the de Rham differential}. 
 In \cite{GetzlerAnnals}  Getzler consider  Maurer-Cartan elements\footnote{They can be regarded as $Q$-manifold morphisms from $(T[1]M,Q_{dR})$ to $(\h[1],Q_{\h[1]})$, as A. Kotov pointed out to us.}
 of $\h\otimes \Omega(M)$; in the special case that $\h$ is a Lie algebra, they correspond to flat connections on the trivial $H$-bundle over $M$, where $H$ is a Lie group integrating $\h$.\\
 
{In this section we focus on 
a special case. Let $M$ be a manifold and let $\g$ be an $L_{\infty}$-algebra   concentrated in degrees $\le 0$. We define $\h$ to be the DGLA 
\begin{equation}\label{eq:h}
\h=\bigl(\vX(\g[1]),  
{-}[Q_{\g[1]}, \ ], [\ ,\ ]\bigr).
\end{equation}
We will  interpret  Maurer-Cartan elements of $\h\otimes \Omega^{\ge 1}(M)$ as 
 a certain $L_{\infty}$-actions. {Notice that $\h\otimes \Omega^{\ge 1}(M)$ is a DGLA, and that, by  denoting  $D$ and $\textbf{[}\cdot,\cdot\textbf{]}$ its differential and graded Lie bracket, the Maurer-Cartan equation reads $D\alpha-\frac{1}{2}\textbf{[}\alpha,\alpha\textbf{]}=0$, where the minus signs comes from the d\'ecalage isomorphisms.}
 It is easily seen that all the results of this section hold replacing\footnote{Therefore replacing $\Omega(M)$ with $\Gamma(\wedge A^*)$.}
 $TM$ with any Lie algebroid $A$ over $M$.\\}

{
First consider the product manifold $T[1]M\times \g[1]$. 
\begin{lemma}\label{prop:productQ}
 Let\footnote{This lemma holds even allowing $\alpha$ to lie in $\vX(\g[1])\otimes\Omega(M)$.}
$\alpha\in  (\vX(\g[1])\otimes \Omega^{\ge 1}(M))_1$, and interpret $\alpha$ as a  
(vertical) vector field on $T[1]M\times \g[1]$. Then $\alpha$ 
is a Maurer-Cartan element {of the DGLA $\vX(\g[1])\otimes \Omega^{\ge 1}(M)$} if{f} $Q_{dR}+\alpha+Q_{\g[1]}$ is a homological vector field. 
\end{lemma}
\begin{proof}
Writing $\alpha$ as a sum of terms $v\otimes \omega$ for 
$v\in \vX(\g[1])$ and $\omega\in \Omega^{\ge 1}(M)$, we see that the Maurer-Cartan equation for $\alpha$ amounts to
$$[Q_{\g[1]},\alpha]+[Q_{dR},\alpha]+\frac{1}{2}[\alpha,\alpha]=0,$$
which in turn is equivalent to the fact that  $Q_{dR}+\alpha+Q_{\g[1]}$ commutes with itself.
\end{proof}
}

{Now let  $\g_M:=\g\times M$ (a trivial vector bundle over $M$), and denote by $Q_{\g_M[1]}$ the vector field on $\g_{M}[1]$ obtained extending trivially $Q_{\g[1]}$.
Notice that it belongs to $\vX^{vert}(\g_M[1])$, the vector fields which are vertical w.r.t. the projection to $M$.}
{ Lemma \ref{prop:productQ} immediately gives:
\begin{lemma}\label{lem:ge1}
There is a bijection between Maurer-Cartan elements of the DGLA $\vX(\g[1])\otimes \Omega^{\ge 1}(M)$ and homological vector fields $Q_{pr}$ on $T[1]M\times \g[1]$ such that
$$(\g_{{M}}[1],Q_{\g_{{M}}[1]})\to (T[1]M\times \g[1], Q) \to (T[1]M, Q_{dR})$$
is a   sequence of $Q$-manifolds. It is given by $\alpha\mapsto Q_{dR}+\alpha+Q_{\g[1]}$.
\end{lemma}
\begin{proof}
Let $Q_{pr}$ be a degree one vector field  on $T[1]M\times \g[1]$ such that the above is a  sequence of N-manifolds with degree one vector fields. Then $Q_{pr}$ is of the form $Q_{dR}+\alpha+Q_{\g[1]}$ for some degree one element  $\alpha\in \vX(\g[1])\otimes \Omega^{\ge 1}(M)$. Now apply   Lemma \ref{prop:productQ}.
\end{proof}
}

{
\begin{prop}\label{prop:alphaaction}
Let $M$ be a manifold and $\g$ a $L_{\infty}$-algebra  concentrated in degrees $\le 0$. Consider the DGLA
$(\vX(\g[1]),  
{-}[Q_{\g[1]}, \ ], [\ ,\ ])$. There is a bijection between  
\begin{itemize}
\item Maurer-Cartan elements
$\alpha\in  \vX(\g[1])\otimes \Omega^{\ge 1}(M)$, and
\item  $L_{\infty}$-actions of $TM$ on $\g_M[1]$ compatible with $Q_{\g_M[1]}$.
\end{itemize}
 \end{prop}
\begin{proof}
Consider the diffeomorphism of N-manifolds
\begin{equation}\label{eq:diffeo}
T[1]M\times \g[1]\;\;\cong \;\;T[1]M\times_M \g_M[1]=:\cN.
\end{equation}
 Let $Q_{pr}$ be a homological vector field on $T[1]M\times \g[1]$ as in Lemma \ref{lem:ge1}.
Applying this diffeomorphism, from $Q_{pr}$ we obtain a homological vector field $Q_{tot}$ on $\cN$. We have a  sequence of $Q$-manifolds
$$(\g_M[1],Q_{\g_M[1]})\to (\cN, Q_{tot}) \to (T[1]M,Q_{dR}),$$
since the diffeomorphism \eqref{eq:diffeo} intertwines the projections onto $T[1]M$  and the embeddings of $\g[1]$. (For the second morphism we also 
 use that $Q_{\g[1]}$ and $\alpha$ are  vector fields on $T[1]M\times \g[1]$ which are vertical w.r.t. the first projection). 
 Hence we can apply Lemma \ref{lem:ge1} and Thm. \ref{charvfQM}, and obtain the desired $L_{\infty}$-action. Inverting the argument one sees that this assignment is a bijection.
\end{proof}
We make the above $L_{\infty}$-action more explicit. Choose coordinates on $M$ and linear coordinates on $\g$, giving  canonical coordinates\footnote{In these coordinates, $Q_{dR}=\xi_i \partial_{x_i}$.} $x_i,\xi_i$ (of degrees $0$ resp. $1$)  on $T[1]M$ and coordinates $y_i,\eta_a$ on $\g_M[1]$ (of degrees $0$ resp. $1$). In these coordinates, the diffeomorphism \eqref{eq:diffeo} is $((x,\xi),\eta)\mapsto ((x,\xi),(x,\eta))$.  The homological vector field $Q_{tot}$ is the push-forward of $Q_{dR}+\alpha+Q_{\g[1]}$, and as such it is a sum of 3 terms.
The first term is the restriction to $\cN$ of $(\xi_i \partial_{x_i},\xi_i \partial_{y_i})$ 
The second term is the image of  $\alpha$ under the identification $\Omega(M)\otimes\vX(\g[1])\cong \Omega(M)\otimes_{C^{\infty}(M)}\vX^{{vert}}(\g_M[1])$. The third term is $Q_{\g_M[1]}$.
Hence we can write $$Q_{tot}=Q_{dR}|_{\cN}+\tilde{\alpha}+Q_{\g_M[1]},$$
where $Q_{dR}|_{\cN}$ denotes the restriction to $\cN$ of the   vector field $(\xi_i \partial_{x_i},0)$ on the product $T[1]M\times \g_M[1]$, and where  $\tilde{\alpha}$ is obtained from $\alpha$ adding the restriction to $\cN$ of $(0,\xi_i\partial_{y_i})$.  By applying Thm. \ref{charvfQM}, we obtain an  $L_{\infty}$-action of $TM$ on  $\g_M[1]$ compatible with $Q_{\g_M[1]}$. Its components $F_i$ ($i\ge 1$) correspond to $\tilde{\alpha}$ under the correspondence of Lemma
\ref{lem:bijeasy}.
}
 
\begin{ex}
{Let $\g$ be a Lie algebra, and consider the DGLA $\h=\vX(\g[1])$ as in eq. \eqref{eq:h}. 
Fix a  Maurer-Cartan element $\alpha$ of the DGLA $\h\otimes \Omega^{\ge 1}(M)$. By Lemma \ref{prop:productQ}, it delivers a homological vector field on $T[1]M\times \g[1]=(TM\oplus \g_M)[1]$, \emph{i.e.} a  Lie algebroid structure on $TM\oplus \g_M$, {whose anchor is the first projection (hence in particular a transitive Lie algebroid), and whose bundle of isotropy Lie algebras is trivial with typical fiber the Lie algebra $\g$.
Notice that $\alpha$, being of degree one, is a sum of components $\alpha_1\in \Omega^1(M)\otimes \End(\g)$ and $\alpha_2\in \Omega^2(M)\otimes  \g$. Further, decomposing the Maurer-Cartan equation for $\alpha$ in terms of the differential form degree actually implies that  $\alpha_1\in \Omega^1(M)\otimes \der(\g)$, where $\der(\g)$ denotes the infinitesimal automorphisms of the Lie algebra $\g$.} This
 is consistent with the description of Lie algebroid extensions at the beginning of \S \ref{subsec:extclas}. 
}
Finally,  
Prop. \ref{prop:alphaaction} delivers an  $L_{\infty}$-action  of $TM$ on  $\g_M[1]$ compatible with $Q_{\g_M[1]}$. 
\end{ex}

  %%%%%%%%%

\section{Transitive Courant algebroids and induced $L_\infty$-actions}\label{sec:CA} 
 
{In the next two sections we shall consider Courant algebroids whose anchor is surjective, and carry out two types of constructions. In this section 
we use the correspondence \cite{Dima} between Courant algebroids and degree 2 symplectic manifolds endowed with a hamiltonian homological vector field  \cite{Dima} in order to apply one implication of Thm. \ref{charvf} and obtain
an $L_{\infty}$-action of $TM$, see \S \ref{subsec:tcaR}. In the special case of exact Courant algebroids worked out in  \S \ref{sec:exactCA}, this recovers the coadjoint representation up to homotopy together with a cocycle.}
\begin{defi}
A \emph{Courant algebroid} is a vector bundle $E$ over $M$ equipped with:
 a symmetric, non-degenerate $C^\infty(M)$-bilinear product: 
$\langle\ ,\, \rangle_E:\Gamma(E)\times \Gamma(E)\to C^\infty(M),$ a vector bundle map $\rho:E\to TM$, called the anchor,
and a $\mathbb{R}$-bilinear bracket $[\ ,\, ]_E:\Gamma(E)\times\Gamma(E)\to\Gamma(E)$ defined on the section of $E$, such that the following conditions are satisfied (see for instance \cite[Def. 4.2]{Dima}):
\begin{align*}
[\alpha,[\beta,\gamma]_E]_E
&=[[\alpha,\beta]_E,\gamma]_E+[\beta,[\alpha,\gamma]_E]_E,\\
\rho[\alpha,\beta]_E&=[\rho(\alpha),\rho(\beta)],\\
[\alpha,f\beta]_E
&=f[\alpha,\beta]_E+\Lie_{\rho(\alpha)}(f)\beta,\\
[\alpha,\alpha]_E
&={\frac{1}{2}}\rho^*d \langle\alpha,\alpha\rangle_E,\\
\Lie_{\rho(\alpha)}\langle\beta,\gamma\rangle_E
&=\langle[\alpha,\beta]_E,\gamma\rangle_E+\langle\alpha,[\beta,\gamma]_E\rangle_E.
\end{align*}
\end{defi}
  
{A Courant algebroid is called \emph{transitive} if the anchor map $ \rho:E\to TM$ is surjective {(more details are given in \S \ref{sec:isotropicsplittings}). An interesting subclass is given by \emph{exact} Courant algebroids, \emph{i.e.} those for which the sequence $T^*\!M\overset{\rho^*}{\to}E^*\cong E\overset{\rho}{\to} TM$ is exact. They are {precisely the transitive Courant algebroids whose rank is $2\dim(M)$}, and they are  all isomorphic to $TM\oplus T^*\!M$ endowed with the $H$-twisted Courant bracket, where $H$ is a closed 3-form.}
 
\subsection{Exact Courant algebroids: {Roytenberg's approach}}\label{sec:exactCA}
{In this subsection we consider a split exact Courant algebroid and show   that it induces a certain $L_{\infty}$-action of $TM$. This
$L_{\infty}$-action consists of a representation up to homotopy together with a cocycle, which were first identified by Sheng-Zhu \cite{CCShengHigherExt}.}

Given a closed 3-form $H$ on $M$, we consider the exact Courant algebroid
$TM\oplus T^*\!M$ with the $H$-twisted Courant bracket. Under Roytenberg's correspondence \cite{Dima}, which we recall {in more generality} in \S \ref{subsec:tcaR}, the associated symplectic $Q$ manifold of degree $2$ is given by:
$$(T^*[2]T[1]M, X_{S_0+H})$$ with the canonical symplectic structure, where  the homological vector field is
the hamiltonian vector field of sum of the de Rham differential of $M$ (viewed as a fiberwise linear function $S_0$)  and the pull-back of the function on $T[1]M$ defined by $H$.

The projection $T^*[2]T[1]M\to T[1]M$ maps the homological vector field
$X_{S_0+H}$ to $Q_{dR}$. Indeed, $X_{S_0}$ is the cotangent lift of $Q_{dR}$, and 
$X_{H}$ is tangent to the fibers since it is the hamiltonian vector field of the pull-back of a function on the base $T[1]M$. {(Alternatively, one can use the coordinate description in eq. \eqref{eq:QCC}.)}

Now recall that for any vector bundle $\pi \colon A \to M$, given a $TM$-connection on $A$, {the corresponding horizontal distribution} 
 induces an identification:
\begin{equation}\label{eq:TA}
TA\cong \Ver \oplus \pi^*TM=\pi^*(A\oplus TM)=A\times_M(A\oplus TM),
\end{equation}
where $\Ver$ denotes the vertical bundle for the projection $\pi:A\to M$.
Hence a connection on $TM$ induces an isomorphism from  $T^*[2]T[1]M$ to
$$\cN:=T[1]M\times_M(T^*[1]M\oplus T^*[2]M),$$ 
which endows $\cN$ with a homological vector field $Q_{SZ}$.
The first projection $\cN\to T[1]M$ is a map of $Q$-manifolds, since under the isomorphism it corresponds to the vector bundle projection 
$T^*[2]T[1]M\to T[1]M$.
Hence by Thm. \ref{charvf} there is an induced $L_{\infty}$-action of $TM$ on
$T^*[1]M\oplus T^*[2]M$. We summarize:

{
\begin{prop}\label{prop:liactionexact}
A closed 3-form $H$ on a manifold $M$, together with a connection on $TM$, induce an $L_{\infty}$-action of $TM$ on
$\cM:=T^*[1]M\oplus T^*[2]M$.\end{prop}
}

\begin{remark}
If one chooses  canonical coordinates $x^i,v^i$ on $T[1]M$, inducing  coordinates $\xi_i,P_i$ (of degrees $1,2$) on the fibers of  $pr\colon T^*[2]T[1]M\to T[1]M$, we obtain \cite[\S 2.3]{CCShengHigherExt} the following description of $X_{S_0+H}$ in canonical coordinates:
 \begin{equation}\label{eq:QCC}
X_{S_0+H}=
P^i\frac{\partial}{\partial{\xi^i}}
+v^i\frac{\partial}{\partial{x^i}} 
-\frac{1}{6}H_{ijk}v^iv^j\frac{\partial}{\partial{\xi^k}}
+\frac{1}{6}\frac{\partial H_{ijk}}{\partial{x^l}}v^iv^jv^k\frac{\partial}{\partial{P^l}}.
\end{equation} 
\end{remark}

\subsection*{Description of the $L_\infty$ action in classical terms}
In order to make the $L_\infty$-action obtained from  Prop. \ref{prop:liactionexact}  more explicit, we shall now spell out in classical terms  the various components of the associated  $L_\infty$-morphism $F\colon\bigl(\Gamma(TM),[\ , \ ]_{TM}\bigr)\rightsquigarrow \bigl(\vX(\cM), -[Q_{\cM},\ ], [\ ,\ ]\bigr)$ as in  Prop. \ref{prop:liacQ}. To do this, it will be convenient to introduce the following DGLA.

\begin{lemma}\label{lem:DGLA-ShengZhu} Denote by  $\WDGLA(T^*\!M\oplus T^*\!M)$ the graded vector space with grading given by:
\begin{equation*}
 \WDGLA_{-2}:=\Gamma(T^*\!M),\quad
\WDGLA_{-1} := \Gamma\bigl( T^*\!M \oplus \Hom(T^*\!M,T^*\!M) \bigr),\quad
 \WDGLA_{0}:= \der(T^*\!M), 
\end{equation*} 
where $\der(T^*\!M)$ denotes the space of covariant differential operators on $T^*\!M$.

 Then there is structure of a DGLA on $\WDGLA(T^*\!M\oplus T^*\!M)$  with differential defined by:
\begin{align*}
\textbf{d}\eta := (\eta,0),  \quad
\textbf{d}(\eta,\varphi):=\varphi,
\end{align*}
for any  $\eta\in \WDGLA_{-2}$ and $(\eta,v) \in \WDGLA_{-1}$,
and with non trivial brackets given as follows:
\begin{align*}
[ D , D'] &:=D\circ D'-D'\circ D,&
[ D, \eta ]  &:= D(\eta), \\
[ D , (\eta,\varphi)] &:= \bigl(D(\eta), D\circ \varphi- \varphi\circ D \bigr) ,&
[ (\eta,\varphi),(\eta',\varphi')] &:= \varphi(\eta')+\varphi'(\eta),
\end{align*}
where $D,D'\in \WDGLA_0,\ \eta\in \WDGLA_{-2}$ and $(\eta,\varphi), (\eta',\varphi') \in \WDGLA_{-1}$.
\end{lemma}

The proof of Lemma \ref{lem:DGLA-ShengZhu}
 follows from routine computations and will be left to the reader. The interpretation of $\WDGLA(T^*\!M\oplus T^*\!M)$ from the graded {geometric} setting shall be made clearer in remark \ref{rem:subdgexac} 1). Now let us pursue the description of the $L_\infty$-action in a purely classical setting.

\begin{prop}\label{prop:exactexpl}
Consider a closed $3$-form  $H\in \Omega^3_{cl}(M)$  defined on a manifold $M$ together with a $TM$-connection $\nabla$ on $TM$. We denote by $c_2\in \Omega^2(M,T^*\!M)$ and by $c_3\in \Omega^3(M,T^*\!M)$ the tensors  respectively given by the following formulas:
\begin{align} 
c_2(u,v)&:=i_{u_\wedge v}H,  \label{eq:cocyclec2}\\
c_3(u,v,w)&:=\oint_{u,v,w} \nabla^*_u (c_2(v,w))-c_2([u,v],w),\label{eq:cocyclec3}
\end{align}
where the integral sign in \eqref{eq:cocyclec3} denotes the cyclic sum, and $\nabla^*$ is the connection dual to $\nabla$.

  Then, with the notations above, the components 
\begin{align*}
F_1&=\nabla^*:\Gamma(TM) \to \WDGLA_{0},\\
F_2&=(c_2,C_{\nabla^*}):\Gamma(\wedge^2 TM) \to \WDGLA_{-1},\\
F_3&=c_3:\Gamma(\wedge^3 TM)\to \WDGLA_{-2},
\end{align*}
define a   $C^\infty(M)$-linear   $L_\infty$-morphism $F:\Gamma(TM) \to \WDGLA(T^*\!M\oplus T^*\!M)$. Here, we denoted by $$C_{\nabla^*}:=[\nabla^*,\nabla^*]-\nabla^*_{[\ ,\ ]}$$ the curvature of $\nabla^*$, and $\WDGLA(T^*\!M\oplus T^*\!M)$ is the DGLA of Lemma \ref{lem:DGLA-ShengZhu}.
\end{prop}
\begin{proof}  First notice that all the components of $F$ are $C^\infty(M)$-linear maps, so $F$ is indeed $C^\infty(M)$-linear. The conditions for $F$ to define an $L_\infty$-morphism are detailed in Remark \ref{rem:explicit-Linfty-morphism}. They
 are either trivially satisfied, or follow from the construction. More precisely, the equation \eqref{eq:explicit1} is trivially checked, and the equation \eqref{eq:explicit2} is equivalent to \eqref{eq:cocyclec3} together with the Bianchi identity.  In order to see that  \eqref{eq:explicit3} holds as well, let us denote by: $$\nabla^*:\Omega^k\bigl(M, \Hom(T^*\!M,T^*\!M)\bigr)\to \Omega^{k+1}\bigl(M, \Hom(T^*\!M,T^*\!M)\bigr)$$ the covariant derivative operator induced by $\nabla^*$. Then the equation \eqref{eq:cocyclec3} reads $c_3=\nabla^* c_2$, and \eqref{eq:explicit3} follows from the well-known fact that the composition $\nabla^*\circ\nabla^*$ is obtained by pairing with the curvature. 
\end{proof}

\begin{remark}[On the coadjoint representation up to homotopy and cocycles]\label{rem:SZ} 
The construction in Prop. \ref{prop:exactexpl} has an interpretation in terms of representations  up to homotopy due to Sheng-Zhu \cite{CCShengHigherExt} that we now briefly recall. In the setting of {Prop. \ref{prop:exactexpl}} we obtain, by the choice of a connection $\nabla$, a  representative of the coadjoint representation up to homotopy  of the Lie algebroid $TM$. This means that we have a representation up to homotopy (as in \S\ref{sec:inftyrep}) of the Lie algebroid $TM$ on the graded vector bundle $T^*\!M\oplus T^*[1]M$, whose components are given by $(\id_{T^*\!M},\nabla^*,C_{\nabla*})$. In that context, the pair $(c_2,c_3)$ obtained from $H\in \Omega^3_{cl}(M)$ by equations \eqref{eq:cocyclec2} and \eqref{eq:cocyclec3} can be  seen as $2$-cocycle with values in this representation up to homotopy  (see \cite[Prop. 4.9]{CCShengHigherExt}). 
\end{remark}

\begin{remark}\label{rem:subdgexac}
{Let us briefly explain how the $L_{\infty}$-morphism of Prop. \ref{prop:exactexpl} can be viewed as an $L_{\infty}$-action, and why this $L_{\infty}$-action
agrees with the one displayed in Prop. \ref{prop:liactionexact}.}
\begin{enumerate}[1)]
\item In the graded setting, $\WDGLA(T^*\!M\oplus T^*\!M)$  can be interpreted as a sub-DGLA of symmetries of $T^*[1]M\oplus T^*[2]M$. More precisely,  there is a $C^\infty(M)$-linear injective morphism of DGLAs: 
$$\WDGLA(T^*\!M\oplus T^*\!M)\hookrightarrow  \bigl(\vX(\cM), -[Q_{\cM}, \ ], [\,\ ,\ ]\bigr)$$
 where $\cM:=T^*[1]M\oplus T^*[2]M,$ with  homological vector field $Q_{\cM}=P_i\partial_{\xi_i}$. Here we choose local coordinates on $\cM$ as in equation \eqref{eq:QCC}, \emph{i.e.} coordinates $x_i$ on $M$, which induce  linear coordinates $\xi_i$ on the fibers of $T^*[1]M$ and $P_i$ on the fibers of $T^*[2]M$. In order to see this, notice that we have canonical isomorphisms $\vX_{-2}(\cM)\cong \WDGLA_{-2}$ and
$\vX_{-1}(\cM)\cong \WDGLA_{-1}$, while
$\vX_{0}(\cM)\cong \Gamma((\wedge^2TM)\otimes T^*\!M)\oplus \der(T^*\!M\oplus T^*\!M)$, as one can see in coordinates. 
 Define 
$$J\colon \WDGLA_{0}=\der(T^*\!M)\to \der(T^*\!M\oplus T^*\!M)\subset \vX_{0}(\cM)$$ 
to be the diagonal embedding. Then $\vX_{-2}(\cM)\oplus \vX_{-1}(\cM)\oplus J(\WDGLA_0)$ is a sub-DGLA of $(\vX(\cM), -[Q_{\cM}, \ ], [\ ,\ ])$, as can be checked in local coordinates. Note that the injection $\WDGLA(T^*M\oplus T^*M)\hookrightarrow (\vX(\cM), -[Q_{\cM}, \ ], [\ ,\ ])$ obtained this way is $C^\infty(M)$-linear. 

{As a consequence, composing the $L_{\infty}$-morphism $F$ of Prop. \ref{prop:exactexpl} with
this injection we obtain  an $L_\infty$-action of $TM$ on $T^*[1]M\oplus T^*[2]M$.}
\item 
{The $L_{\infty}$-action obtained from Prop. {\ref{prop:exactexpl}}  (see item 1) above) agrees with the one displayed in Prop. {\ref{prop:liactionexact}}.}

 We sketch a proof using the 
notation introduced at beginning of this subsection. Notice that $\cN$ is a graded vector bundle over $M$. In \cite[Prop. 4.9]{CCShengHigherExt} the  homological vector field $Q_{SZ}$ on $\cN$ is described in terms of the corresponding\footnote{The correspondence is recalled in Lemma \ref{nqnla}.} Lie-2 algebroid structure on $\cN[-1]=A\oplus E_0\oplus E_{-1}[1]$, where $
A=TM$, $E_0=E_{-1}=T^*\!M$. The Lie-2 algebroid structure is the one obtained by the coadjoint representation up to homotopy of $TM$ and the 2-cocycle $(c_2,c_3)$. It is
spelled out in \cite[Lemma 4.2 and eq. (38)]{CCShengHigherExt}:
the anchor is the one of $A$ (the identity), the first bracket $l_1$ is the identity $E_{-1}\to E_0$, $l_2$ is given by the Lie algebra bracket on $\Gamma(A)$ together with $\nabla^*$ and $c_2$, and $l_3$ is given by $C_{\nabla^*}$ together with $c_3$. Removing the two contributions {given by the anchor and bracket of the Lie algebroid  $A$} and rearranging the remaining terms according to the number of entries in $\Gamma(A)$ they allow, {as described in the proof of Thm. \ref{charvf}}, one obtains the $L_{\infty}$-action of Prop. \ref{prop:liactionexact}. Comparing the formulae one sees that it agrees with the one displayed in Prop. \ref{prop:exactexpl}.
\end{enumerate}
\end{remark} 

{{Notice that the interpretation given in Rem. \ref{rem:SZ} of the construction of Prop. \ref{prop:exactexpl} -- which by Rem. \ref{rem:subdgexac} is essentially the $L_{\infty}$-action displayed in Prop. \ref{prop:liactionexact} -- 
is consistent with our Thm. \ref{cor:finftyfn}.
}
 Indeed, the expression for {the homological vector field $X_{S_0+H}$} in eq. \eqref{eq:QCC} shows that the 
 $L_{\infty}$-action of $TM$ given in Prop. \ref{prop:liactionexact} is by vector fields on $T^*[1]M\oplus T^*[2]M$ which are fiber-wise linear -- corresponding to the coadjoint  
 representation up to homotopy -- and vector fields which are vertical and  constant along the fibers -- corresponding to the 2-cocycle $(c_2,c_3)$.

 \subsection{{Review of} transitive Courant algebroids and isotropic splittings}\label{sec:isotropicsplittings}

{In this subsection we review some well-known facts about transitive Courant algebroids.}

Consider a Courant algebroid $E\to M$ with anchor map $\rho:E\to TM$, pairing $\la\ ,\, \ra_E:E\oplus E\to \RR$ and   bracket $[\ ,\,]_E:\Gamma(E)\times \Gamma(E)\to \Gamma(E)$. 

If we assume that $E$ is \emph{transitive}, meaning that $\rho$ is surjective, then $\ker \rho$ is a \emph{smooth} subbundle of $E$ that is easily checked to be coisotropic $( \ker \rho^\perp \subset \ker \rho)$ for the pairing $\la\ ,\, \ra_E$. Therefore one can consider the following quotients: 
\begin{align*}
\g_M:=&\ker \rho/(\ker \rho)^\perp,\\
A:=&E/(\ker \rho)^\perp.
\end{align*}
 We obtain four exact sequences of vector bundles as follows:
   \begin{equation}\begin{gathered}\begin{tikzcd}[row sep=tiny, column sep=tiny]
            {}   &                  &  \g_M\arrow[start anchor=-40, end anchor=130, hook]{ddr} &              & \\
               &                  &                &              & \\
               &{\ker \rho}
           \arrow[end anchor=205, two heads]{uur}\arrow[hook]{dr}  &                & A\arrow[end anchor=100, two heads]{dddr}& \\
               &                  & E\arrow[start anchor=-5, end anchor=160, two heads]{ddrr}
                                      \arrow[start anchor=50, end anchor=220, two heads]{ur}   &               & \\
               &                                   &               & \\
   \ker\rho^\perp
   \arrow[start anchor=70, end anchor=235, hook]{uuur}
   \arrow[start anchor=5, end anchor=180, hook]{uurr}&                  &               &               &   TM\quad 
   \end{tikzcd}\end{gathered}\label{diag:transitivecourant}\end{equation}
As it turns out \cite{ChenStienonXu}, the   bracket {$[\ ,\,]_E$} descends to a Lie bracket $[\ ,\, ]_A$ on the sections of $A$  turning $A$ into a transitive Lie algebroid over $M$ with bundle of isotropy Lie algebras $(\g_M,[\ ,\, ]_{\g_M})$. The pairing $\la\ ,\, \ra_E$ also descends to a non-degenerate pairing $\la\ ,\, \ra_{\g_M}$ on $\g_M$, and it is easily checked that $\g_M$ becomes this way a bundle of quadratic Lie algebras. 

Finally, notice that if we denote by $\Xi:E \to E^*$ the musical isomorphism  $e\mapsto \la e,\ \ra_E$, the composition $\Xi^{-1}\circ \rho^*:T^*\!M\to E$ induces an identification:  
$$T^*\!M\simeq (\ker \rho)^\perp. $$

Now let us choose an \emph{isotropic splitting} $\sigma :TM\to E$ of the anchor $\rho$, meaning that its image $\sigma(TM)$ is isotropic for $\la\ ,\ \ra_E$. Note that such a splitting always exits. Then there is an induced splitting  $\lambda_\sigma$ of the canonical projection $\ker \rho \twoheadrightarrow \g_M$,  obtained by imposing its image $\lambda_\sigma(\g_M)$ to be orthogonal to $\sigma(TM) $ with respect to $\la\ ,\ \ra_E$  (more details in \cite[Lem. 1.2]{ChenStienonXu}, {see also \cite[\S 3.3]{Hekmati}}). It follows that one can identify  successively $E$, as a vector bundle, to a direct sum as follows:
\begin{equation}
E= \sigma(TM)\oplus \ker(\rho)\cong TM\oplus \g_M\oplus  T^*\!M.\label{eq:dissection}
\end{equation} 
so that the anchor becomes the projection onto $TM$ and the pairing becomes:
 $$\langle (X,s,\xi), (Y,t,\eta)\rangle=\iota_X\eta+\iota_Y\xi+\langle s,t \rangle_{\g_M}.$$
 Note that in  the identification \eqref{eq:dissection} we use a slightly different convention than in \cite{ChenStienonXu}, in order to obtain the above simpler pairing.

 \subsection{Transitive Courant algebroids: Roytenberg's approach}\label{subsec:tcaR}
 
{In this subsection we extend the results of \S \ref{sec:exactCA} to transitive Courant algebroids (but do not present explicit  formulae as was done there).} {In Prop. \ref{prop:CAroy} we show that a transitive Courant algebroid over $M$, upon making certain choices, induces a  $L_{\infty}$-action of $TM$.}

 {There is a bijection between Courant algebroids and degree 2 symplectic  $Q$ manifolds, as we now recall (see Roytenberg \cite[Thm. 4.5]{Dima}). Given   a Courant algebroid $E$ over $M$, let $\cE$ be the vector bundle over $E[1]$ obtained as the pull back of $T^*[2]E[1]$ along the embedding  $E[1]\hookrightarrow (E\oplus E^*)[1]$ given by $\alpha\mapsto (\alpha,\frac{1}{2}\langle\alpha,\ \rangle_E)$:}
\begin{equation}
\begin{tikzcd}\label{diag:four}
\cE \arrow[hook, dashed]{r} \arrow[two heads, dashed]{d}
& T^*[2]E[1] \arrow[two heads]{d}\\
E[1] \arrow[hook]{r}
& (E\oplus E^*)[1].
\end{tikzcd}
\end{equation}
In other words, $\cE=E[1]\times_{E\oplus E^*[1]}T^*[2]E[1]$.
Note that the projection $\cE\to E[1]$ fits $\cE$ into an exact sequence of graded vector bundles:
\begin{equation*}
 T^*[2]M\hookrightarrow \cE\twoheadrightarrow E[1].
\end{equation*}
 The canonical degree $2$ form on $T^*[2]E[1]$ pulls back to a symplectic form on $\cE$ by the injection $\cE \hookrightarrow T^*[2]E[1]$. Furthermore, $\cE$ comes with a degree 3 hamiltonian function $H$. {The degree two symplectic N-manifold $\cE$ together with $H$ encode the Courant algebroid structure on $E$   \cite{Dima}.}\\

{Now let $E$ be a transitive Courant algebroid over $M$. Let $\cE$ be the associated degree 2 symplectic  $Q$ manifolds as above. 
We will prove in  Lemma \ref{lem:transitive-courant-projection-is-Qmorphism} below that the (surjective) composition $\cE\to E[1]\to T[1]M$ is a $Q$-morphism.} 
In order to apply Thm. \ref{charvf}, we need an identification of the form $\cE\cong T[1]M\times_M \cN$.
}
{
Choosing a connection on the vector bundle $E$ allows us to write $T^*[2]E[1]\cong  E[1]\oplus E^*[1]\oplus T^*[2]M$ as in eq. \eqref{eq:TA}. {On its pull-back $\cE$ as in the diagram \eqref{diag:four}}, this induces an identification
$\cE=E[1]\times_{E\oplus E^*[1]}T^*[2]E[1]\cong E[1]\oplus T^*[2]M$. Hence, we obtain:
$$\cE\cong  E[1]\oplus T^*[2]M
\cong 
T[1]M\times_M\bigl((\g_M[1]\oplus T^*[1]M)\oplus T^*[2]M\bigr),$$
where the second isomorphism is given by \eqref{eq:dissection}  upon choosing an isotropic splitting of $E$.}

We can now apply Thm. \ref{charvf} in order to obtain an induced $L_{\infty}$-action of $TM$ on
$(\g_M[1]\oplus T^*[1]M)\oplus T^*[2]M$. {Notice that $\g_M \oplus T^*\!M\cong A^*$ as vector bundles, using the non-degenerate pairing of $\g_M$ and the isotropic splitting.}

\begin{prop}\label{prop:CAroy}
Let $E$ be a  transitive Courant algebroid $E$ over $M$ with anchor $\rho:E\to TM$. Then the choice of an isotropic splitting $\sigma:TM\to E$ of $\rho$, together with  a $TM$-connection on $E$,  
 induce an $L_{\infty}$-action of $TM$ on  
 {$A^*[1]\oplus T^*[2]M$}. 
{Here  $A$ is the transitive Lie algebroid associated to $E$.}
\end{prop}

In the following lemma, we {fill in the gap left above by  proving} that the projection $\cE\to T[1]M$ is indeed a morphism of $Q$-manifolds.
{
\begin{lemma}\label{lem:transitive-courant-projection-is-Qmorphism}
The projection   $\cE\cong E[1]\oplus T^*[2]M\to E[1]\to T[1]M$,
where the first map is the projection and the last map is induced by the anchor, maps the homological vector field $Q=\{H,\ \}$ to the de Rham vector field. 
\end{lemma}
}
\begin{proof}
Our proof consists of an explicit computation in coordinates. Choose coordinates $x_i$ on an open subset $U$ of $M$, 
giving rise to canonical coordinates $(q_i,p_i)$ on $T^*[2]U$.  Choose a frame $e_a$ of $E|_U$ such that $\langle e_a,e_b \rangle_E=:g_{ab}=\text{constant}$. Then, using the pairing to identify $E$ and $E^*$, this frame gives rise to coordinates $\xi^a$ on the fibers of $E[1]$. Altogether we have coordinates $(q_i,p_i,\xi^a)$ on $\cE$ (over $U$) of degrees $0,2,1$ respectively. According to \cite[Proof of Thm. 4.5]{Dima}, the symplectic form $\Omega$ and the degree 3 function $H$ are
\begin{align*}
\Omega&=dp_idq_i+\frac{1}{2}g_{ab}(x)d\xi^ad\xi^b\\
H&=\rho_{a}^i(x)\xi^ap_i - \frac{1}{6}\phi_{abc}(x)\xi^a\xi^b\xi^c,
\end{align*}
 where $\rho(e_a)=\rho_a^i\partial x_ i$ encodes the anchor and $\phi_{abc}:=\langle [e_a,e_b]_E, e_c\rangle_E$ the bracket. Unless specified otherwise, we use the Einstein summation convention, with $i$ ranging from $1$ to $N:=dim(M)$ and $a$ ranging from $1$ to $rank(E)=2N+rank(\g_M)$.

Under the assumption that the $\rho_a^i$ are constant, we compute that:
$$Q=\{H,\ \}=\rho_a^ig^{-1}_{ab}p_i\frac{\partial}{\partial \xi^b}+\rho_a^i\xi^a\frac{\partial}{\partial q_i}
+\frac{1}{6}(\phi_{abc})_i\xi^a\xi^b\xi^c\frac{\partial}{\partial p_i}+ C \phi_{abc}\xi^a\xi^b  g^{-1}_{cd} \frac{\partial}{\partial \xi^d}
$$
where $g^{-1}$ denotes the inverse matrix  to $g$,
$(\phi_{abc})_i$ denotes the partial derivative of $\phi_{abc}$ w.r.t. $q_i$ and $C$ is a constant.

Now we refine the choice of the frame $e_a$ of $E|_U$. As described in \S \ref{sec:isotropicsplittings}  we have an isomorphism
$E \cong TM\oplus \g_M\oplus  T^*\!M$ under which the anchor and pairing are as described after eq. \eqref{eq:dissection}. 
Take the frame  of $E|_U$ given by $e_i=\frac{\partial}{\partial x_i}, e_{N+i}=dx_i$
for $i\le N$, 
 and the remaining $e_a$'s to be sections of $\g_M$ for which the pairing of $\g_M$ is constant (denote by $\star$ the matrix consisting of the pairings of the latter).
Notice that $\rho_a^i=\delta_a^i$ is the Kronecker delta {(in particular constant)}, and in block form $G=\left(\begin{array}{c|c|c}0 & 1 & 0 \\\hline 1 & 0 & 0 \\\hline 0 & 0 & \star\end{array}\right)$.
The projection $\cE\to T[1]M$ in coordinates reads  $(q_i,p_i,\xi^a)\mapsto ( q_i, \{\xi^a\}_{a\le N})$,     so its  derivative sends $Q$ to
$$\sum_{b=1}^N\rho_a^ig^{-1}_{ab}p_i\frac{\partial}{\partial \xi^b}+\rho_a^i\xi^a\frac{\partial}{\partial q_i}
+ C \sum_{d=1}^N\phi_{abc}\xi^a\xi^b  g^{-1}_{cd} \frac{\partial}{\partial \xi^d}.
$$
To conclude the proof, we observe the following:

\begin{enumerate}
\item The first summand vanishes, since for the range of $b$ considered there we have $g^{-1}_{ab}\neq 0$ only for $N<a\le 2N$, and for those values of $a$ we have $\rho_a^i=0$
\item The second summand equals $\sum_{a\le N}\xi^a \partial_{ q_a}$, \emph{i.e.} the de Rham vector field on $T[1]M$
\item The first summand vanishes, as follows. For the range of $d$ considered there we have $g^{-1}_{cd}\neq 0$ only for $N<c\le 2N$.  For those values of $c$, the expression
$\phi_{abc}=\langle [e_a,e_b]_E, e_c\rangle_E$ vanishes since the tangent component of 
$[e_a,e_b]_E$ vanishes (recall that the Lie bracket of coordinate vector fields vanishes).
\end{enumerate}
\end{proof}

\section{Transitive Courant algebroids and homological vector fields }\label{sec:CAhom} 

\addtocontents{toc}{\protect\mbox{}\protect}
 
{As in the previous section,   we consider transitive Courant algebroids.
In \S \ref{sec:related}  
we construct by hand an   $L_{\infty}$-action and use the implication ``$(2)\rightarrow (1)$'' in 
Thm. \ref{charvf} (the opposite implication to the one used in \S \ref{sec:CA})
 to obtain a homological vector field on a certain N-manifold $\cN$, extending a construction that is known for exact Courant algebroids \cite{s:funny}\cite{UribeDGman}. {In \S \ref{sec:r2} we show that $\cN$ is actually an $\RR[2]$-principal bundle and use it to recover the  {characteristic class}  \cite{ChenStienonXu} of the transitive Courant algebroid.}
 In \S
\ref{subsec:compare} we show that this construction is related {by an obvious  projection} to the one carried out for transitive Lie algebroids in \S \ref{sec:algebroids}.  
}

\subsection{Transitive Courant algebroids: {extending} {\v{S}}evera's and Uribe's approach}
\label{sec:related}

Let $E$ be an transitive Courant algebroid over $M$.  {As seen in \S  \ref{sec:isotropicsplittings}, there is an associated bundle of quadratic Lie algebras
$\g_M$. In this subsection we define an $L_{\infty}$-action of $TM$ on a Q-manifold associated to $\g_M$, and use it to construct a homological vector field on  $T[1]M{\oplus} (\RR_M[2]\oplus {\g_M}[1])$, see  Thm. \ref{thm:tottransCA}.}
 
{The following lemma, whose proof we omit, associates a DGLA $\WDGLA(\g_M,\la\ ,\ \ra_{\g_M})$ to a bundle of quadratic Lie algebras.} 
\begin{lemma} \label{lem:W-dgla-struct:transitivecourant}
Let $\g_M\to M$ be a bundle of quadratic Lie algebras with bracket $[\ ,\ ]_{\g_M}$ and non-degenerate pairing $\la\ ,\ \ra_{\g_M}$. We denote by:
\begin{itemize}
\item $\RR_M$ the trivial line bundle over $M$, $\RR_M:=\RR\times M$,
\item  $\der(\g_M)$ the derivations of $\g_M$ as a bundle of quadratic Lie algebras, namely the space of covariant differential  operators $D:\Gamma(\g_M)\to \Gamma(\g_M)$  such that:
\begin{align*}
 D[\eta,\xi]_{\g_M}&=[D\eta,\xi]_{\g_M}+[\eta,D\xi]_{\g_M}, \\
 \Lie_{X_D}\langle \eta,\xi\rangle_{\g_M}&= \langle D\eta, \xi\rangle_{\g_M}+\langle \eta,D\xi\rangle_{\g_M}, 
\end{align*}
where $X_D$ is the symbol of $D$.
\end{itemize}
Then there is a $DGLA$, denoted by $\WDGLA(\g_M,\la\ ,\ \ra_{\g_M})$  with grading given by:
\begin{equation}\label{eq:DGLAtran}
\WDGLA_{-2}=\Gamma(\RR_M),\quad 
\WDGLA_{-1}=\Gamma(\g_M),\quad 
\WDGLA_0:=\der(\g_M),
\end{equation}
and whose differential and bracket are defined by:
\begin{align*}
\textbf{d}f &:= 0, &
[ D , D']&= D\circ D'-D'\circ D,&
[ D , \mu]&= D(\mu),\\
(\textbf{d} \mu)(\nu)&:=-[\mu,\nu]_{\g_M},&
[ \mu,\nu ]&= \langle\mu,\nu\rangle_{\g_M},& [ D, f ] &= \Lie_{X_D}(f), 
\end{align*}
for all $f \in \WDGLA_{-2}$, $\mu,\nu \in \WDGLA_{-1}$ and $D,D'\in \WDGLA_0$, and where $X_D\in \Gamma(TM)$ is the symbol of $D$.
\end{lemma}

\begin{prop}\label{prop:tr-CASE-Loo-map}
Let $E$ be a transitive Courant algebroid over $M$, and $\g_M:=\ker \rho/\ker \rho^\perp$  the associated bundle of quadratic Lie algebras as in \S \ref{sec:isotropicsplittings}.

 Then any isotropic  splitting $\sigma:TM\to E$ of the anchor $\rho$ induces a  $ C^\infty(M) $-linear   
  $L_{\infty}$ morphism:
\begin{equation}\label{eq:loomapCA}
F: \Gamma(TM)\leadsto \WDGLA(\g_M,\la\ ,\, \ra_{\g_M}),
\end{equation}
where $\WDGLA(\g_M,\la\ ,\ \ra_{\g_M})$ is the DGLA of Lemma \ref{lem:W-dgla-struct:transitivecourant}
\end{prop}

\begin{proof}
{
 Similarly to the case of a transitive Lie algebroid, the choice of an {isotropic} splitting of a Courant algebroid induces a $TM$-connection $\nabla$ on the bundle of quadratic Lie algebras $\g_M$, a $\g_M$-valued $2$-form $\omega\in \Omega^2(M)\otimes \Gamma(\g_M)$, together with a $3$-form $H\in \Omega^3(M)$, respectively defined by:
\begin{align*}
 \nabla_a (\mu\boldsymbol{+(\mathrm{ker}\rho)^\perp})&:=[\sigma(a),\mu]_E\boldsymbol{+(\mathrm{ker}\rho)^\perp} ,\\
 \omega(a,b)&:=[\sigma(a),\sigma(b)]_E-\sigma([a,b]_{TM}) \boldsymbol{+(\mathrm{ker}\rho)^\perp},\\
 H(a,b,c)&:=\frac{1}{2}\bigl\la pr_{T^*\!M}[\sigma(a),\sigma(b)]_E, c \bigr\ra,
\end{align*}
where $a,b,c\in\Gamma(TM)$, $\mu\in \Gamma(\ker \rho)$. Here, given $\alpha \in \ker\rho$, we denoted by $\alpha\boldsymbol{+(\mathrm{ker}\rho)^\perp}\in \Gamma(\g_M)$ its equivalence class modulo $(\ker\rho)^\perp$. We also denoted by $pr_{T^*\!M}$ the projection induced by the decomposition \eqref{eq:dissection}. Note that our definition of $H\in \Omega^3(M)$ coincides with that of \cite{ChenStienonXu} despite the constant term $1/2$  appearing above, this is due to our choice of 
convention in \eqref{eq:dissection}.

As a consequence of the axioms of a Courant algebroid, the following compatibility conditions are satisfied:
\begin{align}
\nabla_a[\eta,\xi]_{\g_M}                     
     &=[\nabla_a \eta,\xi]_{\g_M}+[\eta,\nabla_a \xi]_{\g_M},
\label{eq:TCAstruct2}\\
\mathcal{L}_a \la \eta,\xi \rangle_{\g_M}
     &=\la \nabla_a \eta, \xi \ra_{\g_M}+ \la \eta ,\nabla_a \xi \ra_{\g_M},
\label{eq:TCAstruct1}\\
0
     &=\oint_{a,b,c} \nabla_a\, {(\omega(b,c))}-\omega([a,b]_{TM},c),
     \label{eq:TCAstruct3}\\
C_\nabla                                 
     &=\ad^{\g_M}\circ\omega,
\label{eq:TCAstruct4}\\
\d H
     &=\la \omega_\wedge\omega\ra_{\g_M}.
     \label{eq:TCAstruct5}
\end{align}
Here, in \eqref{eq:TCAstruct4} the curvature of $\nabla$ is denoted by $C_{\nabla}:=[\nabla,\nabla]-\nabla_{[\ ,\ ]}$,  in \eqref{eq:TCAstruct5} the de Rham differential is denoted by $\d$, while  in \eqref{eq:TCAstruct5} the term  
$\la \omega_\wedge\omega\ra_{\g_M}\in \Omega^4(M)$ is obtained by pairing $\omega$ with itself by means of $\la\ ,\ \ra_{\g_M}$, more precisely:
\begin{align*}
\la\omega_\wedge\omega\ra_{\g_M}(a_1,a_2,a_3,a_4):=\frac{1}{4}\sum_{\sigma\in S_4} (-1)^\sigma
\left\la \omega(a_{\sigma(1)},a_{\sigma(2)})\,,\,\omega(a_{\sigma(3)},a_{\sigma(4)})\right\ra_{\g_M}
\end{align*}
for any $a_1,\dots,a_4\in \Gamma(TM)$. The above equations are rather well-known (see  \cite{ChenStienonXu, Hekmati}, also \cite{Bressler}) and follow from straightforward computations, we shall refer the reader to \cite[Prop. 2.2]{ChenStienonXu} for more details as to how they can be obtained. As we now explain, they can be interpreted in terms of $L_\infty$-morphisms.

Let $\WDGLA(\g_M,\la\ ,\ \ra_{\g_M})$ be the DGLA of Lemma \ref{lem:W-dgla-struct:transitivecourant}, we define a map  $F:\Gamma(TM)\leadsto \WDGLA(\g_M,\la\ ,\ \ra_{\g_M})$ whose components are given as follows:
\begin{align*}
F_1 =\nabla&:\Gamma(\wedge^1 TM) \to \WDGLA_{0},\\
F_2 =\omega\,&:\Gamma(\wedge^2 TM) \to \WDGLA_{-1},\\
F_3 =H&:\Gamma(\wedge^3 TM)\to \WDGLA_{-2}.
\end{align*}
The map $F_1$ takes values in $\WDGLA_0$ by eq. \eqref{eq:TCAstruct2} and \eqref{eq:TCAstruct1}. Furthermore, the equations \eqref{eq:TCAstruct3},  {\eqref{eq:TCAstruct4}} and \eqref{eq:TCAstruct5} are exactly the conditions
{\eqref{eq:explicit2}, \eqref{eq:explicit1} and \eqref{eq:explicit3}}
 for $F$ to define an $L_\infty$-morphism (see Remark \ref{rem:explicit-Linfty-morphism}  in Appendix \ref{app:li}). Obviously, $F$ is $C^\infty(M)$-linear since each of its components is.
}
\end{proof}

\begin{remark}
{Notice that the  $L_{\infty}$-morphism given in Prop. \ref{prop:tr-CASE-Loo-map} allows us to reconstruct $E$ up to Courant algebroid isomorphism, since  its components allow us to express the bracket on $E$ (see \cite{ChenStienonXu}).}
\end{remark}

{Given a quadratic a Lie algebra $\g$, recall that (a version of) the Lie 2-algebra associated by Roytenberg-Weinstein \cite{rw}
is $\RR[1]\oplus \g$, with the only non-trivial multibrackets being the Lie bracket $[\ ,\ ]_{\g}$ on $\g$ and the trinary bracket ${\frac{1}{2}}\langle [\ ,\ ]_{\g},\  \rangle_{\g}\colon \wedge^3 \g \to \RR$. Shifting by one we obtain an $L_{\infty}[1]$-algebra. We denote by $Q_{RW}$ the corresponding homological vector field. Hence,
$$(\g[1]\oplus \RR[2], Q_{RW})$$ is a Q-manifold.
}

{
As earlier, let $M$ be a manifold, and consider  {a bundle} of quadratic Lie algebras $\g_M$ over $M$. Consider the Q-manifold 
\begin{equation}\label{eq:qrw}
(\cM, Q_{\cM}):=( \g_M[1]\oplus \RR_M[2], Q_{RW}),
\end{equation} 
where  $\RR_M:=\RR\times M$ {and $Q_{RW}$ denotes the vertical homological vector field which on each fiber is given as above.}

\begin{lemma}\label{lem:subdgla} 
{Let $\g_M\to M$ be a bundle of quadratic Lie algebras.}
There is a  $C^\infty(M)$-linear  injective morphism of DGLAs:
$$ J:\WDGLA(\g_M,\la\ ,\ \ra_{\g_M})
\ {\lhook\joinrel\relbar\joinrel\rightarrow}\ 
(\vX(\cM), -[Q_{\cM}, \ ], [\ ,\ ])$$ where $\WDGLA(\g_M,\la\ ,\ \ra_{\g_M})$  is the DGLA defined in Lemma \ref{lem:W-dgla-struct:transitivecourant}. 
\end{lemma} 
}

\begin{proof} 
{ Choosing local coordinates $x_i$ on $M$, linear coordinates $\xi_a$ on {the fibers of $\g_M[1]$} and $r$ on $\RR[2]$ (of degrees $0$, $1$ and $2$ respectively) one sees that the vector fields on $\cM$ in low degrees are given by:
\begin{align*}
\vX_{-2}(\cM)=\{f\partial_r:f\in C^{\infty}(M)\}&= C^{\infty}(M)\otimes \Span\{\partial_r\},\\
\vX_{-1}(\cM)=\{f_a\partial_{\xi_a}+g_a\xi_a\partial_r: f_a,g_a\in C^{\infty}(M)\}&= \Gamma(\g_M)\oplus \Gamma(\g_M^*)\otimes \Span\{\partial_r\},\\
\vX_{0}(\cM)=\{f_i\partial_{x_i}+g_{ab}\xi_a\partial_{\xi_b}+(hr+k_{ab}\xi_a\xi_b)\partial_r\}
&= \CDO(\g_M)\oplus (C^{\infty}(M)\oplus \Gamma(\wedge^2\g_M^*))\otimes \Span\{\partial_r\}.
\end{align*}
Define $J$ to be the injective linear map
which in degree $-2$ is $f\mapsto f\partial_r$,
in degree $0$ is the natural inclusion $\der(\g_M)\hookrightarrow \Gamma(\CDO(\g_M))$, and in degree $-1$ is given by
 $$J_{-1}\colon \Gamma(\g_M)\to \Gamma(\g_M\oplus \g_M^*),\quad v\mapsto v+{\frac{1}{2}}\langle v,\  \rangle_{{\g_M}}\partial_r.$$ 
 The injection  $J$ obtained this way is $C^\infty(M)$-linear.  We claim that $J$ is a morphism of DGLAs, or equivalently that its image is a sub-DGLA and that the induced differential and bracket correspond to the given ones. Notice that 
$Q_{\cM}$ is a vertical vector field {(see Rem. \ref{rem:Qlocaltranscourant} for an expression in coordinates)} and $\vX_{\le -1}(\cM)$ consist of vertical vector fields, so operations involving only these vector fields can be done separately on each fiber of $\cM$.
}

{
The differential $-[Q_{\cM}, \ ]$ annihilates $\vX_{-2}(\cM)$ and a computation shows that it maps $J_{-1}(v)\in \vX_{-1}(\cM)$ to  
$-[v,\ ]_{\g_M}$ for all $v\in \Gamma(\g_M)$. {The differential vanishes on the image of $J_0$, \emph{i.e.}}  $\der(\g_M)\subset \{X\in \Gamma(\CDO(\g_M)): [Q_{\cM}, X]=0\}$, since the elements on the r.h.s. are exactly the elements of $\Gamma(\CDO(\g_M))$ that preserve both $[\ ,\  ]_{\g_M}$ and $\langle [\ ,\ ]_{\g_M},\  \rangle_{\g_M}$.
}

{
We now consider the graded Lie bracket.  On $\der(\g_M)\subset \vX_{0}(\cM)$ it is the commutator, since for $A,B\in \der(\g_M)\subset\vX_{0}(\cM)$ and
$v\in \Gamma(\g_M)\subset\vX_{-1}(\cM)$ we have $$[[A,B],v]=[A,[B,v]]-[B,[A,v]]$$ by the graded Jacobi identity. For the graded Lie bracket of   $A\in \der(\g_M)\subset\vX_{0}(\cM)$ and  $J_{-1}(v)\in \Gamma(\g_M)\subset\vX_{-1}(\cM)$ we have 
$[A,J_{-1}(v)]=J_{-1}(Av)$. To see this, it suffices to compute that 
$[A,\langle v,\  \rangle_{\g_M}\partial_r]=\langle Av,\  \rangle_{\g_M}\partial_r$. 
To do so, use the fact that $A$ is derivation of the pairing
to show that $A^*\langle v,\  \rangle_{\g_M}=\langle Av,\ \rangle_{\g_M}$
(where the dual covariant differential operator $A^*\in \CDO(\g_M^*)$ is defined in \cite[Eq. (2)]{ZZL}), then use   \cite[Lemma 1.6]{ZZL}. Finally, for elements of degree $-1$, we have 
$$[J_{-1}(v),J_{-1}(w)]={\frac{1}{2}}\bigl([v, \langle w,\  \rangle_{\g_M} \partial_r]+
[\langle v,\  \rangle \partial_r, w]\bigr)= \langle v,w \rangle_{\g_M} \partial_r.$$ }
\end{proof}

{We now draw conclusions for transitive Courant algebroids.}
\begin{thm}\label{thm:tottransCA}
Consider a transitive Courant algebroid $E$ over $M$, together with an isotropic splitting $\sigma:TM\to E$ of the anchor $\rho$. Denote by $\g_M:=\ker \rho/(\ker \rho)^\perp$  the associated bundle of quadratic Lie algebras as in \S\ref{sec:isotropicsplittings}.
Let $(\cM, Q_{\cM})$ be as in eq. \eqref{eq:qrw}.

 Then there is an induced homological vector field $Q_{tot}$ defined on 
\begin{equation*}
\cN:=T[1]M\times_M \cM  \cong T[1]M\oplus ( {\g_M}[1]\oplus \RR_M[2])
\end{equation*}
 for which the following is a sequence of $Q$-manifold morphisms:
\begin{equation*}
(\cM,Q_\cM)\to (\cN, Q_{tot}) \to (T[1]M,Q_{dR})
\end{equation*}
\end{thm}  
\begin{proof}
From the properties of $F$ {(see eq. \eqref{eq:loomapCA})} and Lemma \ref{lem:subdgla} 
 it follows that $ J\circ F$ is an $L_{\infty}$-action of $TM$ on $\cM$ {compatible with $Q_{\cM}$ (see Prop. \ref{prop:liacQ})}. We conclude using the implication
``$(1)\rightarrow (2)$'' in
Thm. \ref{charvfQM}. 
\end{proof}
The homological vector field $Q_{tot}$ on the product $T[1]M\oplus ( {\g_M}[1]\oplus \RR[2])$ is the sum of $Q_{RW}$ and  the vector field that naturally corresponds to $ J\circ F$ (a component of which is $Q_{dR}$). The vertical part of $Q_{tot}$ is not constant along the fibers of the projection
$T[1]M\oplus ( {\g_M}[1]\oplus\RR_M[2])\to T[1]M$.

\begin{remark}\label{rem:Qlocaltranscourant}
Let us describe locally the homological vector field $Q_{tot}$
obtained in Thm. \ref{thm:tottransCA}. For this, we choose local coordinates $(x,v,\xi,r)$ on $ T[1]M\oplus \g_M[1]\oplus \RR_M[2]$, where $x^i$ are coordinates on an open neighborhood  $\cU$ of $M$, $v^i$ denote the corresponding coordinates on the fibers of $T[1]M$, the  $\xi^i$ are coordinates on $\g_M[1]$, and $r$ is the degree $2$ canonical coordinate on the fibers of $\RR_M[2]= \RR[2]\times M$. Then locally the negative of $Q_{tot}$ is:
\begin{multline*}\label{eq:Qtot:atiyahcourant}
-Q_{tot}=
\frac{1}{6}\sum_{i,j,k}\phi_{ijk}\xi^i\xi^j\xi^k\frac{\partial}{\partial r}
+
\frac{1}{2}\sum_{i,j,k} c_{ij}^k\xi^i\xi^j\frac{\partial}{\partial{\xi^k}}\\
+
\sum_{i,j,k} v^i\left(\Gamma_{ij}^k \xi^j \frac{\partial}{\partial{ \xi^k}}
 -\frac{\partial}{\partial{x^i}}\right)
+\sum_{i,j,k}\frac{1}{2}\omega_{ij}^k v^i v^j\Bigl(\frac{\partial}{\partial{\xi^k}}
+\sum_l \frac{1}{2}g_{kl}\xi^l\frac{\partial}{\partial r}\Bigr)
+\sum_{i,j,k}\frac{1}{6} H_{ijk}v^iv^jv^k \frac{\partial}{\partial r}.
\end{multline*} 
Here $\phi_{ijk},\ c_{ij}^k,\ \Gamma_{ij}^k,\ \omega_{ij}^k, \ g_{ij},\ H_{ijk} \in C^\infty(\mathcal{U})$ are defined as follows:
\begin{align*}
\phi_{ijk}:=\frac{1}{2}\la [\xi_i,\xi_j]_{\g_M}, \xi_k\ra_{\g_M},\quad 
[\xi_i,\xi_j]_{\g_M}=\sum_{k} c_{ij}^k\xi_k,\quad
 \nabla_{v_i}\xi_j=\sum_{k}\Gamma_{ij}^k \xi_k,\\
 \omega(v_i,v_j)=\sum_{k}\omega_{ij}^k \xi_k,\quad
 {g}_{jl}:=\la \xi_j,\xi_l\ra_{\g_M}\quad
 H_{ijk}:= H(\partial_{x_i},\partial_{x_j},\partial_{x_k}),
\end{align*}
where $v_i,\xi_i$ denote the basis of sections of $TM\oplus \g_M$ dual to  $v^i,\xi^i$.
{Notice that the first two terms on the r.h.s. combine to $-Q_{RW}$, and that the term involving $g_{kl}$ arises because of the embedding given in Lemma \ref{lem:subdgla}.}
\end{remark}

\subsection{$\RR[2]$-principal bundles and transitive Courant algebroids}
\label{sec:r2}

Let $E$ be a transitive Courant algebroid and $A:=E/(\ker \rho)^\perp$ the associated transitive Lie algebroid (see  \S\ref{sec:isotropicsplittings}).

Notice that any isotropic splitting  $\sigma$ of the anchor of $E$ descends to a splitting of the anchor of $A$ (since $\sigma(TM)\cap (\ker \rho)^{\perp}=\{0\}$, as one verifies taking orthogonals and using the isotropicity of $\sigma(TM)$),   leading to an identification $A=TM\oplus \g_M.$
On the other side, we saw in 
Thm. \ref{thm:tottransCA} that $\sigma$ allows to construct a Q-manifold  $(\cN,Q_{tot})$  which is a bundle over   $T[1]M$. Since $\cN$ is of the form  $\cN=T[1]M\oplus(\g_M[1]\oplus \RR_M[2])$
 it is natural to consider $\cN$ as a bundle over $A[1]$ as well,  via the projection:
$\pi^\cN_{A}:\cN=A[1]\oplus \RR_M[2] \to A[1].$

Summarizing the situation, we obtain a commutative diagram as follows:
\begin{equation}\label{diagram:courantransitivealgebroidtransitive}
\begin{gathered}
\SelectTips{cm}{}\xymatrix@C=0pt{ 
*+[c]{\cN=T[1]M\oplus(\g_M[1]\oplus \RR_M[2])} \ar[rr]^<<<<<<<{\pi^\cN_A}\ar[dr]_{\pi^\cN_{TM}}& & *+[c]{ A[1]=T[1]M\oplus \g_M[1]}\ar[dl]^{\pi^A_{TM}}\\
& T[1]M}
\end{gathered}
\end{equation}
All the maps in this diagram are the obvious projections, note that they all come with a splitting (induced only by $\sigma$).   
Furthermore, both $\pi^\cN_{TM}$ and $\pi^A_{TM}$ are morphisms of $Q$ manifolds ({the former by Thm. \ref{thm:tottransCA}}, the latter because it is the anchor of $A$). This situation raises two natural questions:
\begin{itemize}
\item Is $\pi^\cN_{A}$ a morphism of $Q$-manifolds ? 
\item In that case, how are the two $L_\infty$-actions of $TM$  on $\g_M[1]\oplus \RR[2]$ (induced by $\pi^\cN_{TM}$)  and on $\g_M[1]$ (induced by $\pi^A_{TM}$)  related ?
\end{itemize}

We shall answer the second question in the next section \S\ref{subsec:compare}.  The answer to the first question is positive, and indeed more is true:

\begin{cor}\label{cor:r2CA}
 $(\cN = A[1]\times \RR[2],Q_{tot})$ is a  $\RR[2]$-principal bundle  over $(A[1] ,Q_{A[1]})$ in the category of $Q$-manifolds.
\end{cor}
\begin{proof}

The expressions in local coordinates from Remark   \ref{rem:Qlocaltranslie} and Remark \ref{rem:Qlocaltranscourant} show that  
\begin{equation*}\label{eq:Qtot2}
-Q_{tot}=-Q_{A[1]}+
\Bigl(\frac{1}{6}\sum_{i,j,k}\phi_{ijk}\xi^i\xi^j\xi^k 
 +\sum_{ijk}\frac{1}{2}\omega_{ij}^k v^i v^j\Bigl(\sum_l \frac{1}{2}g_{kl}\xi^l\Bigr)
+\sum_{i,j,k}\frac{1}{6} H_{ijk}v^iv^jv^k \Bigr)\frac{\partial}{\partial r}.
\end{equation*}
Hence the projection to $A[1]$ maps the homological vector field $Q_{tot}$   to the homological vector field $Q_{A[1]}$ (a more conceptual proof of this statement is given in   Cor. \ref{cor:QtotQA}). Further the vertical part of $Q_{tot}$ (the one proportional to $\frac{\partial}{\partial r}$) is constant along the fibers $\RR[2]$ (\emph{i.e.} it is independent of $r$).
\end{proof}

{The next two remarks address two consequences of Cor. \ref{cor:r2CA}. The first remark shows that we recover a  3-cocycle on $A$ whose relevance was explained in work of  Chen-Sti\'enon-Xu.} 
\begin{remark}[{On the characteristic class}]
{As a consequence of  Cor. \ref{cor:r2CA}  and Thm. \ref{charvf}, there is an $L_{\infty}$-action of $A $ on $\RR_M[2]$.  The coordinate expression in Cor. \ref{cor:r2CA} shows that the only non-vanishing components of the $L_{\infty}$-action are the unary one (given by the anchor of $A$) and   the trinary one. Therefore, by Cor. \ref{cor:f1fn}, this trinary component is a Lie algebroid cocycle
 $C\in \Gamma(\wedge ^3A^*)$ (for the trivial representation $\RR$). By the coordinate expression in Rem. \ref{rem:Qlocaltranscourant}, we see that:
\begin{align*}
C(\xi_1,\xi_2,\xi_3 )=&\frac{1}{2}\la [\xi_1,\xi_2]_{\g_M}, \xi_3 \ra_{\g_M},
\\
C(\sigma v_1,\sigma v_2,\xi_3)=&\frac{1}{2}\langle \omega(v_1,v_2), \xi_3\rangle_{\g_M},\\
C(\sigma v_1,\xi_2,\xi_3)=&0,\\
C(\sigma v_1,\sigma v_2,\sigma v_3)=& H(v_1,v_2,v_3),
\end{align*}
 where we used the direct sum decomposition $A=\g_M\oplus \sigma(TM)\ni \xi_i+\sigma v_i$ induced by the isotropic splitting $\sigma$. 
 Recall that $\omega$ and $H$ were defined in the proof of Prop. \ref {prop:tr-CASE-Loo-map}.}
 
 {The 3-cocycle $C$ is not new: it is the \emph{standard 3-cocycle} associated with the transitive Courant algebroid $E$ and the isotropic splitting $\sigma$. Further, the cohomology class  of $C$ is the \emph{characteristic class} associated to $E$ \cite{ChenStienonXu}.}
 
 {To explain this, we recall that for any quadratic transitive Lie algebroid\footnote{That is, for any   transitive Lie algebroid  with invariant symmetric pairing on the isotropy Lie algebras. {Actually \cite{ChenStienonXu} considers not only the transitive case but also the regular one.}}
 Chen-Sti\'enon-Xu  single out certain Lie algebroid 3-cocycles which they call \emph{coherent} \cite[Def. 1.11]{ChenStienonXu}. They show that isomorphism classes of transitive Courant algebroids are in bijection with equivalence classes of pairs consisting of a quadratic transitive Lie algebroid and a coherent 3-cocycle on it \cite[Thm. 1.15]{ChenStienonXu}. 
Of course, the Lie algebroid thus associated 
to a transitive Courant algebroid $E$  is just $A$,
 together with a certain coherent 3-cocycle.  The class of this cocycle in the third Lie algebroid cohomology of $A$ is the  \emph{characteristic class} as defined in \cite[\S 1.5,1.6]{ChenStienonXu}.
Now,   our 3-cocycle $C$ above is precisely  the \emph{standard} one as defined in \cite[Eq. (2.18) in \S 2.3]{ChenStienonXu},  and its class {$[C]\in H^3(A)$} is the characteristic class of $E$ \cite[Prop. 2.6]{ChenStienonXu}.   }
\end{remark}

{In the second remark we ask whether  transitive Courant algebroids can be encoded by  certain 
$\RR[2]$-principal bundles over $(A[1] ,Q_{A[1]})$ in the category of $Q$-manifolds, where
$A$ is the associated transitive Lie algebroid.}

\begin{remark}[{On classification and $\RR[2]$-principal bundles}]\label{rem:rnprinc}
{
Given a manifold $M$ and $n\ge 2$, there is a bijection between the  de Rham cohomology group $H^{n+1}(M)$ and isomorphism classes of $\RR[n]$-principal bundles over $(T[1]M,Q_{dR})$ in the category of $Q$-manifolds (see {{\v{S}}evera} \cite[\S 3]{s:funny} and {Uribe} \cite[Lemma 2.6]{UribeDGman}).
The latter are necessarily trivial as principal bundles, and the homological vector field is the sum of $Q_{dR}$ and a vector field which is vertical and constant along the fibers.
Explicitly, a class $[H]\in H^{n+1}(M)$ is mapped\footnote{We recover this assignment, at the level of representatives, from Ex. \ref{ex:closedfQ}.} to the class of  $$( T[1]M\times \RR[n], Q_{dR}+H\frac{\partial}{\partial r})$$ where $r$ is the (degree $n$) standard coordinate on $\RR[n]$. 

Specializing to $n=2$ and using the fact that \emph{exact} Courant algebroids modulo isomorphism are classified by $H^3(M)$ (via the {\v{S}}evera class), we obtain a bijection between isomorphism classes of exact Courant algebroids  and  isomorphism classes of $\RR[2]$-principal bundles over $(T[1]M,Q_{dR})$ in the category of $Q$-manifolds.
}

{The question of how to extend this bijection to \emph{transitive} Courant algebroids is open, to our knowledge.} Note however that
{Cor. \ref{cor:r2CA}} makes a first step in this direction: it  provides a map that assigns
to a transitive Courant algebroid $E\to M$ with an isotropic splittings a 
{$\RR[2]$-principal bundle  over $(A[1] ,Q_{A[1]})$ in the category of $Q$-manifolds. }

{Further, a preliminary computation indicates that the Courant algebroid structure can be recovered from the $\RR[2]$-principal bundle in a way parallel to the case of exact Courant algebroids. {(See \cite[Ex. 5.7]{RengifoJGP} for the exact case,  where Rengifo takes
 derived brackets of $Q$  with invariant vector fields on
  the $\RR[2]$-principal bundle, \emph{i.e.} with sections of the associated Atiyah algebroid)}.

Indeed, one can identify sections of $E=TM\oplus \g_M\oplus T^*M$ with certain invariant  elements of $\vX_{-1}(A[1]\times \RR[2])$ as follows:
$\Gamma(TM)$ identifies naturally with $\vX_{-1}(T[1]M)$, while $\Gamma(T^*\!M)$ identifies naturally with $C_1(T[1]M)\frac{\partial}{\partial r}$, and $ v\in \Gamma(\g_M)$ is identified  
with $v+{\frac{1}{2}}\langle v,\  \rangle \frac{\partial}{\partial r}$ using
the  diagonal embedding $J_{-1}$ encountered in the proof of Lemma \ref{lem:subdgla}.
A preliminary computation indicates that taking the derived bracket $[[Q,X],Y]$ of two such elements of $\vX_{-1}(A[1]\times \RR[2])$ delivers the Courant algebroid bracket (as in \cite[\S 3.3]{Hekmati}), and taking the Lie bracket $[X,Y]$ delivers the pairing between the two corresponding sections of $E$. 
}
\end{remark}

 \subsection{Comparison of the $L_{\infty}$-actions obtained from transitive Courant algebroids
(\S \ref{sec:related}) and transitive Lie algebroids (\S \ref{sec:algebroids}).}\label{subsec:compare}
 
{Let $E$ be a transitive Courant algebroid and $\sigma$ an isotropic splitting of its anchor. We denote by $\g_M$ the associated bundle of quadratic Lie algebras (see  \S\ref{sec:isotropicsplittings}).}
 {Our aim in this subsection is to answer the second question raised just below diagram \eqref{diagram:courantransitivealgebroidtransitive}. In order to do this, recall that
 $\WDGLA(\g_M,\la\ ,\,\ra_{\g_M})$ and $\WDGLA(\g_M)$ denote DGLAs introduced respectively in Lemma \ref{lem:W-dgla-struct:transitivecourant} and before Prop. \ref{prop:liclasstransLA}.
One verifies easily that the obvious map ${\Pi\colon} \WDGLA(\g_M,\la\ ,\,\ra_{\g_M})\to \WDGLA(\g_M)$ is a (strict) morphism of DGLAs.

Next, consider the projection  $$pr\colon (\g_M[1]\oplus \RR_M[2], Q_{RW})\to
 (\g_M[1], Q_{\g_M})$$
between the $Q$-manifold introduced in \eqref{eq:qrw} and the shift by $1$ of the bundle of Lie algebras $\g_M$. It is a map of $Q$-manifolds, since the projection onto the {first}  summand of the $L_{\infty}$-algebra $ \g\oplus \RR$ is a strict morphism. Notice that the vector fields of degree $\leq 0$ on $\g_M[1]\oplus \RR_M[2]$ are all projectable (see for instance the proof of Lemma \ref{lem:subdgla}). Therefore the push-forward of vector fields
$$ pr_*\colon \vX_{\leq 0}( \g_M[1]\oplus \RR_M[2])\to \vX_{\leq 0}(\g_M[1])$$ is a morphism of DGLAs, {where the differentials are $-[Q_{RW},\;]$ and $-[Q_{\g_M[1]},\;]$.}

{The two maps above fit into a commutative diagram of DGLA morphisms} 
\begin{equation}\label{eq:diag4comm}
  \SelectTips{cm}{}
 \xymatrix{
\WDGLA(\g_M,\la\ ,\,\ra_{\g_M}) \ar@{^{(}->}@<-0.5pt>[r]\ar[d]_{{\Pi}}  & \vX_{\leq 0}(\g_M[1]\oplus \RR_M[2])\ar[d]^{pr_*}\\
\WDGLA(\g_M) \ar@{^{(}->}@<-0.5pt>[r]  &   \vX_{\leq 0}(\g_M[1])  
}
\end{equation}
where the upper horizontal arrow is the embedding given in Lemma \ref{lem:subdgla} 
{and the lower horizontal arrow is the natural embedding (see the end of \S \ref{subsec:extclas}).}
 This shows:
\begin{prop}\label{prop: comparingactionstransitivecourantLie} Let $E$ be an transitive Courant algebroid over $M$, consider the transitive Lie algebroid $A:=E/(\ker \rho)^{\perp}$. Choose  an isotropic splitting $\sigma$ of the anchor of $E$.

The $L_{\infty}$-action induced by $A$ {and $\sigma$} via Prop. \ref{prop:sumsurjLA} 
is the composition
$$\Gamma(TM)  \rightsquigarrow \vX_{\leq 0}( \g_M[1]\oplus \RR_M[2]) \overset{pr_*}{\to}\vX_{\leq 0}(\g_M[1]) $$
of the $L_{\infty}$-action $J\circ F$ induced by $E$ (see the proof of Thm. \ref{thm:tottransCA})
 with the map $pr_*$.
\end{prop}
\begin{proof}
 {The expression for the $L_{\infty}$-morphism 
$\Gamma(TM)\leadsto \WDGLA(\g_M,\la\ ,\, \ra_{\g_M})$ in the proof of Prop. \ref{prop:tr-CASE-Loo-map} and for the $L_{\infty}$-morphism $\Gamma(TM)\leadsto \WDGLA(\g_M)$
in  Prop. \ref{prop:liclasstransLA} show that their are intertwined by $\Pi$. Finish using the commutativity of   diagram \eqref{eq:diag4comm}.} 
\end{proof}

The correspondence described in Thm. \ref{charvf} immediately implies the following corollary, which can also  be readily seen from the expressions in local coordinates from Remark   \ref{rem:Qlocaltranslie} and Remark \ref{rem:Qlocaltranscourant}.

{\begin{cor}\label{cor:QtotQA}
The projection $$T[1]M\oplus   \g_M[1]\oplus \RR_M[2]\to T[1]M\oplus \g_M[1]=A[1]$$
maps the the homological vector field $Q_{tot}$ of Thm. \ref{thm:tottransCA} to the homological vector field $Q_{A[1]}$ (see eq. \eqref{eq:isogr}).
\end{cor}
}

\begin{remark}[{The Pontryagin class}]
One can obtain an interpretation of the Pontryagin class of \cite{Bressler}\cite{ChenStienonXu} as an obstruction for the problem of  lifting an $L_\infty$-action, as follows.

Given an $L_\infty$-action of $TM$ on $\g_M[1]$ given by  $F:\Gamma(TM)  \rightsquigarrow \WDGLA(\g_M)$, where the latter was introduced before Prop. \ref{prop:liclasstransLA}, one wonders if it can be lifted to an $L_\infty$-action action on $\g_M[1]\oplus \RR_M[2]$ {given by an $L_{\infty}$-morphism}    $\widetilde{F}:\Gamma(TM)  \rightsquigarrow \WDGLA(\g_M,\la\ ,\,\ra_{\g_M})$ (recall that the latter was defined in  Lemma \ref{lem:W-dgla-struct:transitivecourant}) such that the following diagram commutes:
$$\SelectTips{cm}{}
\xymatrix@C=0pt{
         & *+[c]{ \WDGLA(\g_M,\la\ ,\,\ra_{\g_M})}
                 \ar[d]^{{\Pi}}\\
\Gamma(TM)\ar@{~>}[r]^-{F}\ar@{-->}[ur]^<<<<<<<{\widetilde{F}}
         & *+[c]{\WDGLA(\g_M)}.
}
$$
Notice that $F$ has only two components while $\widetilde{F}$ has three components. Furthermore, it is easily seen from the commutativity of the above diagram that the first two components $\widetilde{F}_1,\widetilde{F}_2$ of $\widetilde{F}$ are determined by $F_1,F_2$. In other words, in order to be able to lift $F$  there is a first obvious necessary condition, which is that $F_1=:\nabla$ {takes values in $\WDGLA(\g_M,\la\ ,\,\ra_{\g_M})$, \emph{i.e.} that it is} a derivation of $\la\ ,\, \ra_{\g_M}$ as in equation \eqref{eq:TCAstruct1}. {Whenever it is the case, a straightforward computation shows that {$F_2=:\omega$ satisfies} $\d_{dR}\la \omega_\wedge\omega\ra_{\g_M}=\la\nabla\omega_\wedge\omega\ra_{\g_M}+\la\omega_\wedge\!\nabla\omega\ra_{\g_M}=0$. Since } one is left with choosing $\widetilde{F}_3=:H\in \Omega^3(M)$ such that equation  \eqref{eq:TCAstruct5} is satisfied {(because eq.  \eqref{eq:TCAstruct5} agrees with eq. \eqref{eq:explicit3})}, we {recover} the vanishing of the  Pontryagin class $[\la \omega,\omega\ra_{\g_M}]\in H^4(M)$. 

\end{remark}

%%%%%%%
 
\appendix
\section{Background on $L_{\infty}$-algebras}\label{app:li}

We recall the necessary background on $L_{\infty}$-algebras and their morphisms, along the lines of \cite[\S 2]{RajMarco}. See also \cite{LadaMarkl},\cite[\S 1.1]{FloDiss}.

\subsection{$L_{\infty}$-algebras}\label{subsec:lialg}

Given a graded vector space $V=\oplus_{n\in \ZZ}V_n$,   
 the  graded skew-symmetric algebra over $V$ is defined as the following quotient of  the tensor algebra:
$$\wedge V:=TV/ \langle v \otimes v' +(-1)^{|v||v'|}v' \otimes v \rangle,$$  
where $|v|$ is the degree of a homogeneous element $v \in V$.

 For any homogeneous elements $v_1,\dots,v_j \in V$ and $\tau\in S_j$ (the group of permutations of $j$ elements), the 
sign $\chi(\tau)$ is defined by  
$$v_{\tau_1}\wedge \cdots \wedge v_{\tau_j}=\chi(\tau)v_{1} \wedge \cdots \wedge v_{j}.$$
 
A permutation $\tau\in S_n$ is called an \emph{$(i, n-i)$-unshuffle} if it satisfies   $\tau(1)<\dots<\tau(i)$ and $\tau(i+1)<\dots<\tau(n)$. The set of $(i, n-i)$-unshuffles is denoted by $Sh(i, n-i).$ Following \cite[Def. 2.1]{LadaMarkl} and \cite[Def. 5]{KajSta}, we define:  

\begin{defi}\label{li}
An \emph{$\li$-algebra} is a graded vector space $V$ equipped with a collection of linear maps $[\cdots]_k \colon \wedge ^kV\longrightarrow V$ of degree $2-k$, for $k\ge1$, such that
$$\sum_{i=1}^n (-1)^{i(n-i)}\sum_{\tau\in Sh(i,n-i)}\chi(\tau)\bigl[[v_{\tau(1)}, \dots,v_{\tau(i)}]_i,v_{\tau(i+1)}, \dots, v_{\tau(n)} \bigr]_{n-i+1}=0$$
for  homogeneous elements $v_1, \dots, v_n \in V$. 
\end{defi}  

\begin{ex}
An $\li$-algebra for which only the first multibracket $[\;\cdot\;]_1$  is non-trivial is the same thing as  a cochain complex. 
 \end{ex}
\begin{ex}
An $\li$-algebra for which only $[\cdot,\cdot]_2$ is non-trivial is the same thing as a graded Lie algebra. Recall that 
 a \emph{graded Lie algebra} consists of   a graded vector space $V=\oplus_{i\in \ZZ} V_i$
together with a  bilinear bracket $[\cdot,\cdot] \colon V \times V \rightarrow V$ such that, for all  homogeneous   $a,b,c\in V$:
\begin{itemize}
\item[--] the bracket is degree-preserving, \emph{i.e.} $[V_i,V_j]\subset V_{i+j}$
 \item[--] the bracket is graded skew-symmetric, \emph{i.e.}
 $[a,b]=-(-1)^{|a||b|}[b,a]$
\item[--] 
{the Jacobi identity holds, \emph{i.e.}} the adjoint action $[a,\cdot]$ is a degree $|a|$ derivation of the bracket: $[a,[b,c]]=[[a,b],c]+(-1)^{|a||b|}[b,[a,c]]$ .
\end{itemize}
\end{ex}
\begin{ex}\label{dgla}
An $\li$-algebra for which only $[\;\cdot\;]_1$ and $[\cdot,\cdot]_2$ are non-trivial is the same thing as a \emph{differential graded Lie algebra} (DGLA). Recall that a DGLA $(V,[\cdot,\cdot], \delta)$ is a graded Lie algebra together with a linear map $\delta : V \rightarrow V$ such that
\begin{itemize}
 \item[--] $\delta$ is a degree $1$ derivation of the bracket, \emph{i.e.}
$\delta(V_i)\subset V_{i+1}$ and $\delta[a,b]=[\delta a,b]+ (-1)^{|a|}[a, \delta b]$
\item[--]  $\delta^2=0$.
\end{itemize}
\end{ex}

Although in this paper we focus on $L_{\infty}$-algebras, it is useful to introduce the closely related notion of $L_{\infty}[1]$-algebra.
The graded symmetric algebra over $V$ is the quotient of  the tensor algebra $$SV:=TV/\langle v \otimes v' -(-1)^{|v||v'|}v' \otimes v \rangle .$$ 
For any homogeneous elements $v_1,\dots,v_j \in V$ and $\tau\in S_j$, the sign
$\epsilon(\tau)$ is defined by 
$$v_{\tau_1}\cdots v_{\tau_j}=\epsilon(\tau)v_{1}\cdots v_{j}$$
where the product appearing here is the one in 
  $SV$.

\begin{defi}\label{li1}
An \emph{$\li[1]$-algebra}
is a graded vector space $V$ equipped with a collection of linear maps $\{\cdots\}_k \colon S^kV\longrightarrow V$ of degree $1$, for $k\ge1$, such that 
$$\sum_{i=1}^n \sum_{\tau\in Sh(i,n-i)}\epsilon (\tau)\bigl\{\{v_{\tau(1)}, \dots,v_{\tau(i)}\}_i,v_{\tau(i+1)}, \dots, v_{\tau(n)} \bigr\}_{n-i+1}=0$$
for  homogeneous elements $v_1, \dots, v_n \in V$.
\end{defi}

The relation between the two notions is the following: there is a bijection \cite[Rem. 2.1]{Vor} between \li-algebra structures on a graded vector space $V$ and $\li[1]$-algebra structures on $V[1]$, the graded vector space defined by $(V[1])_i:=V_{i+1}$. The multibrackets are related by the \emph{d\'ecalage isomorphism}
\begin{equation}\label{deca}
 (\wedge^n V)[n] \cong S^n(V[1]),\;\; v_1\cdots v_n \mapsto v_1\cdots v_n\cdot (-1)^{(n-1)|v_{1}|+\dots+2|v_{n-2}|+|v_{n-1}|}.
 \end{equation}

\begin{remark}[The derived bracket construction]\label{rem:vorcon}
We recall briefly Voronov's derived bracket construction, which delivers $\li[1]$-algebras out of simple data {\cite[Thm. 1, Cor. 1]{Vor}}. Consider a quadruple $(L,\mathfrak{a}, P,\Delta)$ where 
  $L$ is a graded Lie algebra, $\mathfrak{a}$ an abelian  Lie subalgebra,  $P \colon L \to \mathfrak{a}$ a projection whose kernel is a Lie subalgebra of $L$, and 
$\Delta \in Ker(P)_1$ an element such that $[\Delta,\Delta]=0$. Then 
 $\mathfrak{a}$ acquires a $L_\infty[1]$-algebra structure, with   multibrackets   $\{\emptyset\}:=P\Delta$ and 
($n\ge 1$)
\begin{align}\label{eq:adp}
\{ a_1,\dots,a_n\}_n&=P[\dots[[\Delta,a_1],a_2],\dots,a_n].
\end{align}  
The corresponding $L_{\infty}$-algebra structure is obtained applying the d\'ecalage isomorphism.
\end{remark}

\begin{remark}
Given a finite dimensional $\ZZ$-graded vector space $V$, there is a bijection between 
$\li[1]$-algebra structures on $V$ and homological vector fields (see Def. \ref{NQm}) on $V$ which vanish at the origin.
The correspondence is given by Voronov's derived bracket construction, and is analogous to the one of Lemma \ref{nqnla} below.
\end{remark}

\subsection{$L_{\infty}$-algebras morphisms}
In this paper we make use of the notion of (curved) $\li$-algebra morphism. To define it, we first consider 
  (curved) $\li[1]$-algebra morphisms  \cite[Def. 6]{KajSta}
\cite[Def. 1.7]{FloDiss}. We only need a special case:

\begin{defi}\label{li1mor}  
Let $V$  be a {$\li[1]$-algebra} with multibrackets   $\{\cdots\}$, and let $(W,\textbf{d},\textbf{\{}\cdot,\cdot\textbf{\}})$ be 
an $L_{\infty}[1]$-algebra whose only non-trivial multibrackets are the unary and binary one.

a) 
An \emph{$\li[1]$-algebra morphism} from $V$ to $W$, denoted $V \rightsquigarrow W$, is a degree $0$ linear map $\phi \colon S^{\ge 1} V\to W$ such that  for all $n \geq 1$
 \begin{equation}\label{li1morph}
 	\begin{split}
 &\sum_{i=1}^n \sum_{\tau\in Sh(i,n-i)}\epsilon (\tau) \phi_{n-i+1}(\{v_{\tau(1)}, \dots,v_{\tau(i)}\}_i,v_{\tau(i+1)}, \dots, v_{\tau(n)}) \\
=&\textbf{d}\phi_n(v_{1}, \dots,v_{n})
+\frac{1}{2}
\sum_{j=1 }^{n-1}  \sum_{\tau\in Sh(j,n-j)} \epsilon (\tau)
\textbf{\{}\phi_{j}(v_{\tau(1)}\;,\; \dots,v_{\tau(j)}), 
\phi_{n-j}(v_{\tau(j+1)}, \dots,v_{\tau(n)}
)\textbf{\}}.		
 	\end{split}
 \end{equation}
for every collection of homogeneous elements $v_1, \dots, v_n \in V$. Here $S^{\ge 1} V:=\oplus_{i\ge 1}S^i V$.

b) a \emph{curved} $\li[1]$-algebra morphism $V \rightsquigarrow W$ consists of a degree $0$ linear map $\phi: SV \to W$ satisfying, for $n \geq 0$, a variation of 
\eqref{li1morph} where the index $j$ on the right side of the equation runs from $0$ to $n$.
\end{defi}

\begin{remark}
If $\phi: SV \to W$ is a degree $0$ linear map, then the zero component $\phi_0 \colon \RR \to W_0$ gives rise to an element $\phi_0(1)\in W_0$, which by abuse of notation we denote by $\phi_0$. The curved variant of eq. \eqref{li1morph} for $n=0$ then reads 
$0=\textbf{d}\phi_0+\frac{1}{2}\textbf{\{}\phi_0,\phi_0\textbf{\}}.$
In other words, if $\phi$ is a curved $\li[1]$-algebra morphism, then $\phi_0$ is a Maurer-Cartan element of $W$.
\end{remark}

The following fact, whose proof follows easily from the above, is implicit in \cite{K}:
\begin{lemma}\label{lemma:curved}
Let $V$ be an $L_{\infty}[1]$-algebra, let $(W,\textbf{\{}\ ,\ \textbf{\}})$  be a
$L_{\infty}[1]$-algebra whose only non-trivial bracket is the binary one,
 and let $\phi \colon SV\to W$ be a degree zero linear map.  Then the following is equivalent:
 \begin{itemize}
\item $\phi$ is a curved $\li[1]$-algebra morphism from $V$ to $W$ 
\item
$\textbf{\{}\phi_0,\phi_0\textbf{\}}=0$ and $\{\phi_i\}_{i\ge 1}$ is an $\li[1]$-algebra morphism from $V$ to 
$(W,\textbf{\{}\phi_0,\ \textbf{\}},\textbf{\{}\ ,\ \textbf{\}})$.
\end{itemize}
\end{lemma}

In this paper we make use of (curved) $\li$-algebra morphisms into a DGLA.
\begin{defi}\label{limor}  
Let $V$  be an $\li$-algebra and $W$ a DGLA.  An \emph{$\li$-algebra morphism} from $V$ to $W$, denoted $V \rightsquigarrow W$, is a linear map $\phi \colon \wedge^{\ge 1} V\to W$ which, 
under the d\'ecalage isomorphism \eqref{deca}, corresponds to a $\li[1]$-algebra morphism from $V[1]$ to $W[1]$ (see Definition \ref{li1mor}).
\end{defi}

Notice that the $k$-th component  of an $\li$-algebra morphism is a map $\wedge^{k} V\to W$ of degree ${1}-k$.
  In equation \eqref{eq:Aloo}
 we display the conditions satisfied by an $L_\infty$-morphism from a  Lie algebra to a  DGLA, using the notation introduced there. For later convenience, we now spell out special cases of this.

\begin{remark}[{Explicit formulae for $L_{\infty}$-morphisms}]\label{rem:explicit-Linfty-morphism}
We shall be particularly interested in  $L_\infty$-morphisms of the form: 
$$F:\g\leadsto \WDGLA,$$ where  $\g$ is a Lie algebra and $\WDGLA=\WDGLA_{0}\oplus \WDGLA_{-1}\oplus \WDGLA_{-2}$ is a $DGLA$  concentrated in degrees $0,-1$ and $-2$, seen as an $L_\infty$-algebra whose higher brackets $[\cdots]_k $ vanish for $k\geq 3$. Explicitly, such an $L_\infty$-algebra morphism  has components:
$$ F_1\colon \g \to \WDGLA_0,\;\;
   F_2\colon \wedge^2\g \to \WDGLA_{-1},\;\;
   F_3\colon \wedge^3\g \to \WDGLA_{-2},
$$
satisfying the following conditions:
\begin{align}
F_1([x,y]_\g)-[ F_1(x), F_1(y)] 
                          &=\textbf{d}F_2(x,y),\label{eq:explicit1}\\
F_2([x,y]_\g,z)- [F_1(x),F_2(y,z)] +c.p
                          &=\textbf{d}F_3(x,y,z),\label{eq:explicit2}
\end{align}
and:       
      \begin{align}
&F_3([x, y]_\g, z, t)-F_3([x, z]_\g, y, t)
+ F_3([x, t]_\g, y, z)\nonumber\\
 +& F_3([y, z]_\g, x, t) 
-F_3([y, t]_\g, x, z)
 +F_3([z, t]_\g, x, y)\nonumber\\
- &[ F_1(x), F_3(y, z, t)]  +[ F_1(y), F_3 (z, t, x)] 
 - [F_1(z), F_3(t,x,y)] +[ F_1(t), F_3(x,y, z) ]\nonumber\\
=-& [ F_2 (x, y), F_2(z, t)]+[ F_2 (x, z), F_2(y, t)]  -[ F_2 (x, t), F_2(y,z)],\label{eq:explicit3}
\end{align}
for any $x,y,z,t\in \g$.  
Indeed, applying eq. \eqref{eq:Aloo} to the case of hand for $n=2,3,4$ delivers the three
  equations above, whereas for $n=1$ and $n\ge 5$,  eq. \eqref{eq:Aloo} is automatically satisfied by degree reasons.
\end{remark}

\section{Background on {graded geometry} and   Lie $n$-algebroids}\label{app:Q}
 
We review some notions of graded geometry (N-manifolds and Q-manifolds) 
along the lines of \cite[\S 1.1]{ZZL}. Then we introduce Lie $n$-algebroids (for $n=1$ one recovers Lie algebroids as defined in the Introduction), and we recall their relation to Q-manifolds. 
 
 \subsection{N-manifolds}\label{nnman}

N-manifold (``N'' stands for non-negative) were
introduced by \v{S}evera  in  \cite{s:funny}. Here we
adopt the definition given by Mehta in  \cite[\S 2]{rajthesis}.

Given a finite dimensional $\ZZ_{<0}$-graded vector space  $V=\oplus_{i< 0}V_i$,  recall that $V^*$ is the $\ZZ_{>0}$-graded vector space defined by $(V^*)_i=(V_{-i})^*$. As earlier, we denote by  $S(V^*)$   the  {graded} symmetric algebra over $V^*$, a graded commutative algebra  concentrated in positive degrees. 
 
\begin{defi}\label{locm}
Let $V=\oplus_{i< 0}V_i$ be a finite dimensional $\ZZ_{<0}$-graded vector space and $n\ge 0$ an integer.\\
 The \emph{local model for an N-manifold} consists  of a pair  as follows:
\begin{itemize}
\item $U\subset \R^n$ an open subset
\item the sheaf (over $U$) of graded commutative algebras given by
$U'\mapsto C^{\infty}(U')\otimes S(V^*)$.
\end{itemize}
\end{defi}
\begin{defi}\label{def:grm}
An \emph{N-manifold} $\cM$ consists of a pair  as follows:
\begin{itemize}
\item a topological space $M$ (called \emph{body})
\item a sheaf $\cO_M$ over $M$ of graded commutative algebras, locally isomorphic to the
above local model (the sheaf of ``functions'').
\end{itemize}
\end{defi}

We use the notation $C(\cM):=\cO_M(M)$ to denote the space of ``functions
on $\cM$''.
By $C_k(\cM)$ we denote the degree $k$ component of $C(\cM)$, for any non-negative $k$. The
{\em degree} of the N-manifold is the largest $i$ such that  $V_{-i}\neq\{0\}$.
Degree zero N-manifolds are just
 ordinary  manifolds.

{We have a natural projection $p\colon \cM\to M$, induced by the inclusion $C^{\infty}(M)=C_0(\cM)\hookrightarrow C(\cM)$. Further we have an embedding\footnote{The embedding is defined, more generally, for any  $\ZZ$-graded manifold.}
 $M\hookrightarrow \cM$,  induced by the quotient map $C(\cM)\to C(\cM)/C_{\neq 0}(\cM)\cong C^{\infty}(M).$
}
 
 \noindent\begin{ex}\label{VBs}
Let $F=\oplus_{i< 0}F_{i} \rightarrow M$ be a   graded vector bundle.
The N-manifold associated to it has body $M$, and $\cO_M$ is given by the sheaf of sections of  $SF^*$. \end{ex}

\begin{remark}\label{ref:batch}
By Batchelor's theorem in the category of $N$-manifolds, any $N$-manifold of degree $n$ is non-canonically isomorphic to a graded vector bundle concentrated in degrees $-1,\dots,-n$ 
(see for instance \cite[\S 2.2]{BonavolontaPoncin}).\end{remark}

\begin{defi} For any integer $k$, 
A \emph{degree $k$ vector field} on $\cM$ is a
degree $k$ derivation of the algebra $C(\cM)$.
\end{defi}
Since $C(\cM)$ is a graded commutative algebra (concentrated in 
%non-negative 
degrees $\ge 0$), the space of vector fields $\chi(\cM)$,
equipped with the graded commutator $[\cdot,\cdot]$,  is a graded
Lie algebra.

\subsection{{Homological vector fields and} Q-manifolds}\label{enq}
The $N$-manifolds appearing in this paper often come equipped with extra structure:
\begin{defi}\label{NQm} \cite{s:funny}
A\footnote{We are abusing language here, calling $Q$-manifold what is often referred to as   $NQ$-manifold.}
 {\em $Q$-manifold} is an N-manifold   equipped
with a \emph{homological vector field}, \emph{i.e.} a degree 1 vector field $Q$ such that $[Q, Q]=0$.
\end{defi}

A well-known example of $Q$-manifolds is given by Lie algebroids \cite{MK2}, as defined in the Introduction. More precisely, we have the following well-known correspondence    (\cite{vaintrob}, see also \cite{yvette}):
\begin{lemma}\label{nq1la}
There is a bijection between $Q$-manifolds of degree $1$ and Lie algebroids.
\end{lemma}

We describe the  correspondence, which will be extended to Lie $n$-algebroids in Lemma \ref{nqnla}. 
There is a bijection between vector bundles and
degree 1 $N$-manifolds. If $A$ is a Lie algebroid, then the homological vector field  is just the   Lie algebroid differential acting on $ \Gamma(\wedge A^*)=C(A[1])$. Conversely, if
$(\cM:=A[1], Q)$ is an $Q$-manifold, then
the Lie algebroid structure on $A$ can be recovered \cite[\S 4.3]{yvette} by the derived bracket
construction: using the
 identification $\chi_{-1}(\cM)= \Gamma(A)$, we define
for all $a,b\in \Gamma(A)$ and $f\in C^{\infty}(M)$:
\begin{equation}\label{eq:br-rho}
[a,b]_A=[[Q,a],b], \;\;\;\;\quad  \rho_A(a)f=[[Q,a],f],
\end{equation}

In coordinates the correspondence is as follows.
Choose coordinates $x_{\alpha}$ on $M$ and a frame of sections $e_i$ of $A$, inducing (degree $1$) coordinates $\xi_i$ on the fibers of $A[1]$. Then
\begin{equation}\label{QA}
Q_A=-\frac{1}{2}\xi^i\xi^j
c_{ij}^k(x)\pd{\xi_k}+\rho_{i}^{\alpha}(x)\xi^i\pd{x_{\alpha}}
\end{equation}
where $[e_i,e_j]_A=c_{ij}^k(x)e_k$ and the anchor applied to $e_i$ is $\rho_{i}^{\alpha}(x)\pd{x_{\alpha}}$.

\begin{remark}
 For a $Q$-manifold $\cM$ of degree $n$, the space of vector fields  $$(\chi(\cM),  -[Q, \cdot], [\cdot,\cdot])$$
is a DGLA with concentrated in degrees  $\ge -n$. (We put a minus sign in front of $[Q, \cdot]$ for mere convenience, and omitting it still delivers a DGLA). 
\end{remark}

\subsection{Lie $n$-algebroids}\label{subsec:split}

We review  Lie $n$-algebroids are defined in 
\cite[\S 2]{CCShengHigherExt} (where they are called \emph{split Lie $n$-algebroids}). For $n=1$ this definition recovers Lie algebroids.
\begin{defi} 
A \emph{Lie $n$-algebroid} consists of 
 a graded vector bundle $$A:=A_0\oplus A_{-1}\oplus\dots\oplus A_{-n+1}\to M$$ (where the fibers of $A_i$ are concentrated in degree $i$) over a manifold $M$ together with
 \begin{itemize}
\item   a bundle map $\rho\colon A_0\to TM$,
\item   an $L_{\infty}$-algebra structure on $\Gamma(A)$ (with multibrackets 
$[\cdots]_k$ for $k=1,\dots,n+1$)
\end{itemize}
 such that the following compatibily conditions are satisfied:
\begin{itemize}
\item[a)] $[\cdots]_k$ is $C^{\infty}(M)$-multilinear for $k\neq 2$
\item[b)] $[\cdots]_2$ applied to two sections of $A_{\neq 0}$ is $C^{\infty}(M)$-bilinear, and satisfies the Leibniz rule w.r.t. $\rho$: for all $s_0\in \Gamma(A_0)$, $t\in \Gamma(A)$, $f\in C^{\infty}(M)$, $$[s_0,ft]_2=f[s_0,t]_2+\rho(s_0)f\cdot t.$$ \end{itemize}
\end{defi}

Lie $n$-algebroids are also examples of $Q$-manifolds, by the 
 following known correspondence 
\cite[\S 2]{CCShengHigherExt}\cite[\S 3.4]{BonavolontaPoncin}.  Here, $A=A_0\oplus \dots\oplus A_{-n+1}\to M$ is a   graded vector bundle. Notice that $A[1]$ is concentrated in degrees $-1,\dots,-n$, hence it is an $N$-manifold with body $M$.
\begin{lemma}\label{nqnla}
There is a bijection between   Lie $n$-algebroid structures on $A$ and  homological vector fields on the N-manifold $A[1]$.
\end{lemma} 

We review briefly one direction of this correspondence. It makes use of Voronov's derived bracket construction, as recalled in Remark \ref{rem:vorcon}. 
Let $Q$ be a homological vector field on $A[1]$. Applying this construction to the vector fields $L=\vX(A[1])$, to the vertical and fiberwise constant vector fields $\mathfrak{a}=\Gamma(A[1])$, to the projection $P$ obtained restricting to the body $M$ and taking the vertical part, and to $\Delta=Q$, one obtains an $L_{\infty}[1]$-algebra structure on $\Gamma(A[1])$. Upon applying the d\'ecalage isomorphism \eqref{deca}, one obtains an $L_{\infty}$-algebra structure on $\Gamma(A)$. The formula $\rho(a)f=[[Q,a],f]$ for $a\in \Gamma(A_0[1])$ 
and $f\in C^{\infty}(M)$ determines the anchor. 

\begin{remark}
By Batchelor's theorem (see Rem. \ref{ref:batch}) and Lemma \ref{nqnla}, any $Q$-manifold of degree $n$ gives rise to a split Lie $n$-algebroid. 
\end{remark}
  
\bibliographystyle{habbrv} 
\bibliography{bibLoo}
\end{document}